\documentclass[a4paper, 11pt, twoside]{amsart}
\usepackage[top=3cm, bottom=3cm, left=2.5cm, right=2.5cm]{geometry}
\usepackage[pdftex]{graphicx}

\usepackage[utf8]{inputenc}
\usepackage[T1]{fontenc}
\usepackage[english]{babel}
\usepackage[noadjust]{cite}
\usepackage{amsthm,amssymb,epsfig,mathtools}
\usepackage{mathrsfs}  
\usepackage{bbm}        
\usepackage{esint}     
\usepackage{dsfont}
\usepackage{color}
\usepackage{enumerate}
\usepackage{url}
\usepackage{algorithm2e}
\usepackage{xcolor}

\usepackage{calc}
\usepackage[inline]{enumitem}
\usepackage[normalem]{ulem}
\usepackage{url}

\usepackage[bookmarks, bookmarksnumbered, pdfstartview={XYZ null null 1.00}]{hyperref}

\usepackage{tikz}
\usetikzlibrary{shapes, arrows, arrows.meta, decorations.markings}

\usepackage{todonotes}
\usepackage[ulem=normalem,draft]{changes}
\usepackage{cancel}

\makeindex

\numberwithin{equation}{section}

\newcommand{\theoname}{Theorem}
\newcommand{\lemmname}{Lemma}
\newcommand{\coroname}{Corollary}
\newcommand{\propname}{Proposition}
\newcommand{\definame}{Definition}

\newcommand{\remkname}{Remark}
\newcommand{\explname}{Example}

\newcommand{\R}{\mathbb{R}}

\theoremstyle{plain}
\newtheorem{theorem}{\theoname}[section]
\newtheorem{lemma}[theorem]{\lemmname}

\newtheorem{proposition}[theorem]{\propname}

\theoremstyle{definition}
\newtheorem{definition}[theorem]{\definame}
\newtheorem{remark}[theorem]{\remkname}
\newtheorem{example}[theorem]{\explname}

\newlist{hypothesis}{enumerate}{1}
\setlist[hypothesis]{label={\textup{(H\arabic*)}}, ref={(H\arabic*)}, leftmargin=*, widest*=10}

  {\list{}{\leftmargin=.5in\rightmargin=.5in}  \item[]  }%
  {\endlist}

\def\dd{{\rm d}}
\newcommand{\eqdef}{\ensuremath{\stackrel{\mbox{\upshape\tiny def.}}{=}}}

\newcommand{\norm}[1]{\left\lVert#1\right\rVert}
\newcommand{\inner}[1]{\left\langle#1\right\rangle}

\def\1B{{\bf  1}}

\def\diag{{\mathop{\rm diag}}}
\newcommand\diam{\mathop{\rm diam}}
\def\dist{{\rm dist}}

\def\dom{\mathop{{\rm dom}}}

\newcommand\Lip{\mathop{\rm Lip}}

\def\supp{\mathop{\rm supp}}

\def\id{{\mathop{\rm id}}}

\def\divv{{\mathop{\rm div}}}

\newcommand{\cvweak}[2]{\xrightharpoonup[#1]{#2}}
\newcommand{\cvstrong}[2]{\xrightarrow[#1]{#2}}

\def\inf{\mathop{\rm inf}}
\def\sup{\mathop{\rm sup}}

\def\min{\mathop{\rm min}}
\def\max{\mathop{\rm max}}

\def\argmin{\mathop{\rm argmin}}
\def\argmax{\mathop{\rm argmax}}

\newcommand{\mres}{\mathbin{\vrule height 1.6ex depth 0pt width
		0.13ex\vrule height 0.13ex depth 0pt width 1.3ex}}

\DeclareMathOperator{\Pac}{\mathscr{P}_{\mathrm{ac}}}
\def\varrhovarphi{\varrho_\varphi}
\DeclareMathOperator{\Var}{Var}
\newcommand\BL{\mathop{\rm BL}}
\def\Cvx{{\rm Cvx}}

\title[Quantitative Stability for BL and moment measures]{Quantitative Stability for the Brascamp-Lieb inequality and moment measures}

\author{João Miguel Machado}
\address{Lagrange Mathematical and Computational Center\\
103 rue de Grenelle\\
Paris, 75007, France}
\email{joao-miguel.machado@ceremade.dauphine.fr}

\author{João P. G. Ramos}
\address{Instituto Nacional de Matem\'atica Pura e Aplicada \\ Estrada Dona Castorina 110, Horto - Rio de Janeiro, RJ - Brazil}
\email{joao.ramos@impa.br}

\begin{document}

\begin{abstract} 

We develop a quantitative stability theory for moment measures based on a new sharp uniform stability principle for the Brascamp–Lieb variance inequality in terms of the $L^1$-distance. Our results yield structural stability estimates for solutions of the moment-measure problem that are uniform over a natural class of convex functions, thereby addressing several questions that have been open in this direction.

A key novelty of our approach is that the Brascamp–Lieb stability bound is not only sharp in its stability exponent, but also uniform across a broad class of convex potentials. This uniformity is absent from previous results in the literature and, beyond its intrinsic mathematical interest, it is the mechanism that allows stability of the Brascamp–Lieb inequality to transfer to nonlinear variational problems such as the moment-measure problem. We moreover show that the $L^1-$nature of this stability estimate is \emph{sharp}, in the sense that such a uniform estimate cannot hold in any $L^p$-metric, with $p >1$.

\bigskip

\noindent\textbf{Keywords.} Brascamp-Lieb inequality, moment measures, optimal transport, stability, log-concave measures

\medskip

\noindent\textbf{2020 Mathematics Subject Classification.} 
39B82, 26D10, 49Q22

\end{abstract}

\maketitle
\tableofcontents

\section{Introduction}\label{sec.introduction}

In recent years, the fields of calculus of variations and functional inequalities have become deeply interwined. If from one hand, many of the inequalities used in analysis and probability have deep variational interpretations, functional and geometric inequalities are fundamental in the study of many variational problems. To name a few exemples we can cite the use of Sobolev type inequalities in the stability of matter~\cite{lieb2010stability,lewin2022theorie}, the use of isoperimetric type inequalities in the study of shape optimization problems~\cite{lam2025stability}, and more recently the use of Brascamp-Lieb~\cite{brascamp1976extensions} and Prékopa-Leindler inequalities~\cite{prekopa1971logarithmic,prekopa_1973,leindler_1973} has become fundamental on the quantitative stability theory of optimal transport~\cite{delalande2023quantitative,letrouit2024gluing,letrouit2025lectures}. 

In this paper we address the quantitative stability of the moment-measure problem using sharp functional inequalities as our main tool -- as in, for instance,~\cite{brascamp1976extensions,cordero2015moment,santambrogio2016dealing}. Indeed, that connection between moment measures and the Brascamp--Lieb inequality has not only been made rigorous several times over (see for instance~\cite{klartag2014logarithmically} where a regularity theory for moment measures is combined with a suitable Brascamp-Lieb inequality to obtain bounds on a Poincaré constant~\cite{klartag2025isoperimetric}), but they also have a common denominator in \textit{optimal transport theory}; see~\cite{santambrogio2015optimal,villani2009optimal,ambrosio2021lectures}.

More specifically, the Brascamp-Lieb inequality is not only ubiquitous in the modern stability theory of optimal transport, but it can also be derived with linearization arguments of transportation type inequalities, as done for example in~\cite{cordero2017transport}. On the other hand, in~\cite{santambrogio2016dealing} it has been showcased how solve the moment measure problem by minimizing functionals coming from optimal transport. Indeed, we highlight, for instance, the classical work \cite{bobkov2000brunn} where a proof of Brascamp-Lieb's variance inequality with a linearization argument using Prékopa-Leindler is provided, as well as~\cite{cordero2015moment} (or Section~\ref{subsec.general_moment_measures} below), where a variational approach for moment measures is demonstrated, consisting on the maximization of a functional which is proven to be concave with the same functional inequality. 

By employing the newly developed theory of quantitative stability for the Prékopa-Leindler inequality in a series of recent works~\cite{boroczky2023quantitative,figalli2024improved,boroczky2025isoperimetric,figalli2025sharp}, we are able to obtain a \emph{sharp} stability version of the Brascamp-Lieb variance inequality, which is then used in order to produce novel explicit stability estimates for moment measures. The main features of our quantitative result is a stability constant that is completely independent of the convex potential and the fact that it holds regardless of strong convexity assumptions, allowing even for flat geometries.

\subsection{Stability of the Brascamp-Lieb variance inequality} In order to describe our main contributions, let us first recall the Brascamp-Lieb variance inequality \cite{brascamp1976extensions,bobkov2000brunn}: given a convex function $\varphi : \mathbb{R}^d \to \mathbb{R}$, let $\displaystyle \varrho \propto e^{- \varphi}$ be its associated Gibbs measure. Then for every sufficiently smooth function $f$, the following quantity 
\begin{equation}\label{eq.brascamp_lieb}
    \delta_{\BL}(f) 
    \eqdef 
    \int_{\mathbb{R}^d} 
    \inner{(D^2\varphi)^{-1}\nabla f, \nabla f}\dd \varrho
    - 
    \Var_{\varrho}(f) \ge 0
\end{equation}
is non-negative. It is well known that the optimal functions for such inequality, which make $\delta_{\BL}(f)$ vanish, are the affine functions on the geometry induced by the log-concave measure $\varrho$, \textit{i.e.} $f(x) = a\cdot\nabla \varphi(x) + b$; see for instance~\cite{cordero2017transport} or the derivation at the start of Section~\ref{sec.stability_Brascamp_Lieb}. Given the characterization of the optimizers of a functional inequality, the question of stability consists of: under which topology can we estimate the distance of $f$ to the manifold of optimizers? 

Results of this type have increasingly gained importance over the past two decades. Indeed, the question of stability has been extensively studied in the case of the isoperimetric inequality \cite{figalli_stability_2010}, as well as in other geometric inequalities such as the Faber-Krahn inequality \cite{brasco2015faber}, and, more recently, in~\cite{figalli_stability_2010,figalli2017quantitative, figalli2023sharp,figalli2024sharp,van2021sharp,van2023sharp} the Brunn-Minkowski inequality, in~\cite{dolbeault2016stability,dolbeault2025sharp,balogh2025sharp} for Sobolev and log-Sobolev inequalities and relations with non-linear evolution equations, in~\cite{boroczky2023quantitative, figalli2025sharp, figalli2024improved} for the Prékopa-Leindler and Borell-Brascamp-Lieb inequalities. 

In that vein, our first contribution is a \emph{sharp, uniform $L^1-$based stability estimate} for the Brascamp-Lieb variance inequality. In order to state it, we define the finite dimensional manifold of optimal functions 
\begin{equation}\label{eq.manifold_opt}
    \mathcal{O}_{\BL}
    \eqdef 
    \{
      f: \mathbb{R}^d \to \mathbb{R}: 
      f(x) = a\cdot\nabla\varphi(x) + b, \text{ for some }  
      a \in \mathbb{R}^d, \ b \in \mathbb{R}
    \}. 
\end{equation}
When $\varphi$ has a non-trivial domain, we set $\nabla \varphi(x) = 0$ for $x \in \mathbb{R}^d\setminus \dom \varphi$. As a result, we understand $\mathcal{O}_{\BL}$ as a subset of $L^1(\varrho)$.

Whenever $\varphi$ is sufficiently smooth, namely \textit{essentially continuous} see Definition~\ref{def.essentially_continuous} below, if $f \in \mathcal{O}_{\BL}$ then the parameter $b$ must be given by $\mathbb{E}_{\varrhovarphi}f$ since 
\[
    \mathbb{E}_{\varrhovarphi}f 
    = 
    \int_{\mathbb{R}^d} f \dd\varrhovarphi 
    = 
    a\cdot 
    \int_{\mathbb{R}^d} \nabla \varphi(x) e^{-\varphi}\dd x + b 
    = 
    a\cdot 
    \int_{\mathbb{R}^d} \nabla (e^{-\varphi})\dd x + b 
    = b.
\]
Our first main result is then an estimate of the distance of a function $f$ to the manifold of optimizers
\begin{equation}
    \text{dist}_{L^1(\varrhovarphi)}(f, \mathcal{O}_{\BL}) 
    \eqdef 
    \inf_{g \in \mathcal{O}_{\BL}} 
    \norm{f - g}_{L^1(\varrhovarphi)}
    = 
    \inf_{a \in \mathbb{R}^d} 
    \norm{f - (a\cdot\nabla\varphi + \mathbb{E}_{\varrhovarphi}f)}_{L^1(\varrhovarphi)},
\end{equation}
in terms of the deficit of the Brascamp-Lieb inequality. 

As described below, our result is valid regardless of the strong convexity of $\varphi$, this favorable geometry was the main ingredient in previous quantitative results for this inequality in the literature, see the discussion below for more details. Instead, we replace strong convexity with the geometry of the moment measure associated with $\varphi$ defined as 
\[
    \mu_\varphi \eqdef {(\nabla \varphi)}_\sharp \varrho_\varphi. 
\]
Our proof shows that as long as $\mu_\varphi$ is not concentrated in a hyperplane, we obtain information on the distance of a general function to the set $\mathcal{O}_{\BL}$. This non-concentration is equivalent to the non-negativity of the following functional 
\[
    \mu \mapsto 
    \Theta(\mu) \eqdef 
    \inf_{\theta \in \mathbb{S}^{d-1}} 
    \int_{\mathbb{R}^d} |\inner{\theta, x}|\dd \mu(x).
\]
This same functional will later play a crucial role in the stability of moment measures. 

In addition, the geometry of moment measures not only is capable of accounting for flat geometries, that is convex potential with vanishing Hessians, but also gives a stability constant that is universal for this large class of convex functions whose moment measure is not concentrated in a hyperplane. 

\begin{theorem}\label{thm.brascamp_lieb_quantitative_intro}
    There exists a universal constant $C_d$ depending only on the ambient dimension such that for any convex function $\varphi : \mathbb{R}^d \to \mathbb{R}\cup\{+\infty\}$ such that
    \[
        0 < \int_{\mathbb{R}^d} e^{-\varphi} < +\infty 
        \text{ and }
        \Theta(\mu_\varphi) > 0,
    \]
    we have that
    \[
        \dist_{L^1(\varrhovarphi)}(f, \mathcal{O}_{\BL})
        \le 
        C_d {\delta_{\BL}(f)}^{1/2}
    \]
    for all locally Lipschitz functions $f \in L^2(\varrhovarphi)$.   {Moreover, the inequality above is sharp, in the sense that the factor $\delta_{\BL}(f)^{1/2}$ \emph{cannot} be replaced by $\delta_{\BL}(f)^{\alpha}$ for any $\alpha > 1/2.$}
\end{theorem}

Regarding the scope of the class of convex functions that this result applies, we highlight that from the theory o moment measures discussed below, any essentially continuous and convex $\varphi$ will satisfy the condition $\Theta(\mu_\varphi) > 0$, but it also covers for log-concave densities supported on convex bodies presenting jumps at the boundary of their domain.

As the Brascamp-Lieb inequality is dimension-independent, the first efforts in the literature into improving it were focused on next order terms that would exploit the ambient dimension~\cite{harge2008reinforcement,bolley2018dimensional}. These results are equivalent with bounding the deficit from below with a dimensional term which vanishes in the manifold of optimizers for this inequality. A similar improvement was also obtained in~\cite{cordero2017transport}, by measuring the distance to the manifold of optimal functions with a $L^2$ term in a geometry that compensates for the curvature induced by the potential $\varphi$. This is very natural to expect for the Brascamp-Lieb inequality as it is closely related to the spectral gap of the diffusion operator induced by $\nabla \varphi$ that is central to the \textit{carré du champ method}~\cite{bakry1985diffusions,bakry2013analysis,arnaudon2018intertwinings}, see also Section~\ref{sec.stability_Brascamp_Lieb}. We refer the reader to~\cite{lam2025stability} for a particular version for the Gaussian Poincaré inequality with relations to uncertainty principles. 

The present work, on the other hand, is not only the first that proposes a quantitative stability result for the Brascamp-Lieb variance inequality in full generality with a constant \emph{independent} of $\varphi$. In particular, the independence of our stability results on the convex function turns out to be \emph{essential} in order to deduce several novel results about the stability of moment measures, as we shall see below. In addition, we crucially show show in Theorem~\ref{thm.constant_depends_pot_Lp} that this behavior is only true for the $L^1$-stability above, in the sense that any stability result for the Brascamp-Lieb inequality formulated with an $L^p$ distance to the set of optimizers must have a stability constant dependent of the convex potential. This shows that our $L^1$ result is therefore sharp with respect to the dependence of the stability constant on the convex potential. Our approach consists of constructing a one-parametric family of potentials showing that if the $L^1$ distance were to be replaced by any other $L^p$ distance for $1 < p < +\infty$, then the infimum 
\[
    \inf_{
            \substack{
                f \in L^2(\varrho_\varphi)\\ 
                \varphi \in \Cvx_d 
            }
        }
        \frac{{\delta_{\BL}(f, \varrho_\varphi)}^{1/2}}
        {\dist_{L^p(\varrho_\varphi)}(f;\mathcal{O}_{\BL}(\varphi))} = 0,
\]
where $\Cvx_d$ is a suitable space of convex functions, which is in fact much more restrictive than the class of potentials for which Theorem~\ref{thm.brascamp_lieb_quantitative_intro} holds. The proof of this identity consists of constructing a sequence of strongly convex potentials whose curvature (strong convexity parameter) vanishes asymptotically and approaches the value $0$ in the infimum above. This showcases how the this behavior of the stability constant is closely connected with the ability to deal with convex potentials having flat geometries.

Our method of proof starts with the approach of Bobkov and Ledoux \cite{bobkov2000brunn}. A main new tool in our proof is then the usage of the newly developed \emph{sharp stability estimates} for the Pr\'ekopa-Leindler inequality stemming from the work of A. Figalli, P. van Hintum and M. Tiba \cite{figalli2025sharp}. Nonetheless, the method of proof highlighted above only works a priori for a class of slightly more regular convex functions, which satisfy a lower Hessian bound. For such classes, we are able to exploit the fact that the moment measure associated with the potential $\varphi$ has full-dimensional support in order to deduce a non-degeneracy condition, which allows us to use a limiting argument to conclude in that case. 

On the other hand, in the general case where no such uniform result is available, we need an approximation result to finish the proof. Such a result employs, again, the non-degeneracy of the support of the moment measures at hand in a crucial way, which highlights how deeply our result is connected with the moment measure problem.

\subsection{Stabiliy for moment measures} We now turn to our main application: quantitative stability for the moment measure problem. Indeed, here we have a two-fold usage of stability estimates: using the Prékopa-Leindler inequality, we obtain stability of the Gibbs measures in the moment measures representation, while the quantitative version of Brascamp-Lieb yields stability of the dual potentials.

More concretely: given a convex function $\varphi$, we recall the notation of its associated Gibbs probability measure
\begin{equation}\label{eq.gibbs_measure}
  \varrho_\varphi \eqdef \frac{e^{-\varphi}}{\int_{\mathbb{R}^d}e^{-\varphi}}. 
\end{equation}
On the other hand, given $\psi$ also convex, which we will view informally as in the dual space of variables of $\varphi$ by frequently considering $\psi = \varphi^*$, we define the \emph{moment measure} associated with $\psi$ as 
\begin{equation}\label{eq.moment_measure_def}
  \mu_\psi
  \eqdef 
  {(\nabla \psi)}_\sharp \varrho_\psi. 
\end{equation} 

The moment measure problem consists of the inverse question: under what conditions on a probability measure $\mu \in \mathscr{P}_1(\mathbb{R^d})$, the space of Radon probability distributions with finite first moments, is there a convex function $\psi$ such that $\mu_\psi = \mu$? In~\cite{cordero2015moment,santambrogio2016dealing} it was proven that $\mu$ is centered at the origin and is not concentrated in a hyperplane of $\mathbb{R}^d$ if, and only if, there exists a essentially continuous (see Definition~\ref{def.essentially_continuous}) convex potential $\psi_\mu$ realizing this representation for $\mu$. In this case, $\psi_\mu$ is unique up to translations. Using the Prékopa-Leindler inequality, it was shown in~\cite{cordero2015moment} that the moment measure of $\psi_\mu$ is $\mu$ if and only if its Legendre transform $\varphi_\mu \eqdef \psi_\mu^*$ is a maximizer of 
\begin{equation}\label{eq.functional_intro} 
    \mathscr{J}_\mu(\varphi) 
    \eqdef \log \int_{\mathbb{R}^d} e^{-\varphi^*}\dd x - 
    \int_{\mathbb{R}^d} \varphi \dd \mu. 
\end{equation}

The same result was recovered in~\cite{santambrogio2016dealing} using optimal transport, but this time for a dual perspective: a measure $\varrho$ is the Gibbs measure in~\eqref{eq.moment_measure_def} if and only if it minimizes
\begin{equation}\label{eq.functional_sant_intro} 
    \mathscr{E}_\mu(\varrho) 
    \eqdef 
    \int_{\mathbb{R}^d}\log\varrho(x) \dd\varrho(x) 
    + 
    \mathcal{T}(\varrho,\mu),
\end{equation}
where $\mathcal{T}$ is the maximal correlation formulation of the quadratic optimal transport problem. The reader is referred to Sections~\ref{subsec.optimal_transport} and~\ref{subsec.general_moment_measures} for more details on optimal transport theory and for general results on moment measures. 
  
The quantitative stability question that we wish to answer then becomes: given two measures $\mu,\nu$,  quantify the distance between either $\varphi_\mu, \varphi_\nu$ or $\varrho_\mu,\varrho_\nu$. Since the variational characterization described above is necessary and sufficient, one can only expect to answer this question as an estimate on the distance to the optimal sets defined as
\begin{equation}\label{eq.set_optimalEJ}
    \mathcal{M}_\mu \eqdef \argmin \mathscr{E}_\mu, \quad 
    \mathcal{N}_\mu \eqdef \argmax \mathscr{J}_\mu,
\end{equation}
which can be characterized as follows: fixed some reference $\varrho_0 \in \mathcal{M}_\mu$ and $\varphi_0 \in \mathcal{N}_\mu$, the manifolds of optimizers can be expressed as
\begin{equation}
    \begin{aligned}
        \mathcal{M}_\mu 
        &= 
        \{
            \varrho_{0}(\cdot - x_0): 
            x_0 \in \mathbb{R}^d
        \}, \\
        \mathcal{N}_\mu 
        &= 
        \{
            \varphi_0 + (a\cdot x + b): 
            a \in \mathbb{R}^d, \ 
            b \in \mathbb{R}
        \}.
    \end{aligned}
\end{equation}
On one hand, the structure of $\mathcal{M}_\mu$ expresses the fact that moment measures are invariant w.r.t.~translations of their associated convex potential $\psi_\mu$. The structure of $\mathcal{N}_\mu$ is explained by the fact that adding an affine function to a convex potential translates its Legendre transform, that is $(\varphi_\mu + a\cdot x)^* = \psi_\mu(\cdot - a)$; the addition of the constant $b$ are then canceled out by renormalizing the respective Gibbs measures.

We start with the following Theorem, which is a backbone to the stability results that will follow.

\begin{theorem}\label{lemma.stab_momentmeasures_backbone_intro}
     There exists a universal constant $C_d$ depending only on the ambient dimension such that the following holds.  Given a Radon measure $\mu \in \mathscr{P}_1(\mathbb{R}^d)$ whose barycenter lies that the origin and $\dim \supp \mu = d$, let $\bar \varphi$ be a maximizer of $\mathscr{J}_\mu$. Then for any convex function $\varphi : \mathbb{R}^d \to \mathbb{R}\cup \{+\infty\}$ it holds that
    \begin{equation}\label{eq.backbone_stability_Gibbs_intro}
        \dist_{L^1(\mathbb{R}^d)}\left(
            \varrho_{\varphi^*},
            \mathcal{M}_\mu
        \right)
        \le C_d 
        {\left(
            \mathscr{J}_{\mu}(\bar \varphi)
            - 
            \mathscr{J}_{\mu}(\varphi)
        \right)}^{1/2}.
    \end{equation}
    
    In addition, assuming that $\mu$ has finite second moment, there exists $\lambda \in [1/4,1/2]$ such that, setting $v \eqdef \bar \varphi - \varphi$ and $\varphi_\lambda \eqdef \bar \varphi + \lambda v$, it holds that
    \begin{equation}\label{eq.backbone_stability_potentials_intro}
       \dist_{L^1(\mu_{\lambda})}\left(
            \varphi,
            \mathcal{N}_\mu
        \right)
        \le C_d \left(
            \mathscr{J}_{\mu}(\bar \varphi)
            - 
            \mathscr{J}_{\mu}(\varphi)
        \right)^{1/2},
    \end{equation}
    where $\mu_{\lambda} = \mu_{\varphi_\lambda^*}$ is the moment measure associated with the interpolation $(\varphi_\lambda)^*$. 
\end{theorem}

This result may be regarded as a Polyak-Łojasiewicz type inequality for the functional $\mathscr{J}_\mu$, as estimate~\eqref{eq.backbone_stability_potentials_intro} gives a way of using this functional to measure the distance of any $\varphi$ to the class of maximizers of $\mathscr{J}_\mu$. Estimate~\eqref{eq.backbone_stability_Gibbs_intro} relates the invariance of the moment measure representation with respect to~translations to the equality cases in the Prékopa-Leindler inequality, estimating the distance of $\varrho_{\varphi^*}$ to the set of minimizers of $\mathscr{E}_\mu$. 
On the other hand,~\eqref{eq.backbone_stability_potentials_intro} is proved by establishing a connection between the equality case of the Brascamp-Lieb inequality and the invariance of the functional $\mathscr{J}_\mu$ with respect to~the addition of affine functions, being therefore related to the stability of the Brascamp-Lieb inequality. In Example~\ref{ex.sharp_exponentPL}, we also discuss how the exponent $1/2$ is sharp. 

In fact, the second variation of the functional $\mathscr{J}_\mu$ is shown to be associated with the deficit of the Brascamp-Lieb inequality in Lemma~\ref{lemma.first_second_variation}. There is a very delicate literature around even the first variation of this functional, see the discussion from~\cite{rotem2022surface,rotem2023anisotropic}and the references therein. For our purposes we only require variations along convex interpolations, but even in this case rigorous on the second variation are unknown to the authors. We manage to establish a second variation formula when one of the end points is strongly convex, which is enough to prove~\eqref{eq.backbone_stability_potentials_intro} via an approximation argument that is also used in the proof of the quantitative stability of the Brascamp-Lieb inequality itself, further strengthening the connection between these two problems.

In spite of the fact that Theorem \ref{lemma.stab_momentmeasures_backbone_intro} provides us, by itself, with a novel quantitative way to show that almost maximizers of the functionals $\mathscr{J}_\mu$ and $\mathscr{E}_\mu$ must be close to the original potentials, we then exploit these properties to study the quantitative stability for the moment measure problem. 

\subsection{Applications} We illustrate the scope of our stability theory in three settings: compact domains, quadratic regularization, and the whole space, highlighting how Theorems \ref{thm.brascamp_lieb_quantitative_intro} and \ref{lemma.stab_momentmeasures_backbone_intro} can be utilized in order to deduce many further novel results about the stability of moment measures. These results are, in a sense, similar in spirit to the results of \cite{delalande2023quantitative} with respect to the stability of Brenier maps.

\subsubsection*{Stability in a compact domain}
Using Theorem~\ref{lemma.stab_momentmeasures_backbone_intro}, we are able to obtain several quantitative stability results for moment measures in different contexts. When the moment measure and the associated Gibbs measure in~\eqref{eq.moment_measure_def} are restricted to a compact domain we obtain a strong stability result controlling the $L^1$ distances between $\varrho_\mu, \varrho_\nu$ and $\varphi_\mu, \varphi_\nu$ with the $1$-Wasserstein distance of $\mu,\nu$; this is the content of the next result.

\begin{theorem}\label{thm.stability_compact_case_intro}
    Let $\Omega$ be a convex and compact subset of $\mathbb{R}^d$. Given two Radon measures $\mu, \nu \in \mathscr{P}(\Omega)$ whose barycenters lie at the origin and $\dim \supp \mu = \dim \supp \nu = d$, let $\varphi_\mu, \varphi_\nu$ be maximizers of $\mathscr{J}_{\mu,\Omega}, \mathscr{J}_{\nu,\Omega}$ respectively. Then, there exists a universal constant $C_{d,\Omega}$ depending only on the dimension and on the diameter of $\Omega$ such that
    \begin{equation}\label{eq.stability_Gibbs_compact_intro}
        \inf_{x_0 \in \mathbb{R}^d} 
        \norm{
            \varrho_{\varphi_\mu^*}(\cdot) - \varrho_{\varphi_\nu^*}(\cdot + x_0) 
        }_{L^1(\mathbb{R}^d)}
        \le C_{d,\Omega} {W_1(\mu,\nu)}^{1/2}
    \end{equation}
    and
    \begin{equation}\label{eq.stability_potentials_compact_intro}
        \inf_{a \in \mathbb{R}^d, \ b \in \mathbb{R}} 
        \norm{
            (\varphi_\mu - \varphi_\nu) - (a\cdot x + b) 
        }_{L^1(\mu_{\lambda})}
        \le C_{d,\Omega} {W_1(\mu,\nu)}^{1/2},
    \end{equation}
    where $\mu_\lambda = \mu_{\psi_\lambda}$ is the moment measure of $\psi_\lambda = ((1-\lambda)\varphi_\mu + \lambda\varphi_\nu)^*$ for some $\lambda \in (0,1)$.
\end{theorem}

In spite of the fact that compactness allows for a fairly simpler mathematical structure, diverging from the objects described by the existence theorem by~\cite{cordero2015moment,santambrogio2016dealing}, we remark that such a result means a remarkable control \emph{especially for numerical applications}, as we are able to control a strong norm with a distance which metrizes a fairly weak convergence of Radon measures; see Section~\ref{subsec.optimal_transport} for more details on these topologies. 

\subsubsection*{Rates of convergence for quadratic regularization}
The classical definition of moment measures from~\eqref{eq.moment_measure_def} has strong motivations from convex and toric geometry~\cite{berman2013real}, but variations thereof have recently found several applications to sampling and generative models~\cite{vesseron2025sample}. For such applications, the essential property of moment measures is their capacity of representing a given measure via the pushforward of a log-concave distribution through the gradient of the same convex potential. On the other hand, for such applications in sampling it is particularly interesting for the log-concave measure to have a strongly convex potential. 

Although this cannot be expected in general for the moment measure representation, this property can be enforced with the \textit{quadratically regularized moment measures} recently proposed in~\cite{delalande2025regularized}. That is, given $\mu \in \mathscr{P}_1(\mathbb{R}^d)$ there is a convex potential $\psi_\alpha$ such that 
\begin{equation}\label{eq.reg_moms_intro}
    \mu = {(\nabla \psi_\alpha)}_\sharp\varrho_\alpha, \text{ where }
    \varrho_\alpha \propto e^{-(\psi_\alpha + \frac{\alpha}{2}|\cdot|^2)};
\end{equation} 
see also Section~\ref{subsec.general_moment_measures}. 

By introducing the quadratic regularization, the quantitative stability investigated in~\cite{delalande2025regularized} turns out to become significantly simpler for a few reasons: first, we mention that the uniqueness modulo translations is eliminated to give way to straight-out uniqueness. This way, the quantitative stability can be understood without the need for the introduction of a distance to a certain manifold of optimizers. Secondly, adding quadratic regularization to the functional from~\eqref{eq.functional_sant_intro} makes it strongly geodesically convex in Wasserstein space. On the other hand, the stability result proven in~\cite{delalande2025regularized} is formulated with the $2$-Wasserstein distance. 

For these reasons, both of theoretical and of applied nature, obtaining explicit rates of convergence with respect to the regularization parameter is of great interest. This is achieved once again by using Theorem~\ref{lemma.stab_momentmeasures_backbone_intro}. 
\begin{theorem}\label{thm.conv_regularization_intro}
    Fix a Radon  $\mu \in \mathscr{P}_1(\mathbb{R}^d)$ whose barycenter lies at the origin and such that $\dim\supp\mu = d$. Then there exists a constant $C_{d,\mu}$ depending on the ambient dimension and on $\mu$ such that 
    \begin{equation}
        \dist_{L^1(\mathbb{R}^d)}\left(
            \varrho_\alpha,
            \mathcal{M}_{\mu}
        \right)
        \le C_{d,\mu} \alpha^{1/2}, 
    \end{equation}
    for $\varrho_\alpha$ given in~\eqref{eq.reg_moms_intro}. 

    In addition, set $\varphi_\alpha \eqdef \psi_\alpha^*$, where $\psi_\alpha$ is given in~\eqref{eq.reg_moms_intro}. Then letting $\varphi_\mu$ be the dual potential associated with the classical moment measure representation, for each $\alpha$ there exists $\lambda \in (0,1/2)$ such that, setting $v\eqdef \varphi_\alpha - \varphi_\mu$ and $\varphi_t \eqdef \varphi_\mu + \lambda v$, it holds that
    \begin{equation}
        \dist_{L^1(\mu_{\lambda})}\left(
            \varphi_\alpha,
            \mathcal{N}_{\mu}
        \right)
        \le C_{d,\mu} \alpha^{1/2},
    \end{equation}
    where $\mu_{\lambda} = \mu_{\varphi_\lambda^*}$ is the moment measure associated with the interpolation $(\varphi_\lambda)^*$. 
\end{theorem}

\subsubsection*{Stability in $\mathbb{R}^d$}
To tackle the stability question in $\mathbb{R}^d$ we can in principle combine the quantitative stability of regularized moment measures arguments based on strong geodesic convexity from~\cite{delalande2025regularized} with our explicit rates of convergence~\ref{thm.conv_regularization_intro}. But, in order to achieve this, we have to make quantitative the assumption that $\mu$ is not supported on a hyperplane, otherwise such a result would allow for the construction of a $\psi$ yielding~\eqref{eq.moment_measure_def} for such a singular $\mu$, contradicting the existence theory of~\cite{cordero2015moment,santambrogio2016dealing}. See Proposition~\ref{prop.constant_depends_vartheta} for more details on this construction. 

Our approach is then to either impose a \emph{uniform lower bound} on the geometric quantity that was already used in the stability of Brascamp-Lieb inequality
\[
    \Theta(\mu) \eqdef 
    \inf_{\theta \in \mathbb{S}^{d-1}} 
    \int_{\mathbb{R}^d} |\theta\cdot y|\dd \mu(y) \ge \vartheta,
\]
or to control the Hessian of the log-density of $\mu$. This is done by defining the two classes 
\begin{align*}
    \mathcal{K}_{\vartheta} 
    &\eqdef 
      \left\{
        \mu \in \mathscr{P}(\Omega) : 
        \begin{array}{c}
            \mu \text{ is centered, } \\ 
            \displaystyle
            \Theta(\mu) \ge \vartheta
        \end{array} 
      \right\} \text{ for a compact $\Omega$,}\\
    \mathcal{K}_\Lambda 
    &\eqdef 
    \left\{
        \mu \in \Pac(\mathbb{R}^d) 
        : 
        \begin{array}{c}
            \mu \propto e^{-V} \text{ is centered, } \\ 
            \displaystyle
            D^2V \le \Lambda \id
        \end{array}
    \right\}.
\end{align*}
The results we obtain can be summarized as follows, where the reader is referred to Section~\ref{subsec.stabilityRd} for more information.

\begin{theorem}\label{thm.stability_moment_measures_Rd_intro}
    Take $\mathcal{K}$ to be either $\mathcal{K}_\vartheta$ or $\mathcal{K}_\Lambda$. Then for each $p>2$ there exists a constant $C$ such that, for all $\mu, \nu \in \mathcal{K}$ it holds thats
    \begin{equation}\label{eq.stability_moment_measures_Rd_intro}
        \dist_{W_2}
        \left(
            \mathcal{M}_\mu,
            \mathcal{M}_\nu
        \right)
        \le 
        C {W_2(\mu, \nu)}^{\frac{p-2}{6p-4}}, 
    \end{equation}
    where $C$ grows linearly in $p$.  
\end{theorem}

In both cases, our approach is to control the moments of order $p$ of the Gibbs measures $\varrho_\mu$, uniformly for all elements $\mu$ of the classes  $\mathcal{K}_\vartheta$ and $\mathcal{K}_\Lambda$. With this finer information we can come back to the strong geodesic convexity arguments for regularized moment measures from~\cite{delalande2025regularized} and obtain an optimal dependence on $\alpha$ in the quantitative stability results in this case. For the first class $\mathcal{K}_\vartheta$, this can be done with a Lemma~\ref{lemma.klartag_growth_convex} due to Klartag. For the class $\mathcal{K}_\Lambda$, similar controls on the $p$-moments can be obtained with a regularity result of moment measures, namely
\begin{equation}\label{eq.strong_convexity_moms}
    \mu \propto e^{-V} = {\left(\nabla\psi\right)}_\sharp e^{-\psi}, \text{ such that }
    D^2V \le \Lambda \id, 
    \text{ then $\psi$ is $\Lambda^{-1}$-strongly convex.}
\end{equation}

The proof of~\eqref{eq.strong_convexity_moms} is done in Theorem~\ref{thm.regularity_KLambda}, and is based on a bootstrap argument using Caffarelli's contraction theorem~\cite{caffarelli1992regularity,caffarelli2000monotonicity}. It states that if $\mu$ is as in~\eqref{eq.strong_convexity_moms} and $\nu \propto e^{-W}$ with $D^2 W \ge \lambda \id$, then the optimal transportation map $T$ from $\mu$ to $\nu$ is globally Lipschitz with $\Lip T \le \sqrt{\Lambda/\lambda}$. Since we cannot know in principle that $\psi$ is strongly convex, what we do instead is consider the quadratic regularized moment measure presentation of $\mu$. In this case, the Gibbs measure has indeed a strongly convex potential, with a modulus of convexity $\alpha$. This modulus of convexity can be iteratively increased with successive applications of the contraction theorem, which yields $\Lambda^{-1}$ at the end. Besides the application to the stability of moment measures, we believe that this result might have important applications in sampling; see the discussion after the proof of Theorem~\ref{thm.regularity_KLambda}.

\subsection*{LLM usage}

Some words about the usage of Large Language Models are in order. Indeed, the first version of this manuscript was elaborated essentially completely without the assistance of LLMs, except for the correction of a few typos and misprints. In this updated version, however, LLMs played a significant role: first and foremost, the example yielding the sharpness of the $L^1$ norm in Theorem \ref{thm.brascamp_lieb_quantitative_intro} was achieved by an interaction of the second named author with GPT 5.4. The proof of such a counterexample was later rewritten and formalized into Lean in the $L^2$ case (see \href{https://github.com/joaopgramos95/CounterExamplePotentialBrascampLieb.git}{https://github.com/joaopgramos95/CounterExamplePotentialBrascampLieb.git}) with the aid of Claude Opus 4.7. Other than that, the final wording and insertion of such arguments into the text was heavily modified and edited by the authors. We take full responsbility for this paper's contents.

\section*{Acknowledgments}
During the final preparation steps of the present work we have discovered that other colleagues were simultaneously working on similar questions on the stability of moment measures~\cite{bonnet2026semidiscrete}, obtaining results similar to ours in terms of stability in the specific case of the moment measure problem, but with a particular focus on semi-discrete approximations and applications in Kähler—Ricci solitons. We wish to thank them, specially Guillaume Bonnet, for clarifying discussions and sharing their beautiful preliminary results with us. 

The first author wishes to thank Quentin Mérigot and Filippo Santambrogio for stimulating discussions about moment measures. He also warmly thanks the support of the Lagrange Mathematical and Computational Research Center. Finally, both authors would like to thank Alessio Figalli and Károly Böröczky for insightful comments and suggestions on preliminary versions of this work. 

\section{Some tools on optimal transport and functional inequalities}\label{sec.tools}  

\subsection{A primer on optimal transport and Wasserstein distances}\label{subsec.optimal_transport}

The Wasserstein distances are defined via the value function of the optimal transport problem as 
\begin{equation}\label{eq.Wasserstein_distances}
    W_p^p(\mu,\nu) 
    \eqdef 
    \min_{\gamma \in \Pi(\mu,\nu)} 
    \int_{\mathbb{R}^d\times\mathbb{R}^d} 
    |x-y|^p\dd \gamma(x,y), 
\end{equation}
where $\Pi(\mu,\nu)$ the set of probability measures over $\mathbb{R}^d\times\mathbb{R}^d$ whose marginals are respectively $\mu$ and $\nu$, the so called set o transportation plans. This quantity is finite if and only if $\mu$ and $\nu$ have finite $p$-moments, $M_p(\mu),M_p(\nu) < +\infty$, where 
\begin{equation}\label{eq.p_moment}
    M_p(\mu) \eqdef \int_{\mathbb{R}^d} |x|^p\dd \mu. 
\end{equation}
The reader is referred to~\cite{santambrogio2015optimal,villani2009optimal} for a complete introduction to optimal transport and Wasserstein distances, in what follows we introduce its properties that will be useful in the sequel.

We are particularly interested in the case $p = 2$, where the optimal transport problem~\ref{eq.Wasserstein_distances} has a close relation with convex analysis. This can be easily seen by introducing the \textit{maximal correlation formulation}
\begin{equation}\label{eq.max_correlation}
  \mathcal{T}(\mu, \nu) 
  = 
  \sup_{\gamma \in \Pi(\mu, \nu)} 
  \int_{\mathbb{R}^d\times\mathbb{R}^d}\inner{x,y}\dd \gamma
  = 
  \inf_{\varphi \text{ convex }} 
  \int_{\mathbb{R}^d} \varphi\dd \mu
  + 
  \int_{\mathbb{R}^d} \varphi^* \dd \nu,
\end{equation} 
in such a way that 
\[
    \frac{1}{2}W_2^2(\mu,\nu) 
    = 
    \frac{1}{2}M_2(\mu) + \frac{1}{2}M_2(\nu) 
    - 
    \mathcal{T}(\mu, \nu). 
\]
The infimum on the RHS is attained by a convex function $\varphi$, and Brenier's Theorem states that the optimal transportation plan is unique and given by $\gamma = (\mathrm{id}, \nabla\varphi)_\sharp \mu$, whenever $\mu$ is absolutely continuous w.r.t.~the Lebesgue measure. 

\subsubsection*{Topological properties of Wasserstein distances}
A very important property of Wasserstein distances is that they metrize the narrow topology of probability measures. A sequence of Radon probability measures ${\left(\mu_n\right)}_{n \in \mathbb{N}}$ converges narrowly to $\mu$ if for all $f \in \mathscr{C}_b(\mathbb{R}^d)$, continuous and bounded function, it holds that
\[
    \int_{\mathbb{R}^d}f\dd\mu_n 
    \cvstrong{n \to \infty}{}
    \int_{\mathbb{R}^d}f\dd\mu,
\]
and we write $\mu_n \cvweak{n \to \infty}{}\mu$. A very important property is that 
\begin{equation}
    W_p(\mu_n,\mu) \cvstrong{n \to \infty}{}0 
    \text{ if and only if } 
    \begin{cases}
        \mu_n \cvweak{n \to \infty}{}\mu,& \\ 
        M_p(\mu_n) \cvstrong{n \to \infty}{} M_p(\mu).  
    \end{cases}
\end{equation}
In addition, the space of Radon probability measures with finite $p$-moments $\mathscr{P}_p(\mathbb{R}^d)$ becomes itself complete and separable when endowed with the topology of $W_p$.

This is conceptually very important for the contributions of this work since many of the quantitative stability results regarding moment measure representations from Section~\ref{sec.stability_moment_measures} are stated with respect to the Wasserstein distance $W_2$. This allows to compare very singular objects with a weak topology. On the other hand, if we know that two probability measures are also in $L^1(\mathbb{R}^d)$ and have finite $p$-moments, then we can interpolate the Wasserstein distance with the $L^1$ distance via the following inequality. 
\begin{lemma}\label{lemma.bound_wasserstein_momentsL1}
    Let $\mu,\nu$ be two probability densities in $L^1(\mathbb{R}^d)$ with finite $p$-moments, for some $p > 1$. Then, for all $1 \le q < p$ there exists a constant $C_{p,q}$ such that
    \begin{equation}
        W_q(\mu,\nu) 
        \le 
        C_{p,q} {(M_p(\mu) + M_p(\nu))}^{1/p} 
        \norm{\mu - \nu}_{L^1}^{1/q - 1/p}. 
    \end{equation}
\end{lemma}
\begin{proof}
    Suppose that $\mu,\nu$ have densities $f,g$, set $\varepsilon \eqdef \norm{f - g}_{L^1(\mathbb{R}^d)}$, and define the common density and the residuals $f_r,g_r$ as 
    \[
        h(x) 
        \eqdef 
        \min\{f(x), g(x)\}
        \ 
        f_r\eqdef f - h,
        \ 
        g_r \eqdef g - h.
    \]
    As a result, we observe that 
    \[
        \int_{\mathbb{R}^d} f_r(x) \dd x 
        +
        \int_{\mathbb{R}^d} g_r(x) \dd x 
        = 
        \varepsilon, 
        \quad 
        \int_{\mathbb{R}^d} f_r(x) \dd x 
        -
        \int_{\mathbb{R}^d} g_r(x) \dd x 
        = 0. 
    \]
    so that $\int_{\mathbb{R}^d} f_r(x) \dd x = \int_{\mathbb{R}^d} g_r(x) \dd x = \varepsilon/2$. 

    Now, we can construct a transportation plan $\gamma \in \Pi(\mu,\nu)$ as follows: we first transport the common part of both densities via the identity map, and to transport the residuals we use the product measure. More precisely, we set $\alpha \eqdef \varepsilon/2$ and 
    \[
        \gamma 
        \eqdef 
        {(\mathrm{id}, \mathrm{id})}_\sharp h 
        + 
        \frac{1}{\alpha} f_r(x)g_r(y) \dd x \dd y.    
    \]
    We then have 
    \begin{align*}
        W_q^q(\mu,\nu) 
        \le 
        \int_{\mathbb{R}^d\times\mathbb{R}^d} |x-y|^q \dd \gamma(x,y) 
        = 
        \frac{1}{\alpha}
        \underbrace{
            \int_{\mathbb{R}^d\times\mathbb{R}^d} |x-y|^q f_r(x)g_r(y) \dd x \dd y
        }_{\eqdef I}.
    \end{align*}
    To estimate the RHS, we split the integral into two parts as 
    \begin{align*}
        I
        &= I_1 + I_2,\\  
        &= \int_{B_R\times B_R} |x-y|^q f_r(x)g_r(y) \dd x \dd y
        + 
        \int_{(B_R\times B_R)^c} |x-y|^q f_r(x)g_r(y) \dd x \dd y.
    \end{align*}
    The first term can be easily estimated as 
    \[
        I_1 
        \le
        (2R)^q \alpha^2, 
    \]
    while the second can be estimated as with the $p$-moments, for any $p>q$ by noticing that for $|x|>R$ we have $|x|^q \le R^{q-p}|x|^p$, so that
    \begin{align*}
        I_2 
        &\le 
        2^{q-1} \alpha
        \int_{|x| \ge R} |x|^q f_r(x) \dd x 
        +
        2^{q-1} \alpha
        \int_{|y| \ge R} |y|^q g_r(y) \dd y \\ 
        &\le 
        2^{q-1} R^{q-p} \alpha
        \left( M_p(\mu) + M_p(\nu)\right). 
    \end{align*}
    Combining both estimates, we obtain
    \[
        W_q^q(\mu,\nu) 
        \le 
        (2R)^q \alpha +  2^{q-1} R^{q-p}
        \left( M_p(\mu) + M_p(\nu)\right). 
    \]
    Optimizing in $R > 0$ gives the desired result.
\end{proof}

\subsubsection*{Geodesic convexity in Wasserstein spaces}
Given two probability measures $\mu_0, \mu_1 \in \mathscr{P}_p(\mathbb{R}^d)$, any optimal transportation plan $\gamma$ yields a natural interpolation between them given by
\[
    \mu_t \eqdef 
    {\pi_t}_\sharp \gamma, \text{ where } \pi_t(x,y) = (1-t)x + ty, \quad t \in [0,1].
\] 
It turns out that this interpolation is a constant speed geodesic in the metric space $(\mathscr{P}_p(\mathbb{R}^d), W_p)$, see~\cite{ambrosio2021lectures,santambrogio2015optimal,villani2009optimal} for details.

This allows us to define the notion of geodesic convexity for functionals defined over Wasserstein spaces. A functional $\mathscr{F} : \mathscr{P}_p(\mathbb{R}^d) \to \mathbb{R}\cup\{+\infty\}$ is said to be geodesically convex if for every $\mu_0, \mu_1 \in \mathscr{P}_p(\mathbb{R}^d)$ and any geodesic interpolation $\mu_t$ between them the function $t \mapsto \mathscr{F}(\mu_t)$ is convex in the classical sense. It is strictly geodesically convex if the same function is strictly convex. 

The $2$-Wasserstein distance is not itself geodesically convex - this is a feature of the geometry of the Wasserstein space, see~\cite[Chapter 9]{Ambrosio2008GigliSavare}. In fact, it satisfies the opposite inequality. In~\cite[Propositon~3.3]{santambrogio2016dealing}, Santambrogio exploited this fact to show the geodesic convexity of the maximal correlation functional $\varrho \mapsto \mathcal{T}(\varrho,\mu)$. This fact is at the heart of the study of moment measures via optimal transport, as well as the geodesic convexity of the entropy functional defined as 
\begin{equation}\label{eq.entropy}
  H(\varrho) 
  \eqdef 
  \begin{cases}
    \displaystyle
    \int_{\mathbb{R}^d}\log\varrho \dd\varrho,& \text{ if } \varrho \ll \mathscr{L}^d,\\ 
    +\infty,& \text{ otherwise,}
  \end{cases}
\end{equation}
which follows from McCann's criterion~\cite{mccann1997convexity}\footnote{
     And seems to be equivalent ot the Prékopa-Leindler inequality~\cite{prekopa1971logarithmic,leindler_1973}.
}.

Another important example of geodesically convex functionals are the $p$-moments, $\mu \mapsto M_p(\mu)$, this follows directly from the convexity of $x \mapsto |x|^p$, see~\cite[Chapter~7]{santambrogio2015optimal}. In fact, it is easy to see that the second moment $M_2$ is even \textit{strongly} geodesically convex, that is, for any geodesic interpolation $\mu_t$ between $\mu_0$ and $\mu_1$ it holds that
\begin{equation}\label{eq.strong_convexity_second_moment}
    M_2(\mu_t) 
    \le 
    (1-t)M_2(\mu_0) + t M_2(\mu_1) - \frac{1}{2}t(1-t)W_2^2(\mu_0,\mu_1).
\end{equation}
This property is particularly useful to derive quantitative stability results for minimizers, as done in~\cite{delalande2025regularized} and will be exploited here.

\subsection{General facts on moment measures}\label{subsec.general_moment_measures}

In~\cite{cordero2015moment}, Cordero-Eurasquin and Klartag gave a complete characterization of the measures $\mu$ which admit a meaningful moment measure representation. Their analysis goes through the functionals defined as in \eqref{eq.functional_intro}

They identified that if the potential $\psi$ in~\eqref{eq.moment_measure_def} does not satisfy certain regularity properties, its moment measure can loose a lot of information. They shed some light on this issue with the following example: take $y \in \mathbb{R}^d$, a convex body $C\subset \mathbb{R}^d$ and consider 
\[
  \psi(x) \eqdef 
  \begin{cases}
    x\cdot y,& x \in C\\ 
    +\infty,& \text{ otherwise}. 
  \end{cases}
\]
Then one can readily check that $\mu_\psi = \delta_y$, which does not give much information on the measure $\delta_y$, and certainly is not amenable to a quantitative stability result since one can drastically change the convex set $C$ and obtain the same moment measure. For this reason they focus on characterizing measures that admit a moment measure representation with an \textit{essentially continuous} and convex function. 

\begin{definition}\label{def.essentially_continuous}
    A convex function $\psi : \mathbb{R}^d \to \mathbb{R}\cup\{+\infty\}$ is said to be essentially continuous if it is lower semi-continuous and the set of points where it is discontinuous has zero $\mathscr{H}^{d-1}$ measure. 
\end{definition}

\begin{lemma}[\cite{cordero2015moment}]\label{lem.moms_essentially_continuous}
    Let $\psi$ be a convex and essentially continuous function such that $0 < \int_{\mathbb{R}^d} e^{-\psi} < +\infty$, and such that $0 < \int_{\mathbb{R}^d} e^{-\psi} < +\infty$. Then, its moment measure $\mu_\psi$ is centered and its support spans $\mathbb{R}^d$. In particular, 
    \[
        \inf_{\theta \in \mathbb{S}^{d-1}} 
        \int_{\mathbb{R}^d} |x\cdot \theta| \dd \mu_\psi(x) > 0.
    \]
\end{lemma}

  {Since the latter condition will be of particular importance for us, we highlight its underlying object: we define the \emph{direction absolute first moment} of $\mu_\psi$ in the direction $\theta \in \mathbb{S}^{d-1}$ as 
\[
\Theta(\mu_\psi) \eqdef 
\inf_{\theta \in \mathbb{S}^{d-1}} 
\int_{\R^d} |x \cdot \theta| \, \dd \mu_\psi
\]
which is guaranteed to be a positive constant by Lemma \ref{lem.moms_essentially_continuous} above. 
}

Regarding the lemma above, we note that, alternatively, Santambrogio proposed an optimal transport proof of the same result by exploring the \textit{geodesic convexity} of the following functional 
\begin{equation}
    \min_{\varrho \in \mathscr{P}_1(\mathbb{R}^d)} 
    \mathscr{E}_\mu(\varrho)
    \eqdef 
    H(\varrho) + \mathcal{T}(\varrho, \mu),  
\end{equation}
where $H$ corresponds to the negative entropy functional~\eqref{eq.entropy} and $\mathcal{T}$ denotes Brenier's formulation of the quadratic optimal transport problem, or maximimal correlation problem~\eqref{eq.max_correlation}. Santambrogio relates both problems and recovers the characterization of moment measures with minimizers of $\mathscr{E}_\mu$, which are log-concave probability densities whose potential is essentially continuous, thus recovering the results of  Cordero-Eurasquin and Klartag. 

More recently, Delalande and Farineli~\cite{delalande2025regularized} studied the quantitative stability of \textit{regularized moment measures}, which they define via the unique minimizer of the following functional, which consists of a second moment regularized version of $\mathscr{E}_\mu$: 
\begin{equation}
    \mathscr{E}_{\mu,\alpha}(\varrho) 
    \eqdef 
    H(\varrho) + \mathcal{T}(\varrho, \mu) + \frac{\alpha}{2}M_2(\varrho). 
\end{equation}
Minimizers of the regularized version are now unique and of the following form: there exists $\psi_\alpha$ convex such that
\[
  \varrho_\alpha 
  = 
  e^{-(\psi_\alpha + \frac{\alpha}{2}|x|^2)} \text{ and }
  {(\nabla \psi_\alpha)}_\sharp e^{-(\psi_\alpha + \frac{\alpha}{2}|x|^2)} = \mu. 
\]

A major interest of introducing regularized moment measures is computational; indeed regularizing the functional $\mathscr{E}_\mu$ with the second moment term $M_2(\varrho)$ makes it $\alpha$-strongly geodesicaly convex in the Wasserstein geometry while preserving the most important feature of the representation via moment measures, that is being the push-forward of a Gibbs measure through (a perturbation of) its potential's gradient. For these reasons, its not surprising that the arguments in~\cite{delalande2025regularized} are based on strong geodesic convexity and actually do not require the dual formulation of $\mathscr{E}_{\mu,\alpha}$. 

To fully exploit the sharp stability version of Prékopa-Leindler's inequality, we shall require the connection of both primal and dual formulations. For this reason we introduce the regularized dual functional defined for any $\alpha \ge 0$ as
\begin{equation}
    \mathscr{J}_{\mu,\alpha}(\varphi) 
    \eqdef 
    \log
    \int_{\mathbb{R}^d} e^{-(\varphi^* + \frac{\alpha}{2}|x|^2)} \dd x 
    - 
    \int_{\mathbb{R}^d} \varphi \dd \mu.
\end{equation}
In the following theorem, we summarize the contributions from the literature and show a strong-duality type result with a small adaptation of a remark from Santambrogio in~\cite{santambrogio2016dealing}, but that will be useful in the sequel. More importantly, we identify a way to map minimizers of $\mathscr{E}_{\mu,\alpha}$ into maximizers of $\mathscr{J}_{\mu,\alpha}$ and vice-versa. 
\begin{theorem}[\cite{cordero2015moment,santambrogio2016dealing,delalande2025regularized}]\label{thm.characterization_moment_measures}
    Let $\mu$ be a Radon probability measure in $\mathscr{P}_1(\mathbb{R}^d)$ whose barycenter lies at the origin and such that $\dim \supp \mu = d$. The following assertions hold:
    \begin{itemize}
        \item For every $\alpha \ge 0$ there is strong duality: 
        \begin{equation}\label{eq.strong_duality}
            \sup \mathscr{J}_{\mu,\alpha} = - \inf \mathscr{E}_{\mu,\alpha}.
        \end{equation}
        \item 
        $\mathscr{E}_{\mu,\alpha}$ admits a minimizer of the form 
        \begin{equation}
            \varrho_\alpha = 
            e^{-(\psi_\alpha + \frac{\alpha}{2}|x|^2)} \text{ and }
            {(\nabla \psi_\alpha)}_\sharp \varrho_\alpha = \mu,
        \end{equation}
        where $\psi_\alpha$ is an essentially continuous and convex function, so that $\nabla \psi_\alpha$ is the unique Brenier map from $\varrho_\alpha$ to $\mu$.

        If $\alpha > 0$ this minimizer is unique; and if $\alpha = 0$ it is unique up to translations. In any case, if $\varrho_\alpha$ is a minimizer, defining
        \begin{equation}
            \varphi_\alpha \eqdef \psi_\alpha^* \text{ for } 
            \psi_\alpha \eqdef -\log\varrho_\alpha - \frac{\alpha}{2}|x|^2, 
        \end{equation}
        it follows that $\varphi_\alpha$ is a maximizer of $\mathscr{J}_{\mu,\alpha}$.
        \item $\mathscr{J}_{\mu,\alpha}$ admits a maximizer $\varphi_\alpha$, whose Legendre transform $\psi_\alpha \eqdef \varphi_\alpha^*$ is an essentially continuous and convex function, and gives the regularized moment measure representation for $\mu$ 
        \begin{equation}
            {(\nabla \psi_\alpha)}_\sharp e^{-(\psi_\alpha + \frac{\alpha}{2}|x|^2)} = \mu
        \end{equation}

        If $\alpha > 0$, the maximizer is unique; whereas if $\alpha = 0$ it is unique up to the addition of an affine function. 

        In any case, if $\varphi_\alpha$ is a maximizer, defining $\psi_\alpha = \varphi_\alpha^*$ it follows that
        \begin{equation}
            \varrho_\alpha \eqdef e^{-(\psi_\alpha + \frac{\alpha}{2}|x|^2)} 
        \end{equation}
        is a minimizer of $\mathscr{E}_{\mu,\alpha}$.
    \end{itemize}
\end{theorem}
\begin{proof}
    First let us prove the strong duality type-result. It was already proven in the case $\alpha = 0$ in~\cite{santambrogio2016dealing}, we prove it here for general $\alpha$ for completeness and for the reader's convenience. 

    First, we compute the Legendre transform of $\mathscr{E}_{\mu,\alpha}$. It is defined as a functional over the space of convex functions as 
    \begin{align*}
        \mathscr{E}_{\mu,\alpha}^*(f) 
        &\eqdef 
        \sup_{ \rho \in \mathscr{P}(\mathbb{R}^d)} 
        \inner{f, \rho} - H(\rho) 
        - \frac{\alpha}{2} \int_{\mathbb{R}^d}|x|^2\dd \rho
        - \inf_{\varphi \text{ convex }} 
        \int_{\mathbb{R}^d} \varphi^* \dd \rho +
        \int_{\mathbb{R}^d} \varphi \dd \mu\\ 
        &= 
        \sup_{\varphi \text{ convex }} 
        -  \int_{\mathbb{R}^d} \varphi \dd \mu 
        + 
        \sup_{ \rho \in \Pac(\mathbb{R}^d)} 
        \inner{f - \varphi^* - \frac{\alpha}{2}|x|^2, \rho} 
        - \int_{\mathbb{R}^d}\log\rho\dd\rho
    \end{align*}
    The supremum on $\rho$ is attained, and the maximizer is characterized by the following Euler-Lagrange equations; see~\cite{santambrogio2015optimal}: $\rho > 0$ a.e. over $\mathbb{R}^d$ and 
    \[
        \log \rho + \varphi^* + \frac{\alpha}{2}|x|^2 - f \equiv \text{cte},
    \]
    so that $\rho \propto e^{-( \varphi^* + \frac{\alpha}{2}|x|^2 - f)}$. A straightforward computation then gives 
    \[
        \mathscr{E}_{\mu,\alpha}^*(f) 
        = 
        \sup_{\varphi \text{ convex }} 
        \log
        \int_{\mathbb{R}^d} e^{-( \varphi^* + \frac{\alpha}{2}|x|^2 - f)} \dd x 
        -
        \int_{\mathbb{R}^d} \varphi \dd \mu. 
    \]

    As a result, taking $f \equiv 0$ we get that 
    \[
        \sup_{\varphi \text{ convex }} \mathscr{J}_{\mu,\alpha}(\varphi) 
        = 
        -
        \inf_{\varrho \in \mathscr{P}_1(\mathbb{R}^d)} 
        \mathscr{E}_{\mu,\alpha}(\varrho). 
    \]

    The existence of minimizers for $\mathscr{E}_{\mu,\alpha}$ and maximizers for $\mathscr{J}_{\mu,\alpha}$ is completely equivalent, as we will see in the sequel. However, the proof of existence for $\mathscr{E}_{\mu,\alpha}$ and general $\alpha\ge 0$ has been done in~\cite{delalande2025regularized}, along with the claimed characterization of minimizers and the relation with moment measures. Therefore, to be economic the reader referred to this work for a proof of the second assertion, we only comment that the uniqueness in the case $\alpha > 0$ stems from the strong geodesic convexity of the energy, while the functional $\mathscr{E}_\mu$ is invariant with respect to translations, \cite[Prop.~3.1]{santambrogio2016dealing}. 

    Now given a minimizer $\varrho_\alpha$, for any $\alpha \ge 0$, consider $\varrho_\alpha$ as defined above and let us compute $\mathscr{J}_{\mu, \alpha}(\varrho_\alpha)$ as follows 
    \begin{align*}
        \mathscr{J}_{\mu, \alpha}(\varphi_\alpha) 
        &= 
        \log
        \int_{\mathbb{R}^d} 
        e^{-(\varphi_\alpha^* +\frac{\alpha}{2}|x|^2)}\dd x 
        - 
        \int_{\mathbb{R}^d}
        \varphi_\alpha \dd \mu\\ 
        &= 
        \underbrace{
            \log 
            \int_{\mathbb{R}^d} 
            e^{\log \varrho_\alpha}\dd x 
        }_{ = 0} 
        - 
        \int_{\mathbb{R}^d}
        \varphi_\alpha \dd \mu.
    \end{align*}
    Now using the regularized moment measure representation for $\mu$ in terms of $\psi_\alpha = \varphi_\alpha^*$ and Fenchel's identity $\varphi_\alpha(\nabla \varphi^*_\alpha(x)) = \inner{x, \nabla \varphi^*_\alpha(x)} - \varphi_\alpha^*(x)$, it holds that 
    \begin{align*}
        \mathscr{J}_{\mu, \alpha}(\varphi_\alpha) 
        &= 
        - 
        \int_{\mathbb{R}^d}
        \varphi_\alpha(\nabla \varphi^*_\alpha(x))
        e^{-(\varphi_\alpha^*(x) + \frac{\alpha}{2}|x|^2)}\dd x \\ 
        &= 
        - 
        \underbrace{
            \int_{\mathbb{R}^d} 
            \inner{x, \nabla \varphi^*_\alpha(x)} \dd \varrho_\alpha
        }_{ = \mathcal{T}(\varrho_\alpha, \mu) }
        +
        \underbrace{
             \int_{\mathbb{R}^d}  
             \left(
                \varphi^*_\alpha(x) + \frac{\alpha}{2}|x|^2
             \right)\dd \varrho_\alpha
        }_{ = - H(\varrho_\alpha)}
        - 
        \frac{\alpha}{2}
        \int_{\mathbb{R}^d}|x|^2\dd \varrho_\alpha\\ 
        &= 
        -\mathscr{E}_{\mu,\alpha}(\varrho_\alpha).
    \end{align*}

    As a result, we finally obtain that 
    \[
        \sup \mathscr{J}_{\mu, \alpha} \ge \mathscr{J}_{\mu, \alpha}(\psi_\alpha) 
        = 
        -\mathscr{E}_{\mu,\alpha}(\varrho_\alpha)
        = 
        -\inf \mathscr{E}_{\mu,\alpha}
        = 
        \sup \mathscr{J}_{\mu, \alpha},
    \]
    so that $\varphi_\alpha$ must be a maximizer of $\mathscr{J}_{\mu,\alpha}$. 
    
    Doing the same reasoning, but on the inverse, starting an arbitrary maximizer of $\mathscr{J}_{\mu,\alpha}$, we can construct a log-concave probability measure which minimizes $\mathscr{E}_{\mu,\alpha}$, but since those are unique thanks to the strong geodesic convexity, or at least unique up to a translation in the case $\alpha = 0$, we obtain that the maximizers of $\mathscr{J}_{\mu,\alpha}$ must also be unique, or unique up to adding an affine function in the case $\alpha = 0$. 
\end{proof}

It can also be shown that the Gibbs measure associated with the moment measure (resp. regularized moment measure) representation of a given $\mu$ is a minimizer for $\mathscr{E}_\mu$ (resp. $\mathscr{E}_{\mu,\alpha}$), see~\cite[Prop.~5.1]{santambrogio2016dealing} for classical moment measures,~\cite[Prop.~4.2]{delalande2025regularized} and~\cite[Thm.~8]{cordero2015moment} for the corresponding statement for the dual functional $\mathscr{J}_\mu$.

The moment measure representation can be seen as a stationarity condition for $\mathscr{J}_\mu$ and its maximazation can be studied through its concavity. Indeed, it is shown in~\cite[Theorem 8 and Proposition 9]{cordero2015moment} that the moment measure $\mu_{\varphi^*}$ belong to the subdifferential (in the sense of concave functions) of the log-concave volume functional $\mathscr{J}$ at $\varphi$ if and only if $\varphi^*$ is essentially continuous. In a series of works~\cite{rotem2022surface,rotem2023anisotropic,falah2026functional} the first variation of the functional $\mathscr{J}$ has become clearer; in fact performing variations w.r.t. an arbitrary convex function might induce variations on the domain, but these can only be perceived if $\varrho_{\varphi^*}$ has a positive jump on the boundary of $\dom \varphi^*$, which happens only on a set of $\mathscr{H}^{d-1}$-negligible measure if $\varphi^*$ is essentially continuous. 

The second variation is an even more delicate matter. In the sequel we provide an argument for the derivative of variations induced by convex interpolations assuming that one end-point is strongly convex. To our surprise, the second variation can be formulated in terms of the deficit of the Brascamp-Lieb inequality. This is only a first step to understanding the second variation in full generality, as the strong convexity assumption is quite restrictive, but it is enough for our quantitative stability results derived from the stability of Brascamp-Lieb.

\begin{lemma}\label{lemma.first_second_variation}
    The first and second variations of $\mathscr{J}_\mu$ are computed as follows: 
    given two essentially continuous and convex functions $\varphi_0,\varphi_1$ such that for $i = 0,1$ we have 
    \[
        0 < \int_{\mathbb{R}^d} e^{\varphi_i^*} < + \infty. 
    \]
    Setting $v \eqdef \varphi_1 - \varphi_0$, $\varphi_t \eqdef \varphi_0 + tv$, and the Gibbs probability measure $\displaystyle \varrho_{\varphi_t^*,\alpha} \propto e^{-(\varphi^*_t + \frac{\alpha}{2}|x|^2)}$ and its regularized moment measure $\mu_{\varphi_t^*,\alpha} = {(\nabla \varphi_t^*)}_\sharp \varrho_{\varphi_t^*,\alpha}$ then the following hold
    \begin{enumerate}
        \item the function $t \mapsto \mathscr{J}_{\mu, \alpha}(\varphi_t)$ is concave, and therefore the time derivatives $\frac{\dd}{\dd t}
       \mathscr{J}_{\mu, \alpha}(\varphi_t)$ and $\frac{\dd^2}{\dd t^2}
       \mathscr{J}_{\mu,\alpha}(\varphi_t)$ exist for a.e. $t \in (0,1)$.
       \item For every $t \in (0,1)$ such that the first time derivative exists, we have 
        \begin{align*}
            \frac{\dd}{\dd t}
            \mathscr{J}_{\mu, \alpha}(\varphi_t) 
            = 
            \int_{\mathbb{R}^d} v \dd (\mu_{\varphi_t^*,\alpha} -\mu).
        \end{align*}
       \item Assuming in addition that $\varphi_1$ is $\kappa$-strongly convex, for every $t$ where the second derivative exists it holds that 
       \begin{align*}
            \frac{\dd^2}{\dd t^2}
            \mathscr{J}_{\mu,\alpha}(\varphi_t)
            =
            \Var_{\varrho_{\varphi^*_{t},\alpha}}(w) 
            - 
            \int_{\mathbb{R}^d} 
            \inner{
                {(D^2\varphi_t^*)}^{-1} \nabla w, \nabla w 
                }
            \dd \varrho_{\varphi^*_{t},\alpha},
        \end{align*}
        where $w(x) \eqdef v(\nabla \varphi^*(x))$. 
    \end{enumerate}
\end{lemma}
\begin{proof}
    The concavity follows from the Prékopa-Leindler inequality, as proven in~\cite{cordero2015moment}. As a result, $t \mapsto \mathscr{J}_{\mu, \alpha}(\varphi_t)$ is automatically twice differentiable a.e., and at every point where the first time derivative exists, it can be computed with the right derivative. Therefore, from~\cite[Theorem~1.5]{rotem2023anisotropic}, since both functions are essentially continuous it holds that 
    \[
        \left. 
            \log \int_{\mathbb{R}^d} e^{-\varphi_t^* + \frac{\alpha}{2}|x|^2}\dd x
        \right|_{t = 0^+}
        = 
        \int_{\mathbb{R}^d} v \dd \mu_{\varphi_t^*,\alpha},
    \]    
    so the first characterization follows. 

    This result stems from the first derivative in time for $\varphi_t^*(x)$ from~\cite[Proposition~2.1]{rotem2022surface}
    \[
        \frac{\dd}{\dd t} \varphi_t^*(x) = -v(\nabla \varphi_t^*(x)),
    \]
    whenever $v$ is lower semi-continuous. Here $v = \varphi_1 - \varphi_0$, hence it is lower semi-continuous in the domain of $\varphi_0$. 

    Now, assume that $\varphi_1$ is $\kappa$-strongly convex, so that $\varphi_t$ is $t\kappa$-strongly convex. Let $y_t$ be optimal for the supremum defining $\varphi_t^*(x)$, from the strong convexity of $\varphi_t$ it is uniquely defined and  
    \begin{equation}
        y_t(x) \eqdef \argmax_{y \in \mathbb{R}^d} \inner{y,x} - \varphi_t(y), 
        \text{ so that }
        y_t = {\nabla \varphi_t}^{-1}(x) = \nabla \varphi_t^*(x),
    \end{equation}

    On the other hand, $(t,x) \mapsto F(t,x) \eqdef \varphi_t^*(x)$ is jointly convex, hence in particular it is a.e. twice differentiable. We now have
    \[
        \frac{\dd^2}{\dd t^2} \varphi_t^*(x) 
        = 
        -\frac{\dd}{\dd t} v(\nabla \varphi_t^*(x)) 
        = 
        -\nabla v(\nabla \varphi_t^*(x)) \cdot \partial_t y_t(x). 
    \]
    Taking derivatives w.r.t.~$t$ on the optimality condition defining $y_t$, that is $x = \nabla \varphi_t(y_t)$, we get
    \[
        \partial_t y_t = -D^2\varphi_t(y_t)^{-1}\nabla v(y_t).  
    \]
    As a result, assuming temporary enough regularity, we can then compute
    \begin{align}
        \frac{\dd^2}{\dd t^2}
       \mathscr{J}_{\mu, \alpha}(\varphi_t) 
       &= 
       - \frac{Z_t'}{Z_t} \int_{\mathbb{R}^d} v(\nabla\varphi_t^*) \dd\varrho_{\varphi_t^*,\alpha}
       + 
       \int_{\mathbb{R}^d} 
       \left( v(\nabla\varphi_t^*)^2 - \frac{\partial^2 \varphi_t^*}{\partial  t^2} \right)
       \dd \varrho_{\varphi_t^*,\alpha} \\ 
       \label{eq.second_variation_formula}
       &= 
       \Var_{\varrho_{\varphi_t^*,\alpha}}(w) 
       -
       \int_{\mathbb{R}^d} 
       \inner{
        {\left(D^2\varphi_t(\nabla \varphi_t^*)\right)} \nabla v \circ \nabla \varphi_t^*,
        \nabla v\circ \nabla \varphi_t^*
        } 
       \dd \varrho_{\varphi_t^*,\alpha}.
    \end{align}

    To make the above computation rigorous, it suffices to check that $\partial_t^2 \varphi_t^* e^{-(\varphi_t^* + \frac{\alpha}{2}|\cdot|^2)} \in L^1(\mathbb{R}^d)$ to use differentiation under the integral sign. For this, notice that from convexity of $F(t,x)$, we have that $\partial_{tt} F(t,x) \ge 0$. Next, consider the function 
    \[
        \Phi(x) \eqdef \max\{\varphi_0(x), \varphi_1(x)\}, 
    \]
    so that $\Phi^* \le \varphi_t^*$ and $e^{-\Phi^*}$ is integrable. Now fixing an interval $I = (a,b) \subset [0,1]$ we have that 
    \begin{align*}
        \int_{a}^{b} 
        \int_{\mathbb{R}^d} \partial_{tt} F(t,x) e^{-(\varphi_t^* + \frac{\alpha}{2}|\cdot|^2)} \dd x \dd t
        &\le 
        \int_{a}^{b} 
        \int_{\mathbb{R}^d} \partial_{tt} F(t,x) e^{-\Phi^*(x)} \dd x \dd t \\ 
        &
        =  
        \int_{\mathbb{R}^d} 
        e^{-\Phi^*(x)} \int_{a}^{b} \partial_{tt} F(t,x) \dd t \dd x \\ 
        &\le 
        \int_{\mathbb{R}^d} 
        e^{-\Phi^*(x)} \left(
            \partial_t F(b,x) - \partial_t F(a,x)
        \right) \dd x \\ 
        &\le 
         \int_{\mathbb{R}^d} 
        \left(
            |w_b(x)| + |w_a(x)|
        \right) e^{-\Phi^*(x)} \dd x,
    \end{align*} 
    where we have used the fact that the second derivative $\partial_{tt} F(t,x)$ is composed of an absolutely continuous part and an atomic measure with positive jumps~\cite{ambrosio2000functions,evans2015measure}. On the other hand, as the log-concave measure $e^{-\Phi(x)}$ has finite and positive total mass, all its moments are finite. 

    Hence to conclude, it suffices to show that 
    \begin{equation}\label{eq.growth_wt}
        |w_t(x)| \le C_I(1 + |x|^2). 
    \end{equation}
    To this end, we will show that 
    \begin{align}
        \label{eq.growth_yt}
        |y_t(x)| &\le C_I(1 + |x|), \\ 
        \label{eq.growth_varphit}
        |\varphi_t(x)| \le C_I(1 + |x|^2),
        \quad 
        |\varphi_0(x)| &\le C_I(1 + |x|^2),
        \quad \text{and}
        |\varphi_1(x)| \le C_I(1 + |x|^2). 
    \end{align}
    First, recall that from the convexity of $\varphi_0$ and the strong convexity of $\varphi_1$ it holds for some $p_0, p_1$, and $c_0,c_1$ that
    \[
        \varphi_0(y) \ge \inner{p_0,y} - c_0, \quad 
        \varphi_1(y) \ge \frac{\kappa}{2}|y|^2 + \inner{p_1,y} - c_1,
        \text{ and hence }
        \varphi_t(y) \ge \frac{t\kappa}{2}|y|^2 + \inner{\bar p,y} - \bar c. 
    \]
    So to prove~\eqref{eq.growth_yt}, fix some $\bar y$, so that the fact that $y_t(x)$ is optimal gives 
    \[
        \inner{x,y_t(x)} - \varphi_t(y_t(x)) \geq
        \inner{x,\bar y} - \varphi_t(\bar y). 
    \]
    As a result, we have 
    \[
        \varphi_t(y_t(x)) 
        \le \inner{x, y_t - \bar y} + \varphi_t(\bar y)
        \le |x||y_t| + C_I(1 + |x|). 
    \]
    Combining this with the strong convexity of $\varphi_t$ above we have that 
    \begin{align*}
        \frac{t\kappa}{2}|y_t|^2 
        &\le C_I(1 + |x|)|y_t| + C_I(1 + |x|) 
        &\le  \frac{t\kappa}{4}|y_t|^2 + 
        C_I(1 + |x|^2). 
    \end{align*}
    This shows~\eqref{eq.growth_yt} as desired. 

    For~\eqref{eq.growth_varphit}, we have from the Fenchel inequality that 
    \[
        \varphi_t^*(x) \ge \inner{x,\bar y} - \varphi_t(\bar y) 
        \ge - C_I(1 + |x|). 
    \]
    But from the equality case of Fenchel inequality it also holds that 
    \[
        \varphi_t(y_t) = 
        \inner{x,y_t} - \varphi_t^*(x) 
        \le 
        \inner{x,y_t} + C_I(1 + |x|)
        \le 
        C_I(1 + |x|^2). 
    \]
    For the opposite inequality, it suffices to recall that $\varphi_t(y_t) \ge - C_I(1 + |x|)$. We obtain the same estimates for $\varphi_i(y_t)$ for $i=0,1$ similarly. But more importantly, recalling that $w_t(x) = \varphi_1(y_t) - \varphi_0(y_t)$, from this we obtain~\eqref{eq.growth_wt} and the differentiation under the integral sign holds. 

    Finally, we remark that the gradient of quantity $w$ can be written as
    \[
        \nabla w(x)
        =
        \nabla(v(y_t)) 
        = 
        D^2\varphi_t^*(x)(\nabla v) \circ y_t.
    \]
    Therefore, using that $D^2\varphi_t(y_t)^{-1} = D^2\varphi_t^*(x)$, the previous computation gives 
    \[
        \frac{\dd^2}{\dd t^2} \varphi_t^*(x) 
        =
        \inner{{\left(D^2\varphi_t^*\right)}^{-1} \nabla w,
        \nabla w}, 
    \]
    which can be used in~\eqref{eq.second_variation_formula} above to conclude. 
\end{proof}

For $\alpha = 0$, the above computation for the second variation of $\mathscr{J}_\mu$ gives precisely minus the deficit of the Brascamp-Lieb inequality. In the case that $\alpha > 0$ we need to be a bit careful to perform a Taylor expansion on the Hessian and get the same deficit plus lower order terms in $\alpha$. Either way, these computations give a clear intuition on why Cordero-Eurasquin and Klartag's characterization should hold, but it does not provide the essential continuity of the potential, which is at the heart of their statement. 

\begin{remark}\label{remark.comparison_withCE}
    More importantly to our purposes, it becomes clear that having access to a sharp quantitative stability version of Brascamp-Lieb's inequality shall give stability results for the potentials of moment measure representations. As discussed in the introduction, there are a few next order expansions of the Brascamp-Lieb inequality in the literature~\cite{cordero2017transport,bolley2018dimensional,harge2008reinforcement}. However, not only our Theorem~\ref{thm.stability_moment_measures_Rd_Lambda} is more adapted to our purposes as we show that the deficit $\delta_{\BL}$ controls a distance, but also we do it with a multiplicative constant $C_d$ that is independent of the measure $\varrhovarphi$. This last feature is crucial for the applications to moment measures, as we shall see below.
\end{remark}

\subsection{State of the art on the Prékopa-Leindler inequality}
The original version of the Prékopa-Leindler inequality was initially proposed independently by Prékopa in~\cite{prekopa1971logarithmic,prekopa_1973} and Leindler~\cite{leindler_1973} and can be stated as follows: given a triplet of measurable functions $f,g,h : \mathbb{R}^d \to \mathbb{R}_+$ such that for a given parameter $s \in (0,1)$ the so-called \textit{Prékopa condition} holds 
\begin{equation}\label{eq.Prekopa_condition}
    f(x)^sg(y)^{1-s}
    \le h(s x + (1-s) y)
    \text{ for all } x,y \in \mathbb{R}^d,
\end{equation}
then we have the following bound
\begin{equation}\label{eq.PLinequality}
    {
    \left(
        \int_{\mathbb{R}^d}f(x)\dd x
    \right)
    }^{s}
    {
    \left(
        \int_{\mathbb{R}^d}g(x)\dd x
    \right)
    }^{1-s}
    \le 
    \int_{\mathbb{R}^d}h(x)\dd x. 
\end{equation}
This inequality is a functional generalization of the Brunn-Minkowski inequality, for which many stability results have been derived, see for instance the seminal work~\cite{figalli_stability_2010}, where stability version of such geometric inequalities have been derived.   

A stability version for the Prékopa-Leindler inequality quantifies how functions nearly satisfying the equality case must be close to an optimal log-concave profile. The result takes the following form: if $h, f, g: \mathbb{R}^d \to \mathbb{R}_+$ satisfy Prékopa condition~\eqref{eq.Prekopa_condition} and~\eqref{eq.PLinequality} is almost satisfied in the sense that 
\begin{equation}\label{eq.PL_eps_equality}
    \int_{\mathbb{R}^d} h(x)\dd x \leq (1 + \varepsilon) \left( \int_{\mathbb{R}^d} f(x)\dd x \right)^s \left( \int_{\mathbb{R}^d} g(x)\dd x \right)^{1 - s}.
\end{equation}
for some $\varepsilon>0$, there exists a log-concave function $\tilde{h}: \mathbb{R}^d \to \mathbb{R}_+$ and parameters $a > 0$, $x_0 \in \mathbb{R}^d$ such that the following proximity estimates hold:
\begin{equation}\label{eq.sharpPL_consequences}
    \begin{aligned}
        \int_{\mathbb{R}^d} |f(x) - a^{-s} \tilde{h}(x - s x_0)|  dx &\leq C(\tau) \varepsilon^{\alpha_n(\tau)} \int_{\mathbb{R}^d} f(x)  dx, \\
        \int_{\mathbb{R}^d} |g(x) - a^{1-s} \tilde{h}(x + (1-s)x_0)|  dx &\leq C(\tau) \varepsilon^{\alpha_n(\tau)} \int_{\mathbb{R}^d} g(x)  dx, \\
        \int_{\mathbb{R}^d} |h(x) - \tilde{h}(x)|  dx &\leq C(\tau) \varepsilon^{\alpha_n(\tau)} \int_{\mathbb{R}^d} h(x)  dx,
    \end{aligned}
\end{equation}
where  
\[
    \tau \eqdef \min(s, 1-s), \quad 
    a \eqdef 
    \frac{
        \int_{\mathbb{R}^d} f
    }{
        \int_{\mathbb{R}^d} g
    }
\]
$C(\tau)$ is a constant that depends on $\tau$, and $\alpha_n(\tau) > 0$ is a computable exponent that typically depends on both the dimension $n$ and the parameter $\tau$.

Obtaining sharper versions of this inequality is currently a very active area of research. On one side, the sharp exponent $\alpha$ is expected to be $1/2$, as is the case, for instance, for the isoperimetric and Brunn-Minkowski inequalities~\cite{figalli_stability_2010}. Indeed, this exponent has been obtained in~\cite{figalli2025sharp}, but at the cost of a sub-obtimal constant $C(\tau)$ that does not behave properly as $\tau \to 0^+$. On the other hand, in~\cite{figalli2024improved}, a quantitative stability version has been shown with $\alpha = 1/2$ and $C \propto \tau^{-1/2}$ in the 1D or radial in $\mathbb{R}^d$, when the involved functions are log-concave. 

In the present work, we will only need a stability version of Prékopa-Leindler for the case $\tau = 1/2$. Therefore, the $\tau-$dependence of the constant $C(\tau)$ is not particularly important for us, and we may hence use some of the results from~\cite{figalli2025sharp} directly, which we collect in the following Lemma. 

\begin{lemma}\label{lemma.sharp_stabilityPL}
    Let $f,g,h:\mathbb{R}^d \to \mathbb{R}_+$ be measurable functions satisfying Prékopa's condition~\eqref{eq.Prekopa_condition} with $s = 1/2$. If there is $\varepsilon > 0$ such that~\eqref{eq.PL_eps_equality} holds, there exists $x_0 \in \mathbb{R}^d$ and a dimensional constant $C_d$ such that 
    \[
        \int_{\mathbb{R}^d} 
        \left|
            \frac{f(x)}{\int_{\mathbb{R}^d} f} 
            - 
            \frac{g(x+x_0)}{\int_{\mathbb{R}^d} g}
        \right|\dd x 
        \le C_d \varepsilon^{1/2}.
    \]
\end{lemma}
\begin{proof}
    Recall $a \eqdef \frac{\int_{\mathbb{R}^d} f}{\int_{\mathbb{R}^d} g}$, so from~\cite{figalli2025sharp}, estimates~\eqref{eq.sharpPL_consequences} hold with $\alpha = 1/2$ and a constant $C(\tau) = C(1/2) = C_d$, that depends only on the dimension. So there is a log-concave function $\tilde h$ such that 
    \begin{equation*}
        \begin{aligned}
            \int_{\mathbb{R}^d} 
            \left|
                \frac{f(x)}{\int_{\mathbb{R}^d} f} 
                - \frac{\tilde{h}(x -  \frac{1}{2} x_0)}{
                {\left(\int_{\mathbb{R}^d} f
                \int_{\mathbb{R}^d} g\right)}^{1/2}
                } 
            \right|  \dd x &\leq C_d \varepsilon^{1/2}, \\
            \int_{\mathbb{R}^d} 
            \left| 
                \frac{g(x)}{\int_{\mathbb{R}^d} g} 
                - \frac{\tilde{h}(x + \frac{1}{2} x_0)}{
                {\left(\int_{\mathbb{R}^d} f
                \int_{\mathbb{R}^d} g\right)}^{1/2}
                } 
                \right|  \dd x &\leq C_d \varepsilon^{1/2}.
        \end{aligned}
    \end{equation*}
    A change of variables on the inequality concerning $g$ and an easy triangular inequelity give the desired bound. 
\end{proof}

\section{Sharp Quantitative Stability for Brascamp-Lieb inequality}\label{sec.stability_Brascamp_Lieb}
In this section we are focused on the quantitative stability of the Brascamp-Lieb inequality, as stated in Theorem \ref{thm.brascamp_lieb_quantitative_intro}. For the reader's convenience, we present a proof of this inequality proposed by Helffer~\cite{helffer1998remarks}, with the advantage that the equality cases appear naturally from the computations. This method of proof can now be reinterpreted as an application of the \textit{carré du champ} method~\cite{bakry1985diffusions,bakry2013analysis}, which is based on the geometry of the following elliptic operator 
\[
    Lu \eqdef \Delta u - \nabla \varphi\cdot\nabla u.
\]
This method is based on the so called $\Gamma$-calculus, which revolves around the following commutator operators
\begin{align*}
    \Gamma(f,g) 
    &\eqdef \frac{1}{2}\left(
        L(fg) - fLg - gLf 
    \right) = \inner{\nabla f, \nabla g},\\
    \Gamma_2(f,g) 
    &\eqdef \frac{1}{2}\left(
        L\Gamma(f,g) - \Gamma(f,Lg) - \Gamma(g,Lf) 
    \right) 
    = 
    D^2f\colon D^2g + \inner{D^2\varphi \nabla f, \nabla g}.
\end{align*}
It is then a straight-forward computation to check the integration by parts formula 
\begin{equation}
    \int_{\mathbb{R}^d}\Gamma(f,g)\dd \varrhovarphi = 
    \int_{\mathbb{R}^d}\inner{f,g}\dd \varrhovarphi = 
    - \int_{\mathbb{R}^d}fLg \dd \varrhovarphi = 
    - \int_{\mathbb{R}^d}gLf \dd \varrhovarphi.
\end{equation}

\begin{lemma}[\cite{brascamp1976extensions,helffer1998remarks}]
\label{lemma.brascamp_lieb_equality}
    Inequality~\eqref{eq.brascamp_lieb} holds, and we have equality if, and only if, there exists $a \in \mathbb{R}^d$ such that
    \[
        f(x) = a\cdot \nabla \varphi(x) + \mathbb{E}_{\varrhovarphi}f.
    \]
\end{lemma}
\begin{proof}
    Given $f \in \mathscr{C}^\infty_c(\mathbb{R}^d)$, set $\bar f = f - \mathbb{E}_{\varrhovarphi} f$ and let $u$ be the unique weak solution of 
    \[
        Lu = \bar f. 
    \]
    From the classical regularity theory of uniformly elliptic equations, $u \in \mathscr{C}^\infty(\mathbb{R}^d)$. So we can use the $\Gamma$ calculus tools to write the variance of $f$ as 
    \begin{align*}
        \Var_{\varrhovarphi}(f) 
        &= 
        \int_{\mathbb{R}^d} \bar f^2 \dd \varrhovarphi 
        = 
        \int_{\mathbb{R}^d} \bar fLu \dd \varrhovarphi 
        = 
        -\int_{\mathbb{R}^d} \inner{\nabla f, \nabla u}\dd \varrhovarphi\\ 
        &= 
        -\int_{\mathbb{R}^d} 
        \inner{{(D^2\varphi)}^{-1/2}\nabla f,{(D^2\varphi)}^{1/2}\nabla u}
        \dd \varrhovarphi\\ 
        &\le 
        {\left(
            \int_{\mathbb{R}^d} 
            \norm{{(D^2\varphi)}^{-1/2}\nabla f}^2
            \dd \varrhovarphi
        \right)}^{1/2}
        {\left(
            \int_{\mathbb{R}^d} 
            \norm{{(D^2\varphi)}^{1/2}\nabla u}^2
            \dd \varrhovarphi
        \right)}^{1/2}\\
        &=
        {\left(
            \int_{\mathbb{R}^d} 
            \inner{{(D^2\varphi)}^{-1}\nabla f,\nabla f}
            \dd \varrhovarphi
        \right)}^{1/2}
        {\left(
            \int_{\mathbb{R}^d} 
            \inner{{D^2\varphi}\nabla u, \nabla u}
            \dd \varrhovarphi
        \right)}^{1/2}.
    \end{align*}

    Now we notice that 
    \[
        \int_{\mathbb{R}^d}
        \Gamma_2(u)\dd \varrhovarphi 
        = 
        \frac{1}{2}
        \int_{\mathbb{R}^d}
        L|\nabla u|^2\dd \varrhovarphi  - 
        \int_{\mathbb{R}^d} \inner{\nabla u, \nabla Lu}\dd \varrhovarphi 
        = 
        - \int_{\mathbb{R}^d} \inner{\nabla u, \nabla Lu}\dd \varrhovarphi 
    \]
    since the integrand of the first integral is a divergence. In addition, by definition of $\Gamma_2$ we have that 
    \begin{align*}
        \int_{\mathbb{R}^d} 
        \inner{{D^2\varphi}\nabla u, \nabla u}
        \dd \varrhovarphi 
        &= 
        \int_{\mathbb{R}^d}
        \Gamma_2(u)\dd \varrhovarphi 
        - 
        \int_{\mathbb{R}^d}
        \norm{D^2u}^2_F\dd \varrhovarphi \\ 
        &= 
        - \int_{\mathbb{R}^d} \inner{\nabla u, \nabla Lu}\dd \varrhovarphi
         - 
        \int_{\mathbb{R}^d}
        \norm{D^2u}^2_F\dd \varrhovarphi \\ 
        &= 
        \Var_{\varrhovarphi}(f) - \int_{\mathbb{R}^d}
        \norm{D^2u}^2_F\dd \varrhovarphi. 
    \end{align*}

    Plugging it back into the original estimations of $\Var_{\varrhovarphi}(f)$, we notice that 
    \begin{align*}
        \Var_{\varrhovarphi}(f) 
        &\le 
        {\left(
            \int_{\mathbb{R}^d} 
            \inner{{(D^2\varphi)}^{-1}\nabla f,\nabla f}
            \dd \varrhovarphi
        \right)}^{1/2}
        {\left(
            \Var_{\varrhovarphi}(f) - \int_{\mathbb{R}^d}
            \norm{D^2u}^2_F\dd \varrhovarphi
        \right)}^{1/2}
    \end{align*}
    which gives inequality~\eqref{eq.brascamp_lieb}. This proof is also enlightening on the equality cases; indeed notice that the chain of inequalities are all equalities, except for the application of a Cauchy-Scharwz and the non-positive term on $-\norm{D^2u}_F^2$. 
    
    As a result, we have equality if, and only if, 
    \[
        D^2u \equiv 0 \text{ and } 
        {(D^2\varphi)}^{-1/2}\nabla f \parallel
        {(D^2\varphi)}^{1/2}\nabla u. 
    \] 
    The first condition implies that $u = a\cdot x + b$ for constant $a \in \mathbb{R}^d$ and $b \in \mathbb{R}$. Since $u$ is a classical solution of the equation $L u = \bar f$, it follows that, to have equality $f$ must satisfy
    \[
        f = a\cdot \nabla \varphi(x) + \mathbb{E}_{\varrhovarphi}f.
    \]
    The result follows. 
\end{proof}

\subsection{Sharp quantitative stability in $L^1$}\label{sec.stability_L^1}
With the knowledge of the equality case for~\eqref{eq.brascamp_lieb}, we now aim at a quantitative stability version of this inequality in terms of the $L^1$ distance to the optimal set (see Remark~\ref{remark.comparison_withCE}), as highlighted in Theorem \ref{thm.brascamp_lieb_quantitative_intro}. We state that result below once more for the reader's convenience.

\begin{theorem}\label{thm.brascamp_lieb_quantitative}
    There exists a universal constant $C_d$ depending only on the ambient dimension such that for any convex function $\varphi : \mathbb{R}^d \to \mathbb{R}\cup\{+\infty\}$ such that
    \[
        0 < \int_{\mathbb{R}^d} e^{-\varphi} < +\infty 
        \text{ and }
        \Theta(\mu_\varphi) > 0,
    \]
    we have that
    \[
        \dist_{L^1(\varrhovarphi)}(f, \mathcal{O}_{\BL})
        \le 
        C_d {\delta_{\BL}(f)}^{1/2}
    \]
    for all locally Lipschitz functions $f \in L^2(\varrhovarphi)$.   {Moreover, the inequality above is sharp, in the sense that the factor $\delta_{\BL}(f)^{1/2}$ \emph{cannot} be replaced by $\delta_{\BL}(f)^{\alpha}$ for any $\alpha > 1/2.$}
\end{theorem}

Our proof of Theorem \ref{thm.brascamp_lieb_quantitative} will exploit, as mentioned in the introduction, the linearization argument from Bobkov-Ledoux~\cite{bobkov2000brunn}, while employing the main results from \cite{figalli2025sharp} in order to control the deficit. While this allows us to take several important steps towards the proof, the crucial ingredient is the geometry induced by the non-concentration of the moment measure. 

\begin{proof}
    We first prove the result in the case that $\varphi$ is strongly convex. That is, we assume that there exists $\alpha > 0$ such that $D^2\varphi \ge \alpha \id$, where the gradients and hessians $\nabla \varphi, D^2\varphi$ are understood in the sense of Alexandrov (see for instance~\cite[Chapter 6]{evans2015measure}). Therefore, they are $\mathscr{L}^d$ and $\varrhovarphi$~a.e.~well defined. 
    
    Given $f \in L^2(\varrhovarphi)$, it can be arbitrarily approximated in $L^2(\varrhovarphi)$ by a sequence of smooth functions, so for the moment we assume $f \in \mathscr{C}^\infty_c(\mathbb{R}^d)$. As a result, from the $\alpha$-strong convexity of $\varphi$ there exists $\delta_0 > 0$ such that for all $\delta < \delta_0$ the function $2\delta f - \varphi$ is concave. 

    Our goal is to apply the quantitative stability version of the Prékopa-Leindler inequality with convexity parameter $s = 1/2$. In order to do so, we consider the functions
    \[
              u_\delta \eqdef e^{2\delta f - \varphi}, 
        \quad v_\delta \eqdef e^{-\varphi}, 
        \quad w_\delta \eqdef e^{f_\delta - \varphi},
    \]
    where $f_\delta(z)$ is pointwise defined as 
    \begin{align*}
        f_\delta(z) 
        &= 
        \sup_{ z = \frac{x+y}{2}} 
        \delta f(x) - 
        \left[
            \frac{1}{2}\varphi(x) + \frac{1}{2}\varphi(y) 
            -\varphi\left(\frac{x+y}{2}\right)  
        \right]\\
        &= 
        \sup_{h \in \mathbb{R}^d} 
        \delta f(z + h) - 
        \left[
            \frac{1}{2}\varphi(z+h) + \frac{1}{2}\varphi(z-h) 
            -\varphi(z)
        \right]. 
    \end{align*}

    It is not too hard to check that the triple of functions $u_\delta,v_\delta,w_\delta$ satisfies the condition in Prékopa--Leindler's inequality. Hence, as the total mass of $v_\delta$ is $1$ by assumption, the standard Prékopa-Leindler inequality holds and we can define the non-negative quantity $\varepsilon_\delta$ such that 
    \begin{equation}
        \int_{\mathbb{R}^d}w_\delta \dd x 
        = 
        (1 + \varepsilon_\delta){\left(
            \int_{\mathbb{R}^d}u_\delta \dd x
        \right)}^{1/2}, \quad 
        \varepsilon_\delta \eqdef 
        \frac{
            \int_{\mathbb{R}^d}w_\delta \dd x  - 
            {\left(
                \int_{\mathbb{R}^d}u_\delta \dd x
            \right)}^{1/2}
        }{{\left(
            \int_{\mathbb{R}^d}u_\delta \dd x
        \right)}^{1/2}}. 
    \end{equation} 

    Analogously as in Bobkov-Ledoux~\cite{bobkov2000brunn}, our argument to obtain the sharp quantity version of the Brascamp-Lieb inequality will  consist of studying the dependence on $\delta$ of the above quantities. Starting with the integral of $u_\delta$, we get the following Taylor expansion on $\delta$: 
    \begin{equation}
        \int_{\mathbb{R}^d}u_\delta \dd x 
        = 
        1 + 2\delta \int_{\mathbb{R}^d}f\dd \varrhovarphi 
        + 2\delta^2 \int_{\mathbb{R}^d}f^2\dd \varrhovarphi + o(\delta^2).
    \end{equation}

    For the expansion on the integral of $w_\delta$, recall the classical result from log-concave measures that 
    \[
        \text{ if } 
        \int_{\mathbb{R}^d} e^{-\varphi} < +\infty, 
        \text{ then $\varphi$ is coercive.}
    \]
    Therefore, the quantity being maximized in the definition of $f_\delta$ is strongly concave and goes to $-\infty$ as $h \to +\infty$. Hence for $\varrhovarphi$-a.e.~$z$ let $h_\delta = h_\delta(z)$ be optimal for the supremum in $f_\delta(z)$. The optimality conditions for $h_\delta$ give, for $\varrhovarphi$-a.e.~point where the gradients and hessian of $\varphi$ are uniquely defined, that
    \begin{align*}
        \delta \nabla f(z+h_\delta) 
        &= 
        \frac{1}{2}\left[  
            \nabla\varphi(z+h_\delta) - \nabla\varphi(z-h_\delta) 
        \right]
        = D^2\varphi(z)h_\delta + o(\delta). 
    \end{align*}
    Since $D^2\varphi \ge \alpha \id$, the hessian is invertible and ${(D^2\varphi)}^{-1} \le \alpha^{-1}\id$, and developing the Taylor expansion for $\nabla f$ in the identity above we obtain that
    \[
        h_\delta = \delta {(D^2\varphi)}^{-1}\nabla f(z) + o(\delta).
    \]
    Note that we have critically used the hessian bound above in order to be able to expand $h_\delta$ until first order in $\delta.$ In particular $h_\delta = O(\delta)$, uniformly in $z$. As a result, we can expand the value of $f_\delta(z)$ as 
    \[
        f_\delta(z) 
        = 
        \delta f(z) + \frac{\delta^2}{2}\inner{{(D^2\varphi)}^{-1}\nabla f, \nabla f} + o(\delta^2).
    \]
    Therefore, with a similar argument to the expansion of the integral of $u_\delta$, we get 
    \begin{align*}
       \int_{\mathbb{R}^d}w_\delta \dd x  
       &= \int_{\mathbb{R}^d} e^{f_\delta} \dd \varrhovarphi\\ 
       &= 
       1 + \delta \int_{\mathbb{R}^d} f \dd\varrhovarphi
       + \frac{\delta^2}{2} \int_{\mathbb{R}^d} f^2 \dd \varrhovarphi
       + \frac{\delta^2}{2} \int_{\mathbb{R}^d} 
       \inner{{(D^2\varphi)}^{-1}\nabla f, \nabla f} \dd \varrhovarphi + o(\delta^2).
    \end{align*}

    In addition, using the Taylor expansion ${(1 + s)}^{1/2} = 1 +\frac{1}{2}s - \frac{1}{8}s^2 + o(s^2)$ and the expansion for the integral of $u_\delta$ we obtain that
    \[
        {
        \left(
            \int_{\mathbb{R}^d}u_\delta \dd x 
        \right)
        }^{1/2} 
        = 
        1 + \delta \int_{\mathbb{R}^d} f \dd\varrhovarphi 
        + \delta^2\int_{\mathbb{R}^d} f^2 \dd \varrhovarphi
        -\frac{\delta^2}{2}{\left( \int_{\mathbb{R}^d} f \dd\varrhovarphi \right)}^2 + o(\delta^2).
    \]

    Joining these estimates, we obtain the following expansion for $\varepsilon_\delta$: 
    \begin{align*}
        \varepsilon_\delta
        &= 
        {\left(\int_{\mathbb{R}^d}u_\delta \dd x\right)}^{-1/2} 
        \left(
            \int_{\mathbb{R}^d}w_\delta \dd x - {\left(\int_{\mathbb{R}^d}u_\delta \dd x\right)}^{1/2}
        \right)\\ 
        &=  
        (1 + O(\delta))
        \frac{\delta^2}{2} 
        \left[
            \int_{\mathbb{R}^d} \inner{{(D^2\varphi)}^{-1}\nabla f, \nabla f} \dd \varrhovarphi 
            - 
            \Var_{\varrhovarphi}(f) + \frac{o(\delta^2)}{\delta^2}
        \right]\\ 
        &= 
        (1 + O(\delta))
        \frac{\delta^2}{2} 
        \left[
            \delta_{\BL}(f)
            + \frac{o(\delta^2)}{\delta^2}
        \right]
    \end{align*}

    From the quantitative stability of the Prékopa-Leindler inequality we obtain that there exists a universal constant $C_d$, depending only on the dimension and on the convexity parameter which is fixed $s = 1/2$, and some vector $x_\delta \in \mathbb{R}^d$ such that 
    \begin{equation}\label{eq.estimation_eps_delta}
        \int_{\mathbb{R}^d} 
        \left|
            \frac{u_\delta(x)}{\int_{\mathbb{R}^d} u_\delta} - e^{-\varphi(x + x_\delta)}
        \right|\dd x
        \le 
        C_d \varepsilon_\delta^{1/2}.
    \end{equation}
    Taking into account the scaling of $\varepsilon_\delta$, which behaves as $\delta^2$, we see that the right-hand side of the above inequality is of the order $\delta$. Developing the left-hand side depends strongly on the behavior of $x_\delta$ when $\delta\to 0$. Clearly it must hold that $x_\delta \to 0$ as $\delta \to 0$, as a result we obtain that
    \begin{align*}
        \int_{\mathbb{R}^d}& 
        \left|
            \frac{u_\delta(x)}{\int_{\mathbb{R}^d} u_\delta} - e^{-\varphi(x + x_\delta)}
        \right|\dd x 
        = 
        {\left(\int_{\mathbb{R}^d} u_\delta\right)}^{-1}
        \int_{\mathbb{R}^d} 
        \left|
            e^{2\delta f} - e^{-( \varphi(x + x_\delta) - \varphi(x))}\int_{\mathbb{R}^d} u_\delta
        \right|\dd \varrhovarphi \\ 
        &= 
        (1 + O(\delta))
        \int_{\mathbb{R}^d} 
        \left|
            1 + 2\delta f + o(\delta) - 
            (1 + \nabla\varphi(x)\cdot x_\delta + o(x_\delta))\left(   
            1 + 2\delta \mathbb{E}_{\varrhovarphi} f + o(\delta)
            \right)
        \right|\dd \varrhovarphi\\ 
        &= 
        (1 + O(\delta))2\delta
        \int_{\mathbb{R}^d} 
        \left|
                f - \nabla\varphi(x)\cdot \frac{x_\delta}{2\delta} - \mathbb{E}_{\varrhovarphi}f
        \right|\dd \varrhovarphi
        + o(\delta + x_\delta). 
    \end{align*}
 
    Therefore, in order to obtain a meaningful estimate in terms of $\dist_{L^1(\varrhovarphi)}(f, \mathcal{O}_{\BL})$, we must show that $x_\delta/2\delta$ has a cluster point as $\delta \to 0^+$. For this, define the sets
    \[
        \Omega_n \eqdef \left\{
            x \in \dom \varphi : 
            D^2\varphi(x) \le n \id
        \right\}. 
    \]
    From the convexity of $\varphi$, for every $n \in \mathbb{N}$ we can choose $\delta$ sufficiently small so that $\varphi(x+x_\delta) - \varphi(x) = O(x_\delta)$. More precisely, there is $\delta_n$ such that for $\delta \le \delta_n$ we have
    \begin{align*}
        \norm{v_\delta - v_\delta(\cdot + x_\delta)}_{L^1(\mathbb{R}^d)} 
        &\ge  
        \norm{v_\delta - v_\delta(\cdot + x_\delta)}_{L^1(\Omega_n)} \\ 
        &= 
        \int_{\Omega_n} 
        \left|
            x_\delta\cdot\nabla\varphi(x) + 
            \frac{1}{2} \inner{D^2\varphi(x)x_\delta,x_\delta} 
            + o(x_\delta^2)
        \right|\dd \varrhovarphi \\
        &\ge 
        \frac{1}{2} 
         \int_{\Omega_n}
        |x_\delta\cdot\nabla\varphi(x)|\dd \varrhovarphi.     
    \end{align*}
    
    Using a suitable triangle inequality, we get for $\delta \le \delta_n$
    \begin{align*}
        \frac{1}{2} 
        \int_{\Omega_n}
        |x_\delta\cdot\nabla\varphi(x)|\dd \varrhovarphi
        &\le 
        \norm{v_\delta - v_\delta(\cdot + x_\delta)}_{L^1(\mathbb{R}^d)}\\
        &\le 
        \norm{v_\delta - u_\delta}_{L^1(\mathbb{R}^d)} 
        + 
        \norm{u_\delta - v_\delta(\cdot + x_\delta)}_{L^1(\mathbb{R}^d)} \le C\delta,
    \end{align*} 
    where the bound $\norm{v_\delta - u_\delta}_{L^1(\mathbb{R}^d)}  = O(\delta)$ follows directly from the definition of $u_\delta$ as a perturbation of $v_\delta$ and $\norm{u_\delta - v_\delta(\cdot + x_\delta)}_{L^1(\mathbb{R}^d)} = O(\delta)$ is a consequence of the previous estimations on the deficit of the Prékopa-Leindler inequality. But then we conclude that the quantity
    \[
        \int_{\Omega_n}
        \left|
            \frac{x_\delta}{2\delta}\cdot\nabla\varphi(x)
        \right|\dd \varrhovarphi
    \]
    must remain uniformly bounded for $\delta \le \delta_n$. Now we exploit the assumption that the moment measure $\mu_\varphi$ has full-dimensional support. Since $x_\delta/\norm{x_\delta} \in \mathbb{S}^{d-1}$, we can find a  subsequence of $\delta_n \cvstrong{k \to \infty}{}0$ such that $x_{\delta_n}/\norm{x_{\delta_n}} \cvstrong{k \to \infty}{} \bar \theta$, for some $\bar \theta \in \mathbb{S}^{d-1}$. But then using Lebesgue's dominated convergence theorem and the fact that $\mathds{1}_{\Omega_n}$ converges strongly in $L^1(\varrhovarphi)$ to $1$, we get that 
    \begin{align*}
        C
        &\ge \limsup_{n \to \infty} 
        \int_{\Omega_n}
        \left|
            \frac{x_{\delta_n}}{2{\delta_n}}\cdot\nabla\varphi(x)
        \right|\dd \varrhovarphi
        = 
        \limsup_{n \to \infty} 
        \frac{\norm{x_{\delta_n}}}{2{\delta_n}}
        \int_{\Omega_n}
        \left|
            \frac{x_{\delta_n}}{\norm{x_{\delta_n}}}\cdot\nabla\varphi(x)
        \right|\dd \varrhovarphi\\
        &= 
        \left(
            \int_{\mathbb{R}^d}
            |\bar \theta \cdot y|\dd \mu_\varphi(y)
        \right)
        \limsup_{n \to \infty} 
        \frac{\norm{x_{\delta_n}}}{2{\delta_n}}
        \ge 
        \Theta(\mu_\varphi) \limsup_{n \to \infty} 
        \frac{\norm{x_{\delta_n}}}{2{\delta_n}}
    \end{align*}
    Since $\Theta(\mu_\varphi) > 0$, we conclude that $x_{\delta_n}/\delta_n$ remains bounded as $n \to \infty$. 

    Therefore, we can extract a subsequence $\delta_n \cvstrong{k \to \infty}{}0$ such that $\frac{x_{\delta_n}}{2\delta_n} \cvstrong{k \to \infty}{} a$, for some vector $a \in \mathbb{R}^d$. Coming back to the estimate on $\varepsilon_\delta$ from~\eqref{eq.estimation_eps_delta}, dividing both sides by $\delta$ and passing to the limit of these estimates as $\delta_n \to 0$ we obtain that
    \begin{align*}
        \dist_{L^1(\varrhovarphi)}(f,\mathcal{O}_{\BL}) 
        & \le 
        \int_{\mathbb{R}^d} 
        \left|
            \left(
                f - \nabla\varphi(x)\cdot a- \mathbb{E}_{\varrhovarphi}f
            \right)
        \right|\dd \varrhovarphi\\
        & \le 
        \liminf_{n \to \infty}
        (1 + O(\delta_n))
        \int_{\mathbb{R}^d} 
        \left|
            \left(
                f - \nabla\varphi(x)\cdot \frac{x_{\delta_n}}{2\delta_n} - \mathbb{E}_{\varrhovarphi}f
            \right)
        \right|\dd \varrhovarphi
        + o(\delta_n)\\ 
        & \le 
        \limsup_{n \to \infty} 
        C_d
        {\left(
            \delta_{\BL}(f)
            + \frac{o(\delta_n^2)}{\delta_n^2}
        \right)}^{1/2}\\
        & = 
        C_d {\delta_{\BL}(f)}^{1/2},
    \end{align*}
    and the result follows when $\varphi$ is $\alpha$-strongly convex and $f \in L^2(\varrhovarphi)\cap \mathscr{C}^{\infty}_c(\mathbb{R}^d)$. By approximation, it follows also for all $f \in L^2(\varrhovarphi)$. 

    For an arbitrary $\varphi$ convex potential such that $\Theta(\mu_\varphi) > 0$, we must finish the proof with another approximation argument, this time on the convex potential. Contrarily to what sometimes happens, this approximation step turns out to be non-trivial and also uses the geometry of moment measures in a decisive way.
    
    In order to make it clearer, we modify the notation used previously in order for it to emphasize the dependence on the potential $\varphi$. Hence, we let
    \[
        \delta_{\BL}(f;\varphi) 
        \text{ and }
        \dist_{L^1(\varrhovarphi)}(f,\mathcal{O}_{\BL,\varphi})
    \]
    denote the deficit of the Brascamp-Lieb inequality and the distance to optimal functions with respect to the potential $\varphi$, respectively. Our strategy is to add a little quadratic regularization by defining $\varphi_\alpha \eqdef \varphi + \frac{\alpha}{2}|\cdot|^2$. The result will follow if we prove that the deficit if upper semi-continuous 
    \begin{equation}\label{eq.usc_defict}
        \limsup_{\alpha \to 0^+} \delta_{\BL}(f;\varphi_\alpha) \le \delta_{\BL}(f;\varphi),
    \end{equation}
    and that the distance to optimal functions is lower semi-continuous 
    \begin{equation}\label{eq.lsc_distance}
        \dist_{L^1(\varrhovarphi)}(f,\mathcal{O}_{\BL,\varphi}) \le 
        \liminf_{\alpha \to 0^+} \dist_{L^1(\varrhovarphi)}(f,\mathcal{O}_{\BL,\varphi_\alpha}). 
    \end{equation}

    For the first identity~\eqref{eq.usc_defict}, first notice that the Gibbs measures $\varrho_{\varphi_\alpha}$ converge pointwise, even uniformly, to $\varrhovarphi$ as $\alpha \to 0^+$. Then, for any $f \in L^2(\varrhovarphi)$, convergence of the variance follows from Lebesgue's dominated convergence theorem:
    \[
        \Var_{\varrho_{\varphi_\alpha}}(f) \cvstrong{\alpha \to 0^+}{}
        \Var_{\varrho_{\varphi}}(f). 
    \]
    For the positive term in the deficit, notice that $D^2\varphi_\alpha = D^2\varphi + \alpha \id \ge D^2\varphi$, hence $D^2\varphi_\alpha$ is invertible and it follows that ${\left(D^2\varphi_\alpha\right)}^{-1} \le {\left(D^2\varphi\right)}^{-1}$ in the partial ordering of positive semi-definite matrices, even if $D^2\varphi$ becomes singular. As a result, assuming that $\delta_{\BL}(f;\varphi)<+\infty$, we have for all $\alpha > 0$ and $f \in L^2(\varrhovarphi)$ that
    \[
        \int_{\mathbb{R}^d} 
        \inner{{\left(D^2\varphi_\alpha\right)}^{-1}\nabla f, \nabla f}\dd \varrho_{\varphi_\alpha} 
        \le 
        \int_{\mathbb{R}^d} 
        \inner{{\left(D^2\varphi \right)}^{-1}\nabla f, \nabla f}\dd \varrho_{\varphi_\alpha} 
    \]
    and passing to the limit as $\alpha \to 0^+$ yields~\eqref{eq.usc_defict}. 

    Moving on to the second identity~\eqref{eq.lsc_distance}, we shall need the following result: 
    
    \begin{lemma}\label{lem.coercivity-nice-case} Let $f \in L^2(\varrhovarphi)$ be a given function, where $\varphi$ is as above. We have that the function 
    \[
        (a,b) 
        \mapsto 
        \norm{f -(a\cdot \nabla \varphi_\alpha + b)}_{L^1(\varrho_{\varphi_\alpha})}
    \]
    is convex, due to the triangle inequality, and coercive.   
    \end{lemma} 
    
    \begin{proof} { Suppose that we take a sequence $(a_n,b_n) \to \infty$ in the topology of $\R^d \times \R$ such that $\|f - (a_n \cdot \nabla \varphi_\alpha + b_n)\|_{L^1(\varrho_{\varphi_\alpha})}$ converges to $0$. Then one of the following three possibility happens: \\
    
    (i) either $a \to \infty$ while $b$ remains bounded, in which case we have that 
    \[
    \|f - (a_n \cdot \nabla \varphi_{\alpha} +b_n) \|_{L^1(\varrho_{\varphi_\alpha})} \ge \| a_n \cdot \nabla \varphi_{\alpha}\|_{L^1(\varrho_{\varphi_\alpha})} - \|f - b_n\|_{L^1(\varrho_{\varphi_\alpha})} \ge \| a_n \cdot \nabla \varphi_{\alpha}\|_{L^1(\varrho_{\varphi_\alpha})} - C(f),
    \]
    where $C(f)$ denotes a positive, absolute constant independent of $n$. Now, we recall that 
    \[
    \|a_n \cdot \nabla \varphi_\alpha\|_{L^1(\varrho_{\varphi_\alpha})} \ge |a_n| \Theta(\mu_{\varphi_\alpha}) \ge c \cdot |a_n| \to \infty,
    \]
    which contradicts the boundedness of the function along the sequence $(a_n,b_n);$ \\ 
    
    (ii) or $a_n$ remains bounded while $b_n \to \infty$, in which case
    \[
    \|f - (a_n \cdot \nabla \varphi_{\alpha} +b_n) \|_{L^1(\varrho_{\varphi_\alpha})} \ge \|b_n\|_{L^1(\varrho_{\varphi_\alpha})} - \|f\|_{L^1(\varrho_{\varphi_\alpha})} - \|a_n \cdot \nabla \varphi_\alpha\|_{L^1(\varrho_{\varphi_\alpha})} \to \infty,
    \]
    a contradiction again; \\

    (iii) or both are unbounded as $n \to \infty.$ In that case, we claim that we may suppose that the ratio $b_n/|a_n|$ converges. Indeed, if that was not the case, then we would have that $b_n/|a_n| \to \infty$, which would imply that
    \begin{align}\label{eqn:coercive}
    \|f- (a_n \cdot \nabla \varphi_\alpha + b_n)\|_{L^1(\varrho_{\varphi_\alpha})} & \ge |a_n| \left\| \frac{a_n}{|a_n|} \cdot \nabla \varphi_{\alpha} + \frac{b_n}{|a_n|} \right\|_{L^1(\varrho_{\varphi_\alpha})}  - \|f\|_{L^1(\varrho_{\varphi_\alpha})} \cr 
    & \ge |a_n| \cdot \left( \frac{b_n}{|a_n|} -  \Theta(\mu_{\varphi_\alpha},a_n/|a_n|)\right) - \|f\|_{L^1(\varrho_{\varphi_\alpha})} \to \infty,
    \end{align}
    since we know that $\Theta(\mu_{\varphi_{\alpha}},a_n/|a_n|)$ is bounded from above and below independently of $n$. Since $b_n/|a_n|,a_n/|a_n|$ may be taken to be convergent, it follows again by the computation in \eqref{eqn:coercive} that 
    \begin{align}\label{eqn:coercive2}
    \frac{1}{|a_n|}\|f- (a_n \cdot \nabla \varphi_\alpha + b_n)\|_{L^1(\varrho_{\varphi_\alpha})} & \ge  \left\| \frac{a_n}{|a_n|} \cdot \nabla \varphi_{\alpha} + \frac{b_n}{|a_n|} \right\|_{L^1(\varrho_{\varphi_\alpha})}  - \frac{1}{|a_n|} \|f\|_{L^1(\varrho_{\varphi_\alpha})} \cr 
    & \to \| \nu \cdot \nabla \varphi_\alpha + \theta\|_{L^1(\varrho_{\varphi_\alpha})}. 
    \end{align}
    It follows by the hypotheses on boundedness of the function at hand that the right-hand side of the equation above has to vanish, and hence 
    \[
    \nu \cdot \nabla \varphi_\alpha + \theta \equiv 0. 
    \]
    On the other hand, this last fact implies that either the function $\varphi_\alpha$ is affine in the direction $\nu,$ which is impossible due to the convexity of $\varphi$ and the fact that $\int e^{-\varphi_\alpha} < + \infty,$ or we have $\theta \equiv 0,$ which again leads us to a contradiction when taking the fact that $\Theta(\mu_\alpha) > 0$ and \eqref{eqn:coercive2} into account. }
    \end{proof}

So the infimum is attained at some $(a_\alpha,b_\alpha)$. 
In addition,  we will need the next result. 

\begin{lemma}\label{lem.boundedness-parameters} Let $(a_\alpha,b_\alpha)$ denote the pair of points in $\R^n \times \R$ where the infimum of 
\[
  \inf_{a,b} \norm{f -(a\cdot \nabla \varphi_\alpha + b)}_{L^1(\varrho_{\varphi_\alpha})}
\]
is attained. Then the set $\{(a_\alpha,b_{\alpha})\}_{\alpha \in (0,1)}$ is bounded. 
\end{lemma}

\begin{proof} This is due to the fact that, if that were not the case, then by employing a similar argument to the one used to prove coercivity of the function $(a,b) \mapsto \|f - a\cdot \nabla \varphi - b\|_{L^1(\varrho_{\varphi_\alpha})}$, we would have that 
   \begin{align}\label{eqn:coercive3}
    \|f- (a_\alpha \cdot \nabla \varphi_\alpha + b_\alpha)\|_{L^1(\varrho_{\varphi_\alpha})} & \ge |a_\alpha| \left( \left\| \frac{a_\alpha}{|a_\alpha|} \cdot \nabla \varphi + \frac{b_\alpha}{|a_\alpha|} \right\|_{L^1(\varrho_{\varphi_\alpha})} - \alpha \|x\|_{L^1(\varrhovarphi)} \right) - \|f\|_{L^1(\varrho_{\varphi_\alpha})} \cr 
    & \ge \frac{|a_\alpha|}{2} \cdot \left( \left\| \nu \cdot \nabla \varphi + \theta \right\|_{L^1(\varrho_{\varphi_\alpha})} - \alpha\|x\|_{L^1(\varrho_{\varphi_\alpha})}\right) - \|f\|_{L^1(\varrho_{\varphi_\alpha})} \to \infty,
    \end{align}
since we may always suppose without loss of generality that, upon passing to a subsequence, $a_{\alpha}/|a_\alpha|$ and $b_{\alpha}/|a_{\alpha}|$ both converge to $\nu$ and $\theta,$ respectively. This shows that we may assume $\{a_\alpha\}_{\alpha \in (0,1)}$ and $\{b_{\alpha}\}_{\alpha \in (0,1)}$ to be bounded.    \end{proof}

By Lemma \ref{lem.boundedness-parameters}, we can assume w.l.o.g.~that $a_\alpha \cvstrong{\alpha\to 0^+}{}\bar a$ and $b_{\alpha} \cvstrong{\alpha\to 0^+}{}\bar b$. It then follows from Fatou's Lemma that
    \begin{align*}
        \dist_{L^1(\varrhovarphi)}(f,\mathcal{O}_{\BL,\varphi})
        &\le 
        \norm{f -(\bar a\cdot \nabla \varphi + \bar b)}_{L^1(\varrho_{\varphi})}
        \le 
        \liminf_{\alpha \to 0^+} 
        \dist_{L^1(\varrho_{\varphi_{\alpha}})}(f,\mathcal{O}_{\BL,\varphi_\alpha}).
    \end{align*}
    The main inequality in Theorem \ref{thm.brascamp_lieb_quantitative} is now proved. In order to show that such an inequality is, moreover, \emph{sharp}, we take any $f_0 \not\in \text{span}(\{a \cdot \nabla \varphi + b\}_{a \in \R^d, b \in \R}).$ We claim that 
    \[
    \dist_{L^1(\varrhovarphi)}(f_0,\mathcal{O}_{\BL,\varphi}) > 0.
    \]
    Indeed, by coercivity it follows that, given $f_0$ fixed, there are $a_0 \in \R^n, b_0 \in \R$ such that 
    \[
    \dist_{L^1(\varrhovarphi)}(f_0,\mathcal{O}_{\BL,\varphi}) = \|f - (a_0 \cdot \nabla \varphi + b_0) \|_{L^1(\varrho_{\varphi_\alpha})}. 
    \]
    If the distance were $0,$ it would mean that $f_0 \in \text{span}(\{a \cdot \nabla \varphi + b\}_{a \in \R^d, b \in \R}),$ which is not the case. 
    
    We now consider the functions $f_{\varepsilon} = \varepsilon \cdot f_0$. On the one hand, note that $\delta_{\BL}(f_\varepsilon) = \varepsilon^2 \cdot \delta_{\BL}(f_0).$ On the other hand, since $\text{span}(\{a \cdot \nabla \varphi + b\}_{a \in \R^d, b \in \R}) = \mathcal{O}_{\BL,\varphi}$ is a \emph{linear} space, we have that $\dist_{L^1(\varrhovarphi)}(f_\varepsilon,\mathcal{O}_{\BL,\varphi}) = \varepsilon \cdot \dist_{L^1(\varrhovarphi)}(f_0,\mathcal{O}_{\BL,\varphi}).$ Since the right-hand side in the latter formula is a fixed positive constant depending only on $f_0$ times $\varepsilon,$ we get that an inequality such as 
     \[
        \dist_{L^1(\varrhovarphi)}(f, \mathcal{O}_{\BL})
        \le 
        C_d {\delta_{\BL}(f)}^{\alpha}
    \]
    could only possibly hold for $f = f_\varepsilon$, as $\varepsilon \to 0,$ if $\alpha \leq \frac{1}{2}.$ This concludes our proof. 
\end{proof}

\subsection{On the stability constant in the $L^p$-topology for $1 < p < +\infty$}\label{sec.counter_example}

    In this paragraph we show that the $L^1$-topology is sharp for the quantitative stability of the Brascamp-Lieb inequality with respect to the independence of the stability constant from the convex potential. We construct a one-parameter family of potentials ${(\varphi_{\varepsilon,d})}_{\varepsilon > 0}$ and of test functions for the inequality ${(f_{\varepsilon,d})}_{\varepsilon > 0}$ demonstrating this behavior. 

    To emphasize the role of the convex potential and facilitate the computations, in this paragraph we let
    \[
        \delta_{\BL}(f;\varphi) 
        \eqdef  
        \mathcal{E}_{\varrho_\varphi}(f) -
        \Var_{\varrhovarphi}(f), 
        \text{ where }
        \mathcal{E}_{\varrho_\varphi}(f) \eqdef 
        \int_{\mathbb{R}^d}
        \inner{(D^2\varphi)^{-1}\nabla f, \nabla f }\dd \varrho_\varphi
    \]
    denotes the deficit of the Brascamp-Lieb inequality associated with $\varphi$, 
    \[
        \dist_{L^p(\varrhovarphi)}(f;\mathcal{O}_{\BL,\varphi}) \eqdef 
        \inf_{u \in \mathcal{O}_{\BL}(\varphi)} \norm{f - u}_{L^p(\varrhovarphi)}, 
        \text{ with } 
        \mathcal{O}_{\BL}(\varphi) 
        \eqdef 
        \left\{
            a\cdot\nabla\varphi + b: 
            (a,b) \in \mathbb{R}^d\times\mathbb{R}
        \right\},
    \]
    denotes the $L^p$-distance to optimal functions, and we also define the space of convex functions over $\mathbb{R}^d$ as 
    \[
        \Cvx_d \eqdef 
        \left\{
            \varphi : \mathbb{R}^d\to \mathbb{R} : 
            \varphi \text{ is strongly convex}
        \right\}. 
    \]
    Notice that the class of convex functions for which our $L^1$ stability result holds with a universal constant is much larger than $\Cvx_d$. 
    
\begin{theorem}\label{thm.constant_depends_pot_Lp}
    For every $d\ge 1$ and $1 < p < +\infty$ it holds that 
    \begin{equation}\label{eq.inf_stability_constant_Lp}
        \inf_{
            \substack{
                f \in L^2(\varrho_\varphi)\\ 
                \varphi \in \Cvx_d 
            }
        }
        \frac{{\delta_{\BL}(f, \varrho_\varphi)}^{1/2}}
        {\dist_{L^p(\varrho_\varphi)}(f;\mathcal{O}_{\BL}(\varphi))} = 0.
    \end{equation}

    In particular, there is no universal constant $0<C<+\infty$ for which it holds that
    \[
        \dist_{L^p(\varrho_\varphi)}(f;\mathcal{O}_{\BL}(\varphi)) 
        \le 
        C {\delta_{\BL}(f, \varrho_\varphi)}^{1/2} 
        \text{ for all $f \in L^2(\varrho_\varphi)$}
    \]
    with $C$ independent of $\varphi \in \Cvx_d$. 
\end{theorem}

The proof is based on the following construction. 
\begin{lemma}\label{lemma.counter_example}
    There is a family of strongly convex functions ${\left(\varphi_{\varepsilon,d}\right)}_{\varepsilon > 0} \subset \Cvx_d$ and a family of Lipschitz functions ${\left(f_{\varepsilon,d}\right)}_{\varepsilon > 0}$ such that $f_{\varepsilon,d} \in L^2(\varrho_{\varepsilon,d})$ for all $\varepsilon > 0$ satisfying 
    \begin{enumerate}
        \item $\delta_{\BL}(f_{\varepsilon,d}, \varrho_{\varepsilon,d}) = \varepsilon^2 + O(\varepsilon^{3})$
        \item ${\dist_{L^p(\varphi_{\varepsilon,d})}(f_{\varepsilon,d}, \mathcal{O}_{\BL}(\varphi_{\varepsilon,d}))}^p \gtrsim \varepsilon + O(\varepsilon^{2})$ \text{ for all $1 < p < +\infty$}. 
    \end{enumerate}
\end{lemma}
\begin{proof}[\underline{\textbf{Proof for $d=1$:}}]
    Fix a non-negative even $C^\infty$ mollifier $\kappa:\R\to[0,\infty)$ normalized such that $\norm{\kappa}_{L^1(\mathbb{R})} = 1$ and supported in $[-1,1]$. For a given $\varepsilon > 0$ we define the one-dimensional potential as
    \begin{equation}\label{eq.phi}
      \varphi''_{\varepsilon}(x)
      \eqdef 
      \varepsilon
      +\varepsilon^{-2}\kappa\!\left(\frac{x-1}{\varepsilon}\right)
      +\varepsilon^{-2}\kappa\!\left(\frac{x+1}{\varepsilon}\right),
    \end{equation}
    and impose the normalization
    \[
      \varphi_{\varepsilon}(0)=0,
      \qquad
      \varphi_{\varepsilon}'(0)=0.
    \]
    The right-hand side of~\eqref{eq.phi} is smooth, even, and bounded below by
    $\varepsilon$, so $\varphi_{\varepsilon}$ is smooth, even, and strictly convex. For simplicity of notation we also abbreviate 
    \[
        \varrho_\varepsilon \eqdef \varrho_{\varphi_\varepsilon} 
        \text{ and }
        Z_\varepsilon \eqdef \int_{\mathbb{R}} e^{-\varphi_\varepsilon}. 
    \]

    We decompose $\R$ into the central interval $C_{\varepsilon}$, two transition layers $I_{\varepsilon}^\pm$, and the
    exterior $E_{\varepsilon}$:
    \[
      C_{\varepsilon}\eqdef [-1+\varepsilon,1-\varepsilon],
      \quad
      I_{\varepsilon}^{\pm}\eqdef (\pm 1-\varepsilon,\pm 1+\varepsilon),
      \quad 
      T_{\varepsilon}\eqdef I_{\varepsilon}^+\cup I_{\varepsilon}^-,
      \quad
      E_{\varepsilon}\eqdef\R\setminus(C_{\varepsilon}\cup T_{\varepsilon}).
    \]
    With this notation, as $\kappa$ is supported inside $(-1,1)$ we see that $\varphi_\varepsilon''$ is flat inside $C_{\varepsilon}$. On the transition layer, we can control the size of $\varphi_\varepsilon$. 
    
    On the other hand, to construct the test function $f_\varepsilon$ we want to concentrate its derivative on the transition layers. Hence we define the following, even, Lipchitz function, bounded between $0$ and $1$ as 
    \begin{equation}\label{eq:test_function}
    f_\varepsilon(x)\eqdef
    \begin{cases}
        0,& x\in C_{\varepsilon},\\
        \displaystyle
        \frac{
            1
        }
        {
            \int_{I_{\varepsilon}^+}\varphi''_{\varepsilon}
            e^{\varphi_{\varepsilon}}
        }
        \int_{1-\varepsilon}^{|x|}
        \varphi_\varepsilon''(t)e^{\varphi_{\varepsilon}(t)}\dd t ,
        & |x|\in(1-\varepsilon,1+\varepsilon),\\
        1,& x\in E_{\varepsilon}.
    \end{cases}
    \end{equation}
    
    To prove item (1), we gather finer estimates on $\varphi_\varepsilon$ defined through~\eqref{eq.phi}. For $\varepsilon$ small enough, the following properties can be easily checked 
    \begin{enumerate}
    \item for $x \in C_{\varepsilon}$ one has
      \[
        \varphi''_{\varepsilon}\equiv\varepsilon,
        \qquad
        \varphi'_{\varepsilon}(x)=\varepsilon x,
        \qquad
        \varphi_{\varepsilon}(x)=\frac{\varepsilon x^2}{2}\in\left[0,\frac{\varepsilon}{2}\right];
      \]
    \item on $T_{\varepsilon}$ one has $\varphi_{\varepsilon}(x)=O(\varepsilon^{2})$ uniformly, hence
      $e^{\varphi_{\varepsilon}(x)}=1+O(\varepsilon^{2})$ uniformly on $T_{\varepsilon}$;
    \item for $x>1+\varepsilon$,
      \[
        \varphi'_{\varepsilon}(x)=\varepsilon^{-1}+\varepsilon x,
      \]
      and therefore
      \[
        \varphi_{\varepsilon}(x)
        =\varphi_{\varepsilon}(1+\varepsilon)
         +\varepsilon^{-1} (x-1-\varepsilon)
         +\frac{\varepsilon}{2}\bigl(x^2-(1+\varepsilon)^2\bigr);
      \]
    \item by symmetry the corresponding formulas on the left tail follow from the
      right-tail ones by the substitution $x\mapsto -x$.
    \end{enumerate}

    To estimate $\delta_{\BL}(f_\varepsilon, \varphi_\varepsilon)$ the term $\mathcal{E}_{\varrho_\varepsilon}(f_\varepsilon)$. Since $f_\varepsilon'$ vanishes outside the transition layers, we get that 
    \begin{equation}\label{eq.estimates_energy}
        \mathcal{E}_{\varrho_\varepsilon}(f_\varepsilon) 
        = 
        \frac{2}{Z_\varepsilon}\int_{I_{\varepsilon}^+}\frac{(f_\varepsilon')^2}{\varphi_\varepsilon''}
        e^{-\varphi_{\varepsilon}}\dd x
        =
        \frac{2}{Z_\varepsilon A_\varepsilon}, 
        \text{ where }
        A_\varepsilon = 
        \int_{I_{\varepsilon}^+}\varphi_\varepsilon''(x)
        e^{\varphi_{\varepsilon}(x)}\dd x. 
    \end{equation}
    So we need to compute the asymptotics of $Z_\varepsilon$ and of $A_\varepsilon$. Starting with $Z_\varepsilon$, from symmetry we can rewrite it as 
    \[
        \frac{1}{2}Z_\varepsilon 
        = 
        \underbrace{
            \int_{C_{\varepsilon} \cap \mathbb{R}_+} e^{-\varphi_\varepsilon} \dd x
        }_{\eqdef S_1}
        + 
        \underbrace{
            \int_{E_{\varepsilon} \cap \mathbb{R}_+} e^{-\varphi_\varepsilon} \dd x 
        }_{\eqdef S_2}
        + 
        \underbrace{\int_{I_{\varepsilon}^+} e^{-\varphi_\varepsilon} \dd x}_{\eqdef S_3}. 
    \]

    For the first sum $S_1$, since $\varphi_\varepsilon(x) \in [0, \varepsilon/2]$ for $x \in C_{\varepsilon}$, we get that $S_1 = 1 - \varepsilon + O(\varepsilon^2)$. For the second, for $x>1+\varepsilon$ we have
    \[
        \varphi_{\varepsilon}(x)
        =\varphi_{\varepsilon}(1+\varepsilon)+\ell\,(x-1-\varepsilon)
        +\frac{\varepsilon}{2}(x-1-\varepsilon)^2,
    \]
    where $\ell= \varepsilon^{-1}+\varepsilon(1+\varepsilon)=\varepsilon^{-1}+O(\varepsilon^{2})$. Thus
    \[
    \int_{1+\varepsilon}^{\infty}e^{-\varphi_{\varepsilon}(x)}\,dx
    =
    e^{-\varphi_{\varepsilon}(1+\varepsilon)}
    \int_0^\infty e^{-\ell u-\varepsilon u^2/2}\dd u.
    \]
    Since $\varphi_{\varepsilon}(1+\varepsilon)=O(\varepsilon^{2})$, the prefactor equals
    $1+O(\varepsilon^{2})$. Also
    \[
    0\le \frac{1}{\ell}
    -\int_0^\infty e^{-\ell u-\varepsilon u^2/2}\dd u
    \le
    \frac{\varepsilon}{2}\int_0^\infty u^2e^{-\ell u}\dd u
    =
    \frac{\varepsilon}{\ell^3}
    =O(\varepsilon^{4}),
    \]
    so the tail integral is $ S_2 = \varepsilon^{-1}+O(\varepsilon^{4})$. 

    Finally, on the transition layer $I_{\varepsilon}^+$, we have that
    $e^{-\varphi_{\varepsilon}(x)}=1+O(\varepsilon^{2})$ uniformly and hence
    \[
        S_3 = 
        \int_{1-\varepsilon}^{1+\varepsilon}e^{-\varphi_{\varepsilon}(x)}\,dx
        =2\varepsilon+O(\varepsilon^{3}).
    \]
    Summing all these contributions we obtain 
    \[
        Z_\varepsilon 
        = 
        2 + 2\varepsilon + O(\varepsilon^{3}). 
    \]
    To compute $A_\varepsilon$, using once again that $e^{\varphi_\varepsilon} = 1 + O(\varepsilon^{2})$ in $I_{\varepsilon}^+$, recalling that $\kappa$ is supported in $(-1,1)$ and has unitary total mass, we get that 
    \[
        A_\varepsilon = 
        (1 + O(\varepsilon^{2})) \int_{I_{\varepsilon}^+}
        \left(
            \varepsilon + \varepsilon^{-1} \, \kappa\left(\frac{x-1}{\varepsilon}\right)
        \right)\frac{\dd x}{\varepsilon}
        =
        (1 + O(\varepsilon^{2})) \left[
            2\varepsilon^2 + \varepsilon^{-1}
        \right]
        = 
        \varepsilon^{-1}\left(1 + O(\varepsilon^{3})\right). 
    \]
    These computations combined with~\eqref{eq.estimates_energy} imply that 
    \[
        \mathcal{E}_{\varrho_\varepsilon}(f_\varepsilon) = 
        \varepsilon - \varepsilon^2 + O(\varepsilon^{3}). 
    \]

    Now we estimate the variance, notice that as $f_\varepsilon = 0$ over $C_{\varepsilon}$, $f_\varepsilon = 1$ over $E_{\varepsilon}$ and $0 \le f_\varepsilon \le 0$ in the transition layers $T_{\varepsilon}$, we get that 
    \[
        \int_{\mathbb{R}} f_\varepsilon \dd \varrho_\varepsilon 
        = 
        \varrho_\varepsilon(E_{\varepsilon}) + O(\varrho_\varepsilon(T_{\varepsilon}))
        \text{ and }
        \int_{\mathbb{R}} f_\varepsilon^2 \dd \varrho_\varepsilon 
        = 
        \varrho_\varepsilon(E_{\varepsilon}) + O(\varrho_\varepsilon(T_{\varepsilon}))
    \]
    Once again, using the previous estimates of $\varphi_\varepsilon$ over each layer, we get that $\varrho_\varepsilon(E_{\varepsilon}) = \varepsilon - \varepsilon^{2} + O(\varepsilon^{3})$ and $\varrho_\varepsilon(T_{\varepsilon}) = O(\varepsilon^{3})$. As a result, we get that 
    \[
        \Var_{\varrho_\varepsilon}(f_\varepsilon) = 
        \int_{\mathbb{R}} f_\varepsilon^2 \dd \varrho_\varepsilon - 
        {\left(
            \int_{\mathbb{R}} f_\varepsilon \dd \varrho_\varepsilon
        \right)}^2 
        = 
        \varrho_\varepsilon(E_{\varepsilon}) - \varrho_\varepsilon(E_{\varepsilon})^2 + O(\varepsilon^{3})
        = 
        \varepsilon - 2\varepsilon^2 + O(\varepsilon^{3}). 
    \]
    Combining these two contributions, we get the asymptotics for the deficit of the Brascamp-Lieb inequality
    \[
        \delta_{\BL}(f_\varepsilon, \varrho_\varepsilon) = \varepsilon^2 + O(\varepsilon^{3}). 
    \]

    Finally, we compute the asymptotics for the distance and prove item (2). Notice that as $f_\varepsilon$ is even and $\varphi_\varepsilon'$ is odd, for $p = 2$ it is very clear from the Hilbertian structure of $L^2(\varrho_\varepsilon)$ that these two functions are orthogonal and hence the distance of $f_\varepsilon$ to $\mathcal{O}_{\BL}(\varphi_\varepsilon)$ does not see the direction $\varphi_\varepsilon'$. This is less clear in the $p \neq 2$ case, so we first remove the odd optimizer direction. For every $a,b\in\R$ the convexity of $u\mapsto |u|^p$ gives
    \[
    \frac{|f_{\varepsilon}(x)-b-a\varphi'_{\varepsilon}(x)|^p
            +|f_{\varepsilon}(x)-b+a\varphi'_{\varepsilon}(x)|^p}{2}
    \ge |f_{\varepsilon}(x)-b|^p.
    \]
    But since $f_\varepsilon - b$ is even and $\varphi'_\varepsilon$ is odd, we get that $|f_\varepsilon(x) - b - a\varphi_\varepsilon'(x)| = |f_\varepsilon(-x) - b + a\varphi_\varepsilon'(-x)|$, after integration we obtain
    \begin{equation}\label{eq:lp-drop-odd}
    \|f_{\varepsilon}-a\varphi'_{\varepsilon}-b\|_{L^p(\varrho_{\varepsilon})}^p
    \ge
    \|f_{\varepsilon}-b\|_{L^p(\varrho_{\varepsilon})}^p.
    \end{equation}
    Thus the lower bound for the full optimizer space reduces to a lower bound
    against constants.

    At least in the $p = 2$ case, the intuition that the optimal $b$ should be $b = \mathbb{E}_{\varrho_\varepsilon}f_\varepsilon$. Hence, we set 
    \[
        \bar f_\varepsilon \eqdef f_{\varepsilon}-\mathbb{E}_{\varrho_\varepsilon}f_\varepsilon. 
    \]
    Next, we fix $b\in\R$. If $|b+\mathbb{E}_{\varrho_\varepsilon}f_\varepsilon|>1/2$, then on $C_\varepsilon$ we have $|\bar f_{\varepsilon}-b|=|-\mathbb{E}_{\varrho_\varepsilon}f_\varepsilon-b|>\frac12$, so that 
    \[
        \|
            \bar f_{\varepsilon}
            -b
        \|_{L^p(\varrho_{\varepsilon})}^p
        \ge 2^{-p}\varrho_{\varepsilon}(C_\varepsilon).
    \]
    If $|b+\mathbb{E}_{\varrho_\varepsilon}f_\varepsilon|\le 1/2$, then on $E_\varepsilon$, $|\bar f_{\varepsilon}-b|=|1-\mathbb{E}_{\varrho_\varepsilon}f_\varepsilon-b|\ge \frac12$,
    and hence
    \[
    \|\bar f_{\varepsilon}-b\|_{L^p(\varrho_{\varepsilon})}^p\ge 2^{-p}\varrho_\varepsilon(E_\varepsilon).
    \]
    
    But taking $\varepsilon$ small enough, we have $\varrho_{\varepsilon}(C_\varepsilon) \ge 1/2$, and since we always have $\varrho_{\varepsilon}(E_\varepsilon) \le 1$, together with~\eqref{eq:lp-drop-odd} this proves
    \[
        \dist_{L^p(\varrho_{\varepsilon})}(f_{\varepsilon},\mathcal{O}_{\BL}(\varphi_\varepsilon))^p\gtrsim \varrho_\varepsilon(E_\varepsilon)
        = \varepsilon + O(\varepsilon^2), 
    \]
    and item (2) follows for $d=1$. 
\end{proof}

We now lift this construction to $\mathbb{R}^d$ via a tensorization argument. 
\begin{proof}[\underline{\textbf{Proof for $d\geq 2$:}}]
    Write points of $\R^d$ as $(x,y)\in \R\times\R^{d-1}$ and define
    \[
    \varphi_{\varepsilon,d}(x,y)\eqdef \varphi_{\varepsilon}(x)+\frac{|y|^2}{2},
    \qquad
    f_{\varepsilon,d}(x,y)\eqdef f_{\varepsilon}(x).
    \]
    Let $\gamma_{d-1}$ denote the standard Gaussian probability measure on
    $\R^{d-1}$. Then
    \[
    \varrho_{\varepsilon,d}=\varrho_{\varepsilon}\otimes\gamma_{d-1}
    \text{ and }
    D^2\varphi_{\varepsilon,d}(x,y)=\diag(\varphi_{\varepsilon}''(x),I_{d-1}).
    \]
    Hence
    \[
    \nabla f_{\varepsilon,d}(x,y)=(f_{\varepsilon}'(x),0),
    \qquad
    \langle (D^2\varphi_{\varepsilon,d})^{-1}\nabla f_{\varepsilon,d},\nabla f_{\varepsilon,d}\rangle
    =\frac{(f_{\varepsilon}'(x))^2}{\varphi_{\varepsilon}''(x)}.
    \]
    It then follows that
    \begin{equation}\label{eq:prod-energy}
    \mathcal E_{\varphi_{\varepsilon,d}}(f_{\varepsilon,d})
    =
    \mathcal E_{\varphi_{\varepsilon}}(f_{\varepsilon}),
    \ \ 
    \Var_{\varrho_{\varepsilon,d}}(f_{\varepsilon,d})
    =
    \Var_{\varrho_{\varepsilon}}(f_{\varepsilon}),
    \text{  and  }
    \delta_{\BL}(f_{\varepsilon,d}, \varphi_{\varepsilon,d})
    =
    \delta_{\BL}(f_{\varepsilon},\varphi_{\varepsilon}),
    \end{equation}
    which give the expected behavior from the previous estimates. 

    To compute the $L^p$ distance, first notice that the space of optimal functions becomes
    \[
        \mathcal{O}_{\BL}(\varphi_{\varepsilon,d})
        =
        \text{span}\{1,\varphi'_{\varepsilon}(x),y_1,\dots,y_{d-1}\}.
    \]
    We can use the same symmetry argument from the $d=1$ case. Notice that the probability measure $\varrho_{\varphi_{\varepsilon,d}}=\varrho_{\varepsilon}\otimes\gamma_{d-1}$ has the same symmetries, while $\varphi'_{\varepsilon}(x)$ is odd in $x$ and each $y_j$ is odd in its own Gaussian variable. Therefore, using the same argument from the $d=1$ case we obtain for every $a\in\R$, $\alpha\in\R^{d-1}$, and $b\in\R$ that
    \[
        \|f_{\varepsilon,d}-a\varphi'_{\varepsilon}-\alpha\cdot y-b\|_{L^p(\varrho_{\varphi_{\varepsilon,d}})}^p
        \ge
        \|f_{\varepsilon,d}-b\|_{L^p(\varrho_{\varphi_{\varepsilon,d}})}^p 
        = 
        \|f_{\varepsilon}-b\|_{L^p(\varrho_{\varepsilon})}^p,
    \]
    where the last equality is due to the fact that $f_{\varepsilon,d}$ is independent of $y$. The estimates for $d=1$ imply that 
    \[
        \dist_{L^p(\varrho_{\varphi_{\varepsilon,d}})}
            (f_{\varepsilon,d},\mathcal{O}_{\BL,\varphi_{\varepsilon,d}})^p
        \gtrsim 
        \varrho_{\varepsilon,d}(E_\varepsilon\times\mathbb{R}^{d-1})
        = 
        \varrho_{\varepsilon}(E_\varepsilon)
        \gamma_{d-1}(\mathbb{R}^{d-1}) 
        = \varepsilon + O(\varepsilon^2), 
    \]
    and the result follows. 
\end{proof}

Using this construction, we now move on to the 
\begin{proof}[\underline{\textbf{Proof of Theorem~\ref{thm.constant_depends_pot_Lp}:}}]
    Using the family ${(\varphi_{\varepsilon,d}, f_{\varepsilon,d})}_{\varepsilon > 0}$ constructed in Lemma~\ref{lemma.counter_example}, we get that
    \[
        0
        \le 
        \frac{
            {\delta_{\BL}(f_{\varepsilon,d}, \varrho_{\varepsilon,d})}^{1/2}
        }{
            \dist_{L^p(\varrho_{\varphi,d})}
            (f_{\varepsilon,d};\mathcal{O}_{\BL}(\varphi_{\varepsilon,d}))
        }
        \le O({\varepsilon}^{1 - 1/p}). 
    \]
    Letting $\varepsilon \to 0$, the lower bound $0$ is attained so that it must be the infimum, giving the result. 
\end{proof}

\begin{remark}
At $p=1$ the same computation is exactly borderline. Both the deficit and the distances scale as
\[
    \delta_{\BL}(f_{\varepsilon}, \varphi_{\varepsilon})^{1/2},
  \dist_{L^1(\varrho_{\varepsilon})}(f_{\varepsilon},\mathcal{O}_{\BL}(\varphi_{\varepsilon}))
  \asymp \varepsilon,
\]
which of course is consistent with Theorem~\ref{thm.brascamp_lieb_quantitative}.
\end{remark}

\section{Quantitative Stability of Moment Measures}\label{sec.stability_moment_measures}

In this section we use the previous machinery to obtain a sharp quantitative stability result for moment measures, in Theorem \ref{lemma.stab_momentmeasures_backbone_intro}. As usual, we shall state it again below for the reader's convenience:

\begin{theorem}\label{lemma.stab_momentmeasures_backbone}
    There exists a universal constant $C_d$ depending only on the ambient dimension such that the following holds. Given a Radon measure $\mu \in \mathscr{P}_1(\mathbb{R}^d)$ whose barycenter lies that the origin and $\dim \supp \mu = d$, let $\bar \varphi$ be a maximizer of $\mathscr{J}_\mu$. Then for any convex function $\varphi : \mathbb{R}^d \to \mathbb{R}\cup \{+\infty\}$ it holds that
    \begin{equation}\label{eq.backbone_stability_Gibbs}
        \dist_{L^1(\mathbb{R}^d)}\left(
            \varrho_{\varphi^*},
            \mathcal{M}_\mu
        \right)
        \le C_d 
        {\left(
            \mathscr{J}_{\mu}(\bar \varphi)
            - 
            \mathscr{J}_{\mu}(\varphi)
        \right)}^{1/2}.
    \end{equation}
    
     In addition, assuming that $\mu$ has finite second moment, there exists $\lambda \in [1/4,1/2]$ such that, setting $v \eqdef \bar \varphi - \varphi$ and $\varphi_\lambda \eqdef \bar \varphi + \lambda v$, it holds that
    \begin{equation}\label{eq.backbone_stability_potentials}
       \dist_{L^1(\mu_{\lambda})}\left(
            \varphi,
            \mathcal{N}_\mu
        \right)
        \le C_d \left(
            \mathscr{J}_{\mu}(\bar \varphi)
            - 
            \mathscr{J}_{\mu}(\varphi)
        \right)^{1/2},
    \end{equation}
    where $\mu_{\lambda} = \mu_{\varphi_\lambda^*}$ is the moment measure associated with the interpolation $(\varphi_\lambda)^*$. 
\end{theorem}

\begin{proof}
    To prove the first estimate~\eqref{eq.backbone_stability_Gibbs}, we let $\bar \varphi$ be a maximizer of $\mathscr{J}_\mu$, while $\varphi$ is a general convex function. We then employ the quantitative version of the Prékopa-Leindler inequality with the functions 
    \[
      f = e^{-\bar \varphi^*}, \quad g = e^{-\varphi^*} \text{ and } 
      h = e^{-\varphi_{1/2}^*}
      \text{ with }
      \varphi_{1/2} = \frac{\bar \varphi + \varphi}{2}. 
    \]
    Therefore $f,g,h$ satisfy the Prékopa condition with convexity parameter $s=1/2$. In addition, due to concavity of $\mathscr{J}_\mu$ and the optimality of $\bar\varphi$ we have that
    \begin{align*}
      0 &\le \delta \eqdef \mathscr{J}_\mu(\varphi_{1/2}) 
      - 
      \frac{1}{2}\left(
          \mathscr{J}_\mu(\bar\varphi) + \mathscr{J}_\mu(\varphi)  
      \right)\\ 
      &\le 
      \mathscr{J}_\mu(\bar \varphi) 
      - 
      \frac{1}{2}\left(
          \mathscr{J}_\mu(\bar\varphi) + \mathscr{J}_\mu(\varphi)  
      \right)
      = 
      \frac{1}{2}\left(
          \mathscr{J}_\mu(\bar \varphi) - \mathscr{J}_\mu(\varphi)  
      \right) \eqdef \varepsilon.
    \end{align*}
    The quantity $\delta$ can be precisely related to the deficit of the Prékopa-Leindler inequality. Indeed: 
    \begin{align*}
      \delta =& 
      \log\left( \int_{\mathbb{R}^d} e^{-\varphi_{1/2}^*}\dd x \right)
      - 
      \log{\left( 
        \int_{\mathbb{R}^d} e^{-\bar\varphi^*}\dd x \cdot 
        \int_{\mathbb{R}^d} e^{-\varphi^*}\dd x 
        \right)}^{1/2}\\ 
        & - \int_{\mathbb{R}^d}
        (
          \underbrace{
            \varphi_{1/2} - \frac{1}{2}(\bar\varphi + \varphi)
          }_{=0}
        )
        \dd \mu\\ 
        =& \displaystyle
        \log \left(
          \frac{\displaystyle
             \int_{\mathbb{R}^d} e^{-\varphi_{1/2}^*}\dd x
          }{\displaystyle
            {\left( 
              \int_{\mathbb{R}^d} e^{-\bar\varphi^*}\dd x \cdot 
              \int_{\mathbb{R}^d} e^{-\varphi^*}\dd x 
            \right)}^{1/2}
          }  
        \right) \le \varepsilon.
    \end{align*}

Now, notice that since $\varrho_{\bar \varphi^*}$ and all elements in $\mathcal{M}_\mu$ are probability measures, their distance in $L^1(\mathbb{R}^d)$ is at most $2$. As a result, we can assume that $\mathscr{J}_\mu(\bar \varphi) - \mathscr{J}_\mu(\varphi) < 1$ otherwise the inequality is trivial. From the previous estimates we have that 
\[
  \int_{\mathbb{R}^d} e^{-\varphi_{1/2}^*}\dd x 
  < (1+\varepsilon)
  {\left( 
          \int_{\mathbb{R}^d} e^{-\bar\varphi^*}\dd x \cdot 
          \int_{\mathbb{R}^d} e^{-\varphi^*}\dd x 
  \right)}^{1/2}. 
\]
Then, the quantitative version of Prékopa-Leindler inequality in the form of Lemma~\ref{lemma.sharp_stabilityPL} implies that there exists a universal constant $C_d$ depending only on the dimension and $x_0 \in \mathbb{R}^d$ such that 
\[
  \int_{\mathbb{R}^d} 
  \left|
    \varrho_{\bar \varphi^*}(x)
    - 
    \varrho_{\varphi^*}(x+x_0)
  \right|\dd x 
  \le 
  C_{d}
  {\left(
    \mathscr{J}_\mu(\bar \varphi) - \mathscr{J}_\mu(\varphi)
  \right)}^{1/2}.
\]
Estimate~\eqref{eq.backbone_stability_Gibbs} follows.

For the second estimate we can assume that $\mathscr{J}_{\mu}(\bar \varphi)-\mathscr{J}_{\mu}(\varphi) < +\infty$ is finite, otherwise there is nothing to prove. In particular, this implies that 
\[
    \int_{\mathbb{R}^d} \varphi \dd \mu 
    = 
    \int_{\mathbb{R}^d} \varphi(\nabla \bar \varphi^*) \dd 
    \varrho_{\bar \varphi^*} < + \infty. 
\]
Since $\nabla \bar \varphi^*$ takes values in $\dom \bar \varphi$, this means that $\dom \bar \varphi \subset \dom \varphi$. Therefore, rewritting 
\[
    \dist_{L^1(\mu_{\lambda})}\left(
            \varphi,
            \mathcal{N}_\mu
        \right)
    = 
    \inf_{(a,b) \in \mathbb{R}^d\times \mathbb{R}}
    \norm{
        \varphi(\nabla \varphi_\lambda^*) - \bar \varphi(\nabla \varphi_\lambda^*) - (a\cdot \nabla \varphi_\lambda^* + b )
    }_{L^1(\varrho_{\varphi_\lambda^*})},
\]
we see that both sides of this inequality make sense. 

Now, let us assume momentarily that $\varphi$ is strongly convex and recall the interpolation $\varphi_t = \bar \varphi + tv$, so that $\varphi_0 = \bar \varphi$. Define the one-dimensional function 
\[
    t \mapsto J(t) \eqdef \mathscr{J}_\mu(\varphi_t).
\]
Hence a second order Taylor expansion and the characterization of the second variation of $\mathscr{J}_\mu$ from Lemma~\ref{lemma.first_second_variation}, which can be applied under the assumption of strong convexity of $\varphi$, we obtain that 
\begin{align*}
    \mathscr{J}_\mu(\varphi) - \mathscr{J}_\mu(\bar \varphi) 
    &=
    J(1) - J(0) 
    \le
    J'(0) + \int_0^1(1-t)J''(t) \dd t\\ 
    &\le 
    \underbrace{
        \inner{\nabla \mathscr{J}_\mu(\bar \varphi), v}
    }_{ = 0} 
    + 
    \int_0^1(1-t) \frac{\dd^2}{\dd t^2}\mathscr{J}_\mu(\varphi_t) \dd t\\ 
    &= 
    -
    \int_0^1(1-t) \delta_{\BL}\left(v(\nabla\varphi^*_t), \varrho_{\varphi_t^*}\right)\dd t,
\end{align*}
where we have used the fact that if $\bar\varphi \in \argmax \mathscr{J}_\mu$ then the moment measure $\mu_{\varphi^*}$ coincides with $\mu$ so that the first variation vanishes. In the sequel, using the quantitative stability version of the Brascamp-Lieb inequality provided by Theorem~\ref{thm.brascamp_lieb_quantitative}, there exists a dimensional constant $C_d$ such that 
\begin{align*}
    \int_0^1(1-t) 
    \inf_{a \in \mathbb{R}^d, b \in \mathbb{R}} 
    \norm{(\varphi - \bar \varphi) - (a\cdot x + b)}^2_{L^1(\mu_{\varphi^*_t})} \dd t
    \le 
    C_d^2 \left(
    \mathscr{J}_\mu(\bar \varphi) - \mathscr{J}_\mu(\varphi)
    \right).
\end{align*}
In particular, there must exist some $\lambda \in (1/4,1/2)$ such that the integrand above evaluated at the interpolation $\mu_{\varphi_\lambda^*}$ yields
\begin{align*}
    \inf_{a \in \mathbb{R}^d, b \in \mathbb{R}} 
    \norm{(\varphi - \bar \varphi) - (a\cdot x + b)}^2_{L^1(\mu_{\varphi^*_\lambda})}
    \le 
    2C_d \left(
    \mathscr{J}_\mu(\bar \varphi) - \mathscr{J}_\mu(\varphi)
    \right)^{1/2}.
\end{align*}

For a general convex function $\varphi$, we set $\varphi_\alpha \eqdef \varphi + \frac{\alpha}{2}|x|^2$, as in the proof of Theorem~\ref{thm.brascamp_lieb_quantitative}. For all $\alpha > 0$ the above inequality holds giving 
\[
    \inf_{(a,b) \in \mathbb{R}^d\times \mathbb{R}}
    \norm{
        \varphi(\nabla \varphi_{\lambda_\alpha}^*) - \bar \varphi(\nabla \varphi_{\lambda_\alpha}^*) - (a\cdot \nabla \varphi_{\lambda_\alpha}^* + b )
    }_{L^1(\varrho_{\varphi_{\lambda_\alpha}^*})} 
    \le 
    C_d\left(
        \mathscr{J}_\mu(\bar \varphi) - \mathscr{J}_\mu(\varphi_\alpha)
    \right),
\]
for some $\lambda_\alpha \in [1/4,1/2]$. Now notice that adding the $\alpha$ strong convexity to $\varphi$ acts as an infimum convolution of $\varphi_t^*$ since 
\[
    \varphi_{t,\alpha}^* 
    \eqdef 
    {\left(
        \varphi_t + \frac{t\alpha}{2}|\cdot|^2
    \right)}^* 
    = 
    \varphi_t^* \square \frac{1}{2 t\alpha}|\cdot|^2.  
\]
As a result of well-known proporties of infimal convolution and Moreau-Yosida regularization, it follows that
\[
\varphi_{t,\alpha}^* \cvstrong{\alpha \to 0}{} \varphi_t^*, 
\quad 
\nabla \varphi_{t,\alpha}^* 
\cvstrong{\alpha \to 0}{}
\nabla \varphi_t^*,
\]
pointwise. 

Hence on the righ-hand side, we clearly have 
\[
    \log\left(
        \int_{\mathbb{R}^d} e^{-\varphi_{\alpha}^*}\dd x
    \right)
    \cvstrong{\alpha \to 0}{} 
    \log\left(
        \int_{\mathbb{R}^d} e^{-\varphi^*}\dd x
    \right),
\]
and for the integral w.r.t.~$\mu$ we get 
\[
    \int_{\mathbb{R}^d} \varphi_\alpha \dd \mu 
    = 
    \int_{\mathbb{R}^d} \varphi \dd \mu
    + 
    \frac{\alpha}{2}\int_{\mathbb{R}^d}|x|^2 \dd \mu
    \cvstrong{\alpha \to 0}{}
    \int_{\mathbb{R}^d} \varphi \dd \mu,
\]
since $\mu$ has finite second moment. So the right-hand side converges to $\mathscr{J}_\mu(\bar \varphi) - \mathscr{J}_\mu(\varphi)$. 

For the right-hand side, we can apply Lemma~\ref{lem.coercivity-nice-case} to find $(a_{\alpha}, b_{\alpha})$ attaining the infimum. In addition, as in the proof of Theorem~\ref{thm.brascamp_lieb_quantitative} the family ${(a_{\alpha}, b_{\alpha})}_{\alpha}$ is contained in a compact set of $\mathbb{R}^d\times \mathbb{R}$, so up to subsequences we can assume them to converge to some $(\bar a, \bar b)$. Hence using Fatou's Lemma, we get that 
\begin{align*}
    \liminf_{\alpha\to 0} &\norm{
        \varphi(\nabla \varphi_{\lambda_\alpha}^*) - \bar \varphi(\nabla \varphi_{\lambda_\alpha}^*) - (a_\alpha\cdot\nabla \varphi_{\lambda_\alpha}^* + b_\alpha )
    }_{L^1(\varrho_{\varphi_{\lambda_\alpha}^*})} \\
    &\ge 
    \norm{
        \varphi(\nabla \varphi_{\bar \lambda}^*) - \bar \varphi(\nabla \varphi_{\bar \lambda}^*) - (\bar a\cdot \nabla \varphi_{\bar \lambda}^* + \bar b )
    }_{L^1(\varrho_{\varphi_{\bar \lambda}^*})} 
    \ge 
    \dist_{L^1(\mu_{\lambda})}\left(
            \varphi,
            \mathcal{N}_\mu
        \right). 
\end{align*}
This finishes our proof.
\end{proof}

To finish this section, we discuss an example showing that the exponent $1/2$ above is sharp.

\begin{example}\label{ex.sharp_exponentPL}
    For simplicity, we construct an example in 1D, the higher dimensional construction is straight-forward. We consider
    \[
        \mu \propto e^{-\frac{|x|^2}{2}}
    \]
    to be the standard gaussian. This way, the optimal potential is given by 
    \[
        \varphi_\mu = \varphi_\mu^* = \frac{|x|^2}{2}.
    \]
    Define also the perturbation 
    \[
        \varphi_{\varepsilon} \eqdef 
        \frac{1+\varepsilon}{2}|x|^2 
        \text{ so that }
        \varphi_{\varepsilon}^* =
        \frac{1}{(1+\varepsilon)}
        \frac{|x|^2 }{2}. 
    \]
    Hence we have that 
    \[
        Z_{\varphi_\mu} 
        = 
        \sqrt{2\pi} 
        \text{ and }
        Z_{\varphi_\varepsilon} 
        = 
        \sqrt{2\pi(1+\varepsilon)}. 
    \]
    Therefore
    \[
        \mathscr{J}_\mu(\varphi_\varepsilon) = 
        \frac{1}{2}\log(2\pi(1+\varepsilon)) 
        -
        \frac{1+\varepsilon}{2}
    \]
    and 
    \[  
        \sup \mathscr{J}_\mu 
        -
        \mathscr{J}_\mu(\varphi_\varepsilon) = 
        \frac{\varepsilon^2}{4} + O(\varepsilon^3). 
    \]

    On the other hand, we now give a lower bound to the distance of between Gibbs measures up to translations after optimizing over all translations. Let $\rho_0$ and $\varrho_\varepsilon$ denote the centered Gaussian densities with variances $1$ and $(1+\varepsilon)$, respectively. For every $a\in\mathbb{R}$ and every bounded function $\zeta$ with $|\zeta|\infty\leq 1$,
    \[
    \norm{\varrho_0-\varrho_\varepsilon(\cdot-a)}_{L^1(\mathbb{R})}
    \geq
    \left|
    \int_{\mathbb{R}}
    \zeta(x)\bigl(\varrho_0(x)-\varrho_\varepsilon(x-a)\bigr)\dd x
    \right|.
    \]
    Equivalently, using the complex-valued test function $e^{ix}$ and then the reverse triangle inequality for the Fourier transforms,
    \[
    \begin{aligned}
    \norm{\varrho_0-\varrho_\varepsilon(\cdot-a)}_{L^1(\mathbb{R})}
    &\geq
    \left|
    e^{-1/2}-e^{ia}e^{-(1+\varepsilon)/2}
    \right| \
    &\geq
    e^{-1/2}-e^{-(1+\varepsilon)/2} \
    &=
    e^{-1/2}\left(1-e^{-\varepsilon/2}\right).
    \end{aligned}
    \]
    Combining this with the asymptotics in $\varepsilon$ for the optimality gap in the energy, it shows that~\eqref{eq.backbone_stability_Gibbs} is sharp with the exponent $1/2$. 

    The same example also shows the sharpness of~\eqref{eq.backbone_stability_potentials}. Indeed we have that 
    \begin{align*}
        \inf_{a,b} 
        \norm{\varphi_\mu - \varphi_\varepsilon 
        - 
        (a\cdot \bar \varphi_\varepsilon' + b)}_{L^1(\bar \mu_\varepsilon)}
        &= 
        \inf_{a,b} 
        \norm{\varepsilon \frac{|\cdot|^2}{2}
        - 
        (a(1 + t\varepsilon)x + b)}_{L^1(\bar \mu_\varepsilon)} \\ 
        &= 
        \varepsilon
        \inf_{a',b'} 
        \norm{\frac{|\cdot|^2}{2}
        - 
        (a'(1 + t\varepsilon)x + b')}_{L^1(\bar \mu_\varepsilon)}
        \propto \varepsilon,
    \end{align*}
    with the change of variables $(a',b') = (a/\varepsilon, b/\varepsilon)$, and the interpolation $\bar \mu_\varepsilon$ is still a gaussian. This finishes the proof of sharpness of the exponent $1/2$. 
    
\end{example}

\section{Applications} 

Finally, we move on to several different instances where the previous results in this manuscript can be employed in order to deduce interesting stability properties of moment measures. 

\subsection{Stability in a compact domain}\label{subsec.stability_compact} The simplest case of stability is when both $\varrho$ and $\mu$ are restricted to be supported on a convex and compact domain $\Omega$, that is $\varrho, \mu \in \mathscr{P}(\Omega)$. The attentive reader might have hesitated if this is even possible without contradicting the uniqueness result of Cordero-Eurasquin and Klartag for moment measure representations. Here the question is to find $\psi \in \mathscr{C}(\Omega)$ such that 
\begin{equation}\label{eq.moms_compact}
    \mu = {(\nabla\psi)}_\sharp e^{-\psi}\mathscr{L}^d\mres \Omega. 
\end{equation}
The existence question follows with simplified versions of the arguments from~\cite{cordero2015moment,santambrogio2016dealing} by considering the functionals 
\[
    \mathscr{E}_{\mu,\Omega}(\varrho) \eqdef \int_\Omega \log\varrho \dd \varrho + \mathcal{T}(\varrho, \mu),\quad 
    \mathscr{J}_{\mu,\Omega}(\varphi) \eqdef \log \int_\Omega e^{-\varphi^*}\dd x - \int_\Omega \varphi\dd\mu.
\]
Uniqueness then follows, now without modulo translations, due to the strict geodesic convexity of the entropy since $\Omega$ is convex. 

One could embed $\mu$ in $\mathscr{P}(\Omega')$ for any $\Omega'$ containing its support and obtain another moment measure representation. What is happening here? For any representation in a bounded domain the potential $\psi_\Omega$ is necessarily Lipschitz continuous, for instance due to $\nabla \psi_\Omega$ being the Brenier map with the compactly supported target $\mu$, and hence bounded in $\Omega$. Therefore, if one wishes to use $\psi_\Omega$ for a moment measure representation in $\mathbb{R}^d$ we must extend it with $+\infty$ outside of $\Omega$ so that $e^{-\varphi_\Omega^*}$ remains a probability measure. But this means that this extention is not continuous in all of $\partial \Omega$, and hence it is not essentially continuous. Therefore have no contradiction with the results of~\cite{cordero2015moment}, since their result states uniqueness up to translations of a moment measure representation with an essentially continuous potential.

From one hand, this case is much less interesting from a strictly mathematical point of view, since the compactness of the domain makes many arguments easier. However, it is particularly relevant for numerical applications where one often restricts the problem to a bounded domain. Furthermore, in this case we obtain stronger stability results as a direct corollary of Theorem~\ref{lemma.stab_momentmeasures_backbone}.

\begin{theorem}\label{thm.stability_compact_case}
    Let $\Omega$ be a convex and compact subset of $\mathbb{R}^d$. Given two Radon measures $\mu, \nu \in \mathscr{P}(\Omega)$ whose barycenters lie at the origin and $\dim \supp \mu = \dim \supp \nu = d$, let $\varphi_\mu, \varphi_\nu$ be maximizers of $\mathscr{J}_{\mu,\Omega}, \mathscr{J}_{\nu,\Omega}$ respectively. Then, there exists a universal constant $C_{d,\Omega}$ depending only on the dimension and on the diameter of $\Omega$ such that
    \begin{equation}\label{eq.stability_Gibbs_compact}
        \inf_{x_0 \in \mathbb{R}^d} 
        \norm{
            \varrho_{\varphi_\mu^*}(\cdot) - \varrho_{\varphi_\nu^*}(\cdot + x_0) 
        }_{L^1(\mathbb{R}^d)}
        \le C_{d,\Omega} {W_1(\mu,\nu)}^{1/2}
    \end{equation}
    and
    \begin{equation}\label{eq.stability_potentials_compact}
        \inf_{a \in \mathbb{R}^d, \ b \in \mathbb{R}} 
        \norm{
            (\varphi_\mu - \varphi_\nu) - (a\cdot x + b) 
        }_{L^1(\mu_{\lambda})}
        \le C_{d,\Omega} {W_1(\mu,\nu)}^{1/2},
    \end{equation}
    where $\mu_\lambda = \mu_{\psi_\lambda}$ is the moment measure of $\psi_\lambda = ((1-\lambda)\varphi_\mu + \lambda\varphi_\nu)^*$ for some $\lambda \in (0,1)$.
\end{theorem}
\begin{proof}
    Consider $\mu,\nu \in \mathscr{P}(\Omega)$, and consider two convex functions $\varphi_\mu, \varphi_\nu$ such that
    \[
    \varphi_\mu \in \argmax \mathscr{J}_\mu,
    \quad 
    \varphi_\nu \in \argmax \mathscr{J}_\nu. 
    \]
    Since $\varphi_\mu, \varphi_\nu$ are optimal Kantorovitch potentials for the $L^2$-optimal transportation problem in a compact domain, they are both Lipschitz with constant $\diam \Omega$. Then by optimality of $\varphi_\mu$, for the maximization of $\mathscr{J}_\mu$, and concavity of $\mathscr{J}_\mu$ we have that
    \begin{align*}
    0 
    \le \mathscr{J}_\mu(\varphi_\mu) - \mathscr{J}_\mu(\varphi_\nu) &\le \inner{\nabla \mathscr{J}_\mu(\varphi_\nu), \varphi_\mu - \varphi_\nu} = 
    \inner{\nu - \mu, \varphi_\mu - \varphi_\nu}\\
    &\le \Lip(\varphi_\mu - \varphi_\nu) W_1(\mu,\nu), 
    \end{align*}
    which can then be easily combined with Lemma~\ref{lemma.stab_momentmeasures_backbone} to obtain bounds on the distance between the sets $\argmin \mathscr{E}_\mu, \argmin \mathscr{E}_\nu$, and between $\argmax \mathscr{J}_\mu, \argmax \mathscr{J}_\nu$.     
\end{proof}

\begin{remark}
    A similar argument can be employed to estimate the distance between the sets $\argmin \mathscr{E}_{\mu,\alpha}, \argmin \mathscr{E}_{\nu,\alpha}$ in the $L^1(\mathbb{R}^d)$ topology. Indeed, for all $\alpha \ge 0$ the functional $\mathscr{J}_{\mu,\alpha}$ remains concave, with an application of the classical Prékopa-Leindler inequality. Hence taking $\varphi_{\mu,\alpha}, \varphi_{\nu,\alpha}$ such that
    \[
    \varphi_{\mu,\alpha} \in \argmax \mathscr{J}_{\mu,\alpha},
    \quad 
    \varphi_{\nu,\alpha} \in \argmax \mathscr{J}_{\nu,\alpha}, 
    \]
    these are also Lipschitz continuous, and we have the same bound 
    \begin{align*}
    0 
    \le \mathscr{J}_{\mu,\alpha}(\varphi_{\mu,\alpha}) - \mathscr{J}_{\nu,\alpha}(\varphi_{\nu,\alpha}) &\le \inner{\nabla \mathscr{J}_{\mu,\alpha}(\varphi_{\nu,\alpha}), \varphi_\mu - \varphi_\nu} = 
    \inner{\nu - \mu, \varphi_{\mu,\alpha} - \varphi_{\nu,\alpha}}\\
    &\le \diam\Omega W_1(\mu,\nu). 
    \end{align*}
    Which can also be used as in Lemma~\ref{lemma.stab_momentmeasures_backbone} to bound the distance $\argmin \mathscr{E}_{\mu,\alpha}, \argmin \mathscr{E}_{\nu,\alpha}$, but not of $\argmax \mathscr{J}_{\mu,\alpha}, \argmax \mathscr{J}_{\nu,\alpha}$, since for $\alpha > 0$ the argument via quantitative estability of Brascamp-Lieb does not apply. 
\end{remark}

\subsection{Convergence rates w.r.t.~regularization parameter}\label{subsec.conv_rates_reg} We are also able to prove results on the sharp rates of convergence of regularized moment measures introduced in~\cite{delalande2025regularized} to the classical ones from~\cite{cordero2015moment}. 

Recall that for each $\alpha>0$, the regularization present in the energies $\mathscr{E}_{\mu,\alpha}$ and $\mathscr{J}_{\mu,\alpha}$ induces a unique minimizer/maximizer pair $(\varrho_\alpha, \varphi_\alpha)$ associated with the moment regularized measure representation of $\mu$. Therefore, it makes sense to study the convergence of these pairs as $\alpha \to 0$ by estimating the distance of $\varrho_\alpha$ to $\mathcal{M}_\mu$ and of $\varphi_\alpha$ to $\mathcal{N}_\mu$.

\begin{theorem}\label{thm.conv_regularization}
    Fix a Radon measure $\mu \in \mathscr{P}_1(\mathbb{R}^d)$ whose barycenter lies at the origin and such that $\dim\supp\mu = d$. Then there exists a constant $C_{d,\mu}$ depending on the ambient dimension and on $\mu$ such that 
    \begin{equation}
        \dist_{L^1(\mathbb{R}^d)}\left(
            \varrho_\alpha,
            \mathcal{M}_{\mu}
        \right)
        \le C_{d,\mu} \alpha^{1/2}, 
    \end{equation}
    where the Gibbs probability measures 
    $
        \varrho_\alpha \propto e^{-(\varphi_\alpha^* + \frac{\alpha}{2}|x|^2)},
    $
    give the regularized moment measure representation for $\mu$.

    In addition, let ${\left(\varphi_\alpha\right)}_{\alpha > 0}$ be the family of uniquer maximizers of ${\left(\mathscr{J}_{\mu,\alpha}\right)}_{\alpha > 0}$. Then, for each $\alpha$, there exists $\lambda \in (0,1/2)$ such that, setting $v\eqdef \varphi - \varphi_\alpha$ and $\varphi_t \eqdef \varphi + \lambda v$, it holds that
    \begin{equation}
        \dist_{L^1(\mu_{\lambda})}\left(
            \varphi_\alpha,
            \mathcal{N}_{\mu}
        \right)
        \le C_{d,\mu} \alpha^{1/2},
    \end{equation}
    where $\mu_{\lambda} = \mu_{\varphi_\lambda^*}$ is the moment measure associated with the interpolation $(\varphi_\lambda)^*$. 
\end{theorem}

\begin{example}\label{example.unbounded2ndmom}
  Set $S_\lambda(x) = \lambda x$, and consider the scallings 
  \[
      \mu_\lambda = {(S_\lambda)}_\sharp\mu, \text{ and }
      \varrho_\lambda = e^{\psi_\lambda} 
      \text{ with }
      {(\varphi_\lambda)}_\sharp\varrho_\lambda = \mu.
  \]  
  Then one can check that $\varphi_\lambda(\cdot) = \lambda^d \varphi(\lambda\cdot)$, where $\varphi$ is the moment map. As a result we have that 
  \[
    M_2(\mu_\lambda) = \lambda^2 M_2(\mu), \quad 
    M_2(\varrho_\lambda) = \lambda^{-2} M_2(\varrho),
  \]
  so it is perfectly possible that $M_2(\varrho)$ explodes while $M_2(\mu)$ remains bounded.
\end{example}

For the purposes of Theorem~\ref{thm.conv_regularization}, the only important fact is that $M_2(\varrho_\mu)$ is finite, which is always the case for log-concave measures, but in the next section we identify conditions on $\mu$ which ensure that $M_2(\varrho_\mu)$ can be controlled uniformly.

As we will see in the proof below, the constant $C_{d,\mu} = C_d M_2(\varrho)$, where $\varrho = \varrho_\mu$, the Gibbs measure associated with the moment measure representation of $\mu$. Therefore, the constant depends on $\mu$ only via the second moment of $\varrho_\mu$, which is finite since $\varrho_\mu$ is log-concave, but cannot be controlled in general only by bounding the second moment of $\mu$ as the following example shows. 

\begin{proof}[Proof of Theorem~\ref{thm.conv_regularization}]
    The goal is to compare the $\mathscr{J}_\mu$ energies of $\varphi_\alpha$, the unique maximizer of $\mathscr{J}_{\mu,\alpha}$ for $\alpha > 0$, and the maximizer $\varphi$ of $\mathscr{J}_\mu$ such that $\varrho = e^{-\varphi^*}$ is a centered probability measure.
    
    From Theorem~\ref{thm.characterization_moment_measures}, we know that the unique minimizer of $\mathscr{E}_{\mu,\alpha}$ can be obtained from $\varphi_\alpha$ as 
    \[
        \varrho_\alpha \propto e^{-(\varphi_\alpha^* + \frac{\alpha}{2}|x|^2)}, 
    \]
    and similarly for the case $\alpha = 0$, we can map $\varphi$ into a minimizer of $\mathscr{E}_\mu$ defined as the Gibbs probability measure
    \[
        \varrho \propto e^{-\varphi^*}. 
    \]

    To compare the $\mathscr{J}_{\mu}$ energies, we start by recalling the regularized energy $\mathscr{J}_{\mu,\alpha}$ and noticing the following trivial bound 
    \begin{align*}
        \mathscr{J}_{\mu,\alpha}(\varphi_\alpha) 
        &= 
        \log\int_{\mathbb{R}^d} e^{-(\varphi_\alpha^* + \frac{\alpha}{2}|x|^2)}\dd x 
        - 
        \int_{\mathbb{R}^d} \varphi_\alpha \dd \mu \\
        &\le 
        \log\int_{\mathbb{R}^d} e^{-\varphi_\alpha^*}\dd x 
        - 
        \int_{\mathbb{R}^d} \varphi_\alpha \dd \mu
        = 
        \mathscr{J}_{\mu}(\varphi_\alpha). 
    \end{align*}
    Therefore we have that 
    \begin{align*}
        \mathscr{J}_{\mu}(\varphi) - 
        \mathscr{J}_{\mu}(\varphi_\alpha) 
        &\le 
        \mathscr{J}_{\mu}(\varphi) 
        - 
        \mathscr{J}_{\mu,\alpha}(\varphi_\alpha) 
        = \mathscr{E}_{\mu,\alpha}(\varrho_\alpha)
        - 
        \mathscr{E}_{\mu}(\varrho) \\ 
        &\le 
        \mathscr{E}_{\mu,\alpha}(\varrho)
        - 
        \mathscr{E}_{\mu}(\varrho) 
        = \alpha M_2(\varrho) = C_\mu \alpha,
    \end{align*}
    where the first equality comes from the strong duality formula between the minimization of $\mathscr{E}_{\mu,\alpha}$ and the maximization of $\mathscr{J}_{\mu,\alpha}$ proved in Theorem~\ref{thm.characterization_moment_measures}. Since $\varrho$ is a log-concave measure, its second moment is finite and hence a constant $C_\mu$ depending only on $\mu$. 
    
    Combining this estimate with Lemma~\ref{lemma.stab_momentmeasures_backbone}, the result follows. 
\end{proof}

\subsection{Stability over $\mathbb{R}^d$}\label{subsec.stabilityRd}

The stability result in a compact setting proven in Section~\ref{subsec.stability_compact} covers most cases of applications, since for computational purposes one usually needs to truncate the measures to a compact domain. Although the restriction of $\mu$ to a compact domain is not problematic due to many interesting cases already being in this setting, the Gibbs measure $\varrho_\mu$ is in general supported over the whole $\mathbb{R}^d$. Therefore, it is interesting to obtain stability results also for moment measure representations over the whole $\mathbb{R}^d$.

As mentionned above, this was achieved for regularized moment measures in~\cite{delalande2025regularized}, using the strong geodesic convexity of the functional $\mathscr{E}_{\mu,\alpha}$. The convergence results from Section~\ref{subsec.conv_rates_reg} depending on $M_2(\varrho_\mu)$ give a hint on why the stability for regularized moment measures appears to be a more direct issue, as in this case one can explect to control the second moment of the Gibbs measure $\varrho_{\mu,\alpha}$ independently of $\mu$, thanks to the regularization term. Without such a uniform control, we cannot know if the constant multiplying any quantitative stability inequality will remain bounded. Example~\ref{example.unbounded2ndmom} above shows that we cannot in general bound $M_2(\varrho_\mu)$ in terms of $M_2(\mu)$.

In this section, we identify two very natural classes of measures $\mu$, for which $M_2(\varrho_\mu)$ remains uniformly bounded. We can then leverage the stability proof via geodesic convexity from~\cite{delalande2023quantitative} and the convergence rates in Theorem~\ref{thm.conv_regularization} to obtain quantitative stability results for moment measure representations over $\mathbb{R}^d$ in terms of the $W_2$-distance. 

The first approach is to impose a lower bound on the quantity 
\begin{equation}
    \Theta(\mu) \eqdef 
    \inf_{\theta \in \mathbb{S}^{d-1}} 
    \int_{\mathbb{R}^d} |y\cdot\theta| \dd \mu(y),
\end{equation}
that measures the minimal spread of the measure $\mu$ in all directions, beging therefore very natural for the study of moment measures. We shall also require that $\mu$ is supported on a convex and compact set $\Omega$, so that we can control the Lipschitz constant of the potential $\varphi_\mu^*$. Therefore, we consider the class of measures defined as
\begin{equation}\label{eq.class_meas_half_spaces}
    \mathcal{K}_{\vartheta} \eqdef 
      \left\{
        \mu \in \mathscr{P}(\Omega) : 
        \begin{array}{c}
            \mu \text{ is centered, } \\ 
            \displaystyle
            \Theta(\mu) \ge \vartheta
        \end{array} 
      \right\}.
\end{equation}

The fact that $\mu \in \mathscr{P}(\Omega)$ implies that the potential $\varphi_\mu^*$ is Lipschitz continuous with constant $L = \diam(\Omega)$. This is not sufficient to control $M_2(\varrho_\mu)$, but enforcing the lower bound on $\Theta(\mu)$ we have
\begin{equation}\label{eq.lower_bound_all_directions}
    \Theta(\mu)
    = 
    \inf_{\theta \in \mathbb{S}^{d-1}} 
    \int_{\mathbb{R}^d} |\nabla \varphi_\mu^*(x)\cdot\theta| e^{-\varphi_\mu^*(x)} \dd x
    \ge \vartheta. 
\end{equation}
Combining this with the Lipschitz continuity of $\varphi_\mu^*$, we can employ the following growth estimate for convex functions due to Klartag~\cite{klartag2014logarithmically}.
\begin{lemma}{\cite[Lemma 2]{klartag2014logarithmically}}\label{lemma.klartag_growth_convex}
    Suppose that $\psi:\mathbb{R}^d\to\mathbb{R}$ is a $L$-Lipschitz, convex function such that $\varrho \propto e^{-\psi}$ has barycenter at the origin. If there exists $c>0$ such that
    \[
        \inf_{\theta \in \mathbb{S}^{d-1}} 
        \int_{\mathbb{R}^d} |\nabla\psi(x)\cdot\theta| e^{-\psi(x)} \dd x
        \ge c, 
    \]
    then there exists $\varepsilon,\beta_1,\beta_2$ depending only on $d,L,c$ such that
    \[
        \varepsilon|x| - \beta_1 \le \psi(x) \le L|x| + \beta_2, \quad \text{ for all } x \in \mathbb{R}^d.
    \]
    In particular, the second moment of $\varrho$ is bounded by a constant depending only on $d,L,c$.
\end{lemma}

We can now state our quantitative stability result for measures in $\mathcal{K}_\vartheta$.
\begin{theorem}\label{thm.stability_moment_measures_Rd}
    Given $\Omega$ a convex and compact subset of $\mathbb{R}^d$ and $\vartheta>0$, then for every $p>2$ there exists a constant $C$ such thats for all $\mu, \nu \in \mathcal{K}_\vartheta$
    \begin{equation}\label{eq.stability_moment_measures_Rd}
        \dist_{W_2}
        \left(
            \mathcal{M}_\mu,
            \mathcal{M}_\nu
        \right)
        \le 
        C {W_2(\mu, \nu)}^{\frac{p-2}{6p-4}}, 
    \end{equation}
    where $C$ depends only on the $d$, $\diam\Omega$, $\vartheta$, and grows linearly in $p$. 
\end{theorem}
\begin{proof}[Proof of Theorem~\ref{thm.stability_moment_measures_Rd}]
    The proof consists of combining the convergence rates of regularized moment measures from Theorem~\ref{thm.conv_regularization} with strong geodesic convexity argument from~\cite{delalande2023quantitative} and the uniform bounds on the second moment that we obtain for $\varrho_\mu$ whenever $\mu \in \mathcal{K}_\vartheta$ through Lemma~\ref{lemma.klartag_growth_convex}.

    First, notice that for all $\mu \in \mathcal{K}_\vartheta$ and $\alpha \ge 0$, letting $\varrho_{\mu,\alpha}$ denote the Gibbs measure associated with the (regularized if $\alpha > 0$) moment measure representation of $\mu$, we have that $M_2(\varrho_{\mu,\alpha}) \le C_{\Omega,\vartheta,d}$, where the constant depends only on $\Omega,c$ and $d$. This is a direct consequence of Lemma~\ref{lemma.klartag_growth_convex}, since the potentials $\varphi_{\mu,\alpha}^*$ are all $\diam(\Omega)$-Lipschitz continuous.

    For the reader's convenience, we detail the part of~\cite{delalande2025regularized}'s strong geodesic convexity argument that is relevant to us. Given $\mu,\nu \in \mathcal{K}_\vartheta$ and some $\alpha > 0$, let ${\left(\varrho_{\alpha,t}\right)}_{t \in [0,1]}$ be the constant speed geodesic with endpoints 
    \[
        \varrho_{\alpha,0} = \varrho_{\mu,\alpha}, \quad 
        \varrho_{\alpha,1} = \varrho_{\nu,\alpha}.
    \]
    Since both $\mathscr{E}_{\mu,\alpha}$ and $\mathscr{E}_{\nu,\alpha}$ are $\alpha$-strongly geodesically convex functionals over $\mathscr{P}_2(\mathbb{R}^d)$, we have that
    \begin{align*}
        \mathscr{E}_{\mu,\alpha}(\varrho_{\mu,\alpha}) 
        \le 
        \mathscr{E}_{\mu,\alpha}(\varrho_{\alpha,1/2}) 
        \le  
        \frac{1}{2}\mathscr{E}_{\mu,\alpha}(\varrho_{\mu,\alpha}) + 
        \frac{1}{2}\mathscr{E}_{\mu,\alpha}(\varrho_{\nu,\alpha}) 
        - 
        \frac{\alpha}{8}W_2^2(\varrho_{\mu,\alpha}, \varrho_{\nu,\alpha}),
    \end{align*}
    and similarly for $\mathscr{E}_{\nu,\alpha}$. Rearranging we get 
    \begin{align*}
        \frac{\alpha}{4}W_2^2(\varrho_{\mu,\alpha}, \varrho_{\nu,\alpha}) 
        &\le 
        \mathscr{E}_{\mu,\alpha}(\varrho_{\nu,\alpha}) 
        -
        \mathscr{E}_{\mu,\alpha}(\varrho_{\mu,\alpha}),\\
        \frac{\alpha}{4}W_2^2(\varrho_{\mu,\alpha}, \varrho_{\nu,\alpha}) 
        &\le 
        \mathscr{E}_{\nu,\alpha}(\varrho_{\mu,\alpha}) 
        -
        \mathscr{E}_{\nu,\alpha}(\varrho_{\nu,\alpha}). 
    \end{align*}
    
    Adding these inequalities, and recalling that $\mathcal{T}(\mu_0,\mu_1) = -W_2^2(\mu_0,\mu_1) + M_2(\mu_0) + M_2(\mu_1)$ and $W_2^2(\mu_0,\mu_1) \le  2M_2(\mu_0) + 2M_2(\mu_1)$, we obtain 
    \begin{align*}
        \frac{\alpha}{4}W_2^2(\varrho_{\mu,\alpha}, \varrho_{\nu,\alpha}) 
        &\le 
        \mathcal{T}(\varrho_{\nu,\alpha}, \mu) 
        - \mathcal{T}(\varrho_{\nu,\alpha}, \nu)
        +  \mathcal{T}(\varrho_{\mu,\alpha}, \nu) 
        - \mathcal{T}(\varrho_{\mu,\alpha}, \mu)
        \\
        =&
        (W_2(\varrho_{\nu,\alpha}, \mu) + W_2(\varrho_{\nu,\alpha}, \nu)) 
        (W_2(\varrho_{\nu,\alpha}, \mu) - W_2(\varrho_{\nu,\alpha}, \nu))\\
        &+ 
        (W_2(\varrho_{\mu,\alpha}, \nu) + W_2(\varrho_{\mu,\alpha}, \mu)) 
        (W_2(\varrho_{\mu,\alpha}, \nu) - W_2(\varrho_{\mu,\alpha}, \mu))\\
        \le& 
        (
            W_2(\varrho_{\nu,\alpha}, \mu) + W_2(\varrho_{\nu,\alpha}, \nu)
            + 
            W_2(\varrho_{\mu,\alpha}, \nu) + W_2(\varrho_{\mu,\alpha}, \mu)
        ) 
        W_2(\mu, \nu)\\ 
        \le& 
        C_{\Omega,\vartheta,d} W_2(\mu, \nu),
    \end{align*}
    where we have used that for any $\mu,\nu \in \mathcal{K}_\vartheta$, all second moments above are uniformly bounded by a constant depending only on $\Omega,\vartheta$ and $d$. In particular, this gives an explicit behaviour on $\alpha$ for the bound proven in~\cite[Theorem~3.1]{delalande2025regularized} for the particular case of $\mu,\nu \in \mathcal{K}_\vartheta$, i.e. 
    \begin{equation}\label{eq.regularized_stability_explicit_alpha}
        W_2(\varrho_{\mu,\alpha}, \varrho_{\nu,\alpha})
        \le 
        {\alpha}^{-1/2} C_{\Omega,\vartheta,d}{W_2(\mu, \nu)}^{1/2}. 
    \end{equation}

    To finish the proof we need to exploit the explicit rates of convergence of regularized moment measures from Theorem~\ref{thm.conv_regularization}. The issue is that these convergence rates concern the $L^1$ distance between $\varrho_{\mu,\alpha}$ and $\mathcal{M}_\mu$. But for absolutely continuous measures with bounded $p$-moments, the $L^1$ norm can be used to control Wasserstein distances. There exists a constant $C_{p,q}$ such that 
    \begin{equation}\label{eq.bound_WassersteinL1}
        W_q(\varrho_0,\varrho_1) 
        \le 
        C_{p,q} {(M_p(\varrho_0) + M_p(\varrho_1))}^{1/p} 
        \norm{\varrho_0 - \varrho_1}_{L^1}^{1/q - 1/p},
    \end{equation}
    see Lemma~\ref{lemma.bound_wasserstein_momentsL1} in Section~\ref{sec.tools}. 
    
    As a result, combining~\eqref{eq.bound_WassersteinL1}, with the stability of regularized moment measures from~\eqref{eq.regularized_stability_explicit_alpha}, and the explicit rates of convegence from~\ref{thm.conv_regularization}, we let $\varrho_\mu = \varrho_\mu(\cdot + x_\alpha) \in \mathcal{M}_\mu$ be a Gibbs measure associated with the moment measure representation of $\mu$ such that 
    \[
        \norm{\varrho_{\mu,\alpha} - \varrho_{\mu}}_{L^1(\mathbb{R}^d)} \le 2\dist_{L^1}(\varrho_{\mu,\alpha}, \mathcal{M}_\mu), \le C_{\Omega,d,\vartheta} \alpha^{1/2},
    \]
    and similarly for $\varrho_\nu$. Then, for every $\alpha > 0$, we have
    \begin{align*}
        \dist_{W_2}(\mathcal{M}_\mu, \mathcal{M}_\nu) 
        \le 
        W_2(\varrho_{\mu}, \varrho_{\nu})
        &\le 
        W_2(\varrho_{\mu}, \varrho_{\mu,\alpha})
        + 
        W_2(\varrho_{\mu,\alpha}, \varrho_{\nu,\alpha})
        + 
        W_2(\varrho_{\nu,\alpha}, \varrho_{\nu})\\ 
        &\le 
        C\left(
            \alpha^{1/4 - 1/2p} + \alpha^{-1/2}{W_2(\mu, \nu)}^{1/2}
        \right),
    \end{align*}
    where the constant $C$ above depends on $\Omega, \vartheta, d, p$, more specifically it depends on $M_p(\varrho_\mu)^{1/p}$, which behaves linearly in $p$, so that the constant explodes as $p\to \infty$. 

    But for every finite $p>2$, we can optimize the above estimate in $\alpha$ to obtain the best possible bound, which is achieved for
    \[
        \alpha = {W_2(\mu, \nu)}^{\frac{2p}{3p-2}},
    \]
    which gives 
    \[
        \dist_{W_2}(\mathcal{M}_\mu, \mathcal{M}_\nu) 
        \le 
        C {W_2(\mu, \nu)}^{\frac{p-2}{6p-4}},
    \]
    for all $p > 2$, with a constant $C$ depending on $\Omega, \vartheta, d, p$.
\end{proof}

\begin{proposition}\label{prop.constant_depends_vartheta}
    Inequality~\eqref{eq.stability_moment_measures_Rd} never holds with a constant independent of $\vartheta$. 
\end{proposition}
\begin{proof}
    Suppose that there was a constant $C$, independent of $\vartheta$, such that~\eqref{eq.stability_moment_measures_Rd} holds for all moment measures $\mu,\nu$. Then consider the one-dimensional measure over the segment $L = [0,1]e_1 \subset \mathbb{R}^d$
    \[
        \mu = \mathscr{H}^1\mres L, \quad 
    \]
    and a sequence of normalized mollifications ${\left(\mu_\varepsilon\right)}_{\varepsilon>0}$, so that $W_2(\mu, \mu_\varepsilon) \to 0$.  

    If a inequality such as~\eqref{eq.stability_moment_measures_Rd} held with a constant independent of $\vartheta$, then we would have that $\varrho_\varepsilon = e^{-\psi_\varepsilon}$ would be Cauchy in $W_2$ as $\varepsilon \to 0$, and hence would converge to some limit measure $\varrho$. Since all $\psi_\varepsilon$ as $L$-Lipschitz, we can extract a locally uniformly converging subsequence, so that $\psi_\varepsilon \to \psi$ locally uniformly, and hence $\varrho = e^{-\psi}$. 

    In addition, $\gamma_\varepsilon \eqdef {(\id, \nabla \psi_\varepsilon)}_\sharp\varrho_\varepsilon$ is the optimal transport plan between $\varrho_\varepsilon$ and $\mu_\varepsilon$, and since $W_2(\mu_\varepsilon, \mu) \to 0$ we have that $\gamma_\varepsilon$ converges weakly to some limit plan $\gamma$. By stability of optimal transport plans, $\gamma$ is an optimal transport plan between $\varrho$ and $\mu$ and must be given by ${(\id, \nabla \psi)}_{\sharp}\varrho$. But then $\psi$ is $L$-Lipschitz and a moment measure representation for $\mu$, which is impossible since $\mu$ is supported on a hyperplane.
\end{proof}

\begin{remark}\label{remark.exponent_constant_stabilityRd}
    The exponent $\frac{p-2}{6p-4}$ in Theorem~\ref{thm.stability_moment_measures_Rd} can be made arbitrarily close to $1/6$ by taking $p$ large enough, at the cost of getting a worse constant $C$, which behaves linearly at $p$. It is somewhat natural to expect a worse exponent that $1/2$ obtained in~\cite{delalande2025regularized} for regularized moment measures, since the regularization term improves the second moment of solutions. However, it is not clear whether the limiting exponent $1/6$ is optimal or not.
\end{remark}

The second class of measures for which we can derive quantitative stability results is obtained by bounding the curvature they induce. Given $\Lambda > 0$ we define 
\begin{equation}\label{eq.class_Lambda}
    \mathcal{K}_\Lambda 
    \eqdef 
    \left\{
        \mu \in \Pac(\mathbb{R}^d) 
        : 
        \mu \propto e^{-V}, \text{ where } 
        D^2V \le \Lambda \id
    \right\}.
\end{equation}
In~\eqref{eq.class_Lambda}, the condition $D^2V \le \Lambda\id$ is understood only on $\dom V$ and we do not make any convexity assumption. In this case that $\mu$ has bounded curvature, we can expect that the Gibbs measure $\varrho_\mu$ has a stronger geometry than the log-concavity of the orginal measure. Indeed we shall show that the optimal potential $\varphi_\mu^*$ that yields the moment measure representation of $\mu$ is $\Lambda^{-1}$-strongly convex. 

This improved strong-convexity result will be obtained by bootstraping the suplementary regularity enjoyed by the regularized moment measure representation in conjunction with \textit{Caffarelli's contraction thereom}, a global Lipschitz regularity result for the optimal transport between log-concave measures. See for instance the original works of Caffarelli~\cite{caffarelli1992regularity,caffarelli2000monotonicity}, or~\cite{gozlan2025global} for a modern approach via entropic regularization and the references therein. For the reader's convenience, we recall that it is stated as follows: 
\begin{theorem}[\cite{caffarelli1992regularity,caffarelli2000monotonicity}Caffarelli's contraction theorem]
  Let $\mu,\nu \in \Pac(\mathbb{R}^d)$ be such that $\dom \mu = \mathbb{R}^d$, $\dom \nu$ is convex, and satisfying 
  \begin{equation}
    \mu \propto e^{-V}, \ \nu \propto e^{-W}, \text{ where } 
    D^2V \le \Lambda \id \text{ and }
    D^2W \ge \lambda \id,
  \end{equation}
  where $\Lambda, \lambda > 0$. Then the $L^2$-optimal transport map $T_\sharp\mu = \nu$ is Lipschitz continuous with constant $\sqrt{\Lambda/\lambda}$.
\end{theorem}

With this result we prove the following regularity result for moment measures associated with elements of $\mathcal{K}_\Lambda$. 
\begin{theorem}\label{thm.regularity_KLambda}
    For any $\mu \in \mathcal{K}_\Lambda$, let $\varphi_\mu^*$ be a potential inducing the moment measure representation $\mu = {\left(\nabla \varphi_\mu^*\right)}_\sharp e^{-\varphi_\mu^*}$. Then $\varphi_\mu^*$ is $\Lambda^{-1}$-convex.
\end{theorem}
\begin{proof}
Fix some $\alpha > 0$ and consider the regularized moment measure representation of $\mu$ as 
\[
  \mu 
  = 
  {(\nabla \varphi^*_\alpha)}_\sharp e^{-(\varphi^*_\alpha + \frac{\alpha}{2}|x|^2)}
  = 
  {(\nabla \varphi^*_\alpha)}_\sharp \varrho_\alpha,
\]
so that $\varrho_\alpha \propto e^{-V_\alpha}$ where $D^2V_\alpha \ge \alpha \id$. In addition, from Brenier's Theorem $\nabla \varphi_\alpha$ is the $L^2$ optimal transportation map from $\mu$ to $\varrho_\alpha$. 

From Caffarelli's contraction theorem we have that
\[
  \Lip(\nabla \varphi_\alpha) \le \sqrt{\frac{\Lambda}{\alpha}}, 
\]
which in turns implies further regularity for the Legendre transform $\varphi_\alpha^*$, being therefore $\sqrt{\alpha/\Lambda}$-strongly convex. This implies a better bound for strong convexity of $V_\alpha$ and can be exploited once again in Caffarelli's contraction theorem, implying that 
\[
   \Lip(\nabla \varphi_\alpha) 
   \le \sqrt{\frac{\Lambda}{\sqrt{\alpha/\Lambda}}} 
   = \frac{\Lambda^{3/4}}{\alpha^{1/4}}. 
\]

This argument can be bootstrapped indefinably, and one can check by induction that at the $k$-th time it yields
\[
  \Lip(\nabla \varphi_\alpha) 
  \le \frac{\Lambda^{S_k}}{\alpha^{1/2^k}} \text{ and }
  \varphi_\alpha^* \text{ is $\frac{\alpha^{1/2^k}}{\Lambda^{S_k}}$--strongly convex,}
\] 
with $S_k = \sum_{i = 1}^k 2^{-i}$. Passing to the limit as $k \to \infty$, we obtain that for all $\alpha > 0$
\[
  \Lip(\nabla \varphi_\alpha) 
  \le \Lambda^{1} \text{ and }
  \varphi_\alpha^* \text{ is $\Lambda^{-1}$--strongly convex.}
\] 

In particular, this implies that $\nabla \varphi_\alpha$ converges locally uniformly to $\nabla \varphi_\mu$ a maximizer of $\mathscr{J}_\mu$, so that $\varphi$ is also $\Lambda^{1/3}$-Lipschitz, so that the optimal potential $\varphi^*$ is $\Lambda^{-1}$--strongly convex, as we wanted to prove.
\end{proof}

Besides the relevance to obtaining another quantitative stability result for moment measures in the class $\mathcal{K}_\Lambda$, see Theorem~\ref{thm.stability_moment_measures_Rd_Lambda} below, this result is particularly interesting for applications in sampling. Whenever $\mu \propto e^{-V}$ is log-concave there are efficient sampling algorithms based on the Lagevin dynamics associated with the potential $V$~\cite{chewi2023optimization,chewi2024statistical,dwivedi2019log}. In the case that $V$ is strongly convex, its associate Gibbs measure statisfies a log-Sobolev inequality as we've learned from the Bakry-Émery technique~\cite{bakry2013analysis} and as a result implies exponential covergence rates for the associated Langevin dynamics~\cite{markowich2000trend}. This can then be exploited in sampling algorithms, see~\cite{chehabprovable} for a recent application or the recent monographs~\cite{chewi2023optimization,chewi2024statistical}. With only the upper bound $D^2V \le \Lambda \id$, the rate of convergence of the Langenvin dynamics, and hence of sampling algorithms, can be only polynomial if any convergence guarantee is even available. However, if the moment measure representation for $\mu$ has been computed, we know from Theorem~\ref{thm.regularity_KLambda} that its potential $\varphi_\mu^*$ is strongly convex, so we can use a Langevin-based sampling algorithm with exponential convergence rates to obtain a sample of the (strongly) log-concave measure $\varrho_\mu$ and map this sample into a sample of $\mu$ directly.

In Lemma~\ref{lemma.klartag_growth_convex} we have exploited the lower bound on $\Theta(\mu)$ aligned with the Lipschitz continuity of the moment map induced by $\mu$ to bound the moments $M_p(\varrho_\mu)$. This can even more easily be achieved for $\mu\in \mathcal{K}_\Lambda$ by exploiting the strong convexity result from Theorem~\ref{thm.regularity_KLambda} and the following bound on the $p$-moments. 
\begin{lemma}\label{lemma.bounds_p_moments_Lambda}
    Let $\psi$ be a centered, $\lambda$-strongly convex function. Then for every $p>2$ the $p$-moment of its associated Gibbs measure can be bounded as 
    \[
        {M_p(e^{-\psi})}^{1/p} 
        = 
        {\left(
            \int_{\mathbb{R}^d} |x|^p e^{-\psi(x)}\dd x
        \right)}^{1/p} 
        \le C p,
    \]
    where $C$ is a constant depending only on $\lambda$ and the dimension $d$. 
\end{lemma}
\begin{proof}
    Since the $\lambda$-strongly convex function $\psi$ is centered, it is a known fact in convex analysis that:
    \begin{equation}\label{eq.estimate_value_at0}
        \psi(0) \le d + \inf \psi = d + \psi(x_*)
    \end{equation}
    where $x^*$ denotes its minimizer. This is proven for instance in~\cite{fradelizi1997sections}. As a result of the $\lambda$-strong convexity of $\psi$, and the fact that $\nabla \psi(x_*) = 0$ that
    \begin{equation}\label{eq.control_grad_strongly_convex}
        \inner{\nabla \psi(x), x - x_*} 
        \ge 
        \lambda|x - x_*|^2.
    \end{equation}
    On the other hand, the divergence theorem gives that 
    \begin{equation}\label{eq.divergence_log_concave}
        \int_{\mathbb{R}^d} 
        \divv\left(
            |x-x_*|^{p-2}(x-x_*)
        \right)e^{-\psi}\dd x 
        = 
        \int_{\mathbb{R}^d} 
            |x-x_*|^{p-2}(x-x_*)\cdot \nabla \psi(x) e^{-\psi}\dd x.
    \end{equation}

    Now define the following quantity
    \[
        V(p) \eqdef 
        \int_{\mathbb{R}^d} |x - x_*|^p e^{-\psi(x)}\dd x. 
    \]
    Combining identities~\eqref{eq.control_grad_strongly_convex},~\eqref{eq.control_grad_strongly_convex}, and developing the divergence, we obtain that 
    \[
        V(p) \le \frac{(d + p - 2)}{\lambda}V(p-2). 
    \]

    Noticing that $V(0) = 1$, by induction we obtain that for all $k \in \mathbb{N}$
    \[
        V(2k) \le \lambda^{-k}\prod_{i = 0}^{k-1}(d+2i)
        = 
        \lambda^{-k}2^k \frac{\Gamma\left(\frac{d}{2}+k\right)}{\Gamma\left(\frac{d}{2}\right)},
    \]
    where $\Gamma$ corresponds to the gamma function.

    Using H\"older's inequality, in the form $\norm{f}_p^p \le \norm{f}_{p_0}^{(1-t)p_0}\norm{f}_{p_1}^{t p_1}$, for a convex combination $p = (1-t)p_0 + tp_1$, we can interpolate the above inequality for any $p>2$, by choosing $p_0 = 2k$ and $p_1= 2(k+1)$ for a convenient $k$ such that $p \in (2k, 2(k+1))$. Due to Stirling's formula, this yields an estimate such that 
    \[  
        {\left(
            \int_{\mathbb{R}^d} |x - x_*|^p e^{-\psi(x)}\dd x
        \right)}^{1/p}
        \le C p,
    \]
    for a constant $C$ depending on $d$ and $\lambda$. 
    
    To conclude, we need only to bound $|x_*|$, this can be done with~\eqref{eq.estimate_value_at0} and the strongl convexity. Indeed, using once again that $\nabla \psi(x_*) = 0$, we have
    \[
        \lambda|x - x_*|^2 \le \psi(0) - \psi(x_*) \le d. 
    \]
    Hence, combing all these estimates we get that 
    \[
        {M_p(e^{-\psi})}^{1/p} \le 
        {\left(
            \int_{\mathbb{R}^d} |x - x_*|^p e^{-\psi(x)}\dd x
        \right)}^{1/p} 
        + 
        |x_*|
        \le C p + \sqrt{\frac{d}{\lambda}},
    \]
    and the result follows. 
\end{proof}

Combining this and the reasoning of Theorem~\ref{thm.stability_moment_measures_Rd} we can obtain a stability result for elements of $\mathcal{K}_\Lambda$ with a completely analogous proof. We state such a result below, but we shall omit its proof. 

\begin{theorem}\label{thm.stability_moment_measures_Rd_Lambda}
    Given $\Lambda> 0$, then for each $p>2$ there exists a constant $C$ such that, for all $\mu, \nu \in \mathcal{K}_\Lambda$ it holds thats
    \begin{equation}\label{eq.stability_moment_measures_Rd_Lambda}
        \dist_{W_2}
        \left(
            \mathcal{M}_\mu,
            \mathcal{M}_\nu
        \right)
        \le 
        C {W_2(\mu, \nu)}^{\frac{p-2}{6p-4}}, 
    \end{equation}
    where $C$ depends only on the $d$, $\Lambda$, and grows linearly in $p$. 
\end{theorem}

\bibliographystyle{alpha}
\bibliography{main.bib}

\end{document}